\documentclass[a4paper, 10pt]{article}

\usepackage{palatino} 
\usepackage{a4}
\usepackage[latin1]{inputenc}
\usepackage[T1]{fontenc}
\usepackage{graphicx}
\usepackage{amsmath, amssymb, amsthm}
\usepackage{dsfont}
\usepackage{verbatim}
\usepackage[all]{xy}
\usepackage{enumerate}
\usepackage[hmargin=2.5cm, hcentering, top = 2.5cm, bottom = 3cm]{geometry}

 \numberwithin{equation}{section}

\newtheorem{theorem}{Theorem}[section]
\newtheorem{lemma}[theorem]{Lemma}
\newtheorem{proposition}[theorem]{Proposition}
\newtheorem{corollary}[theorem]{Corollary}

\newtheorem{definition}[theorem]{Definition}
\newtheorem{propdef}[theorem]{Proposition-Definition}
\newtheorem{lemdef}[theorem]{Lemma-Definition}

\newenvironment{example}[1][Example.]{\begin{trivlist}
\item[\hskip \labelsep {\bfseries #1}]}{\end{trivlist}}
\newenvironment{remark}[1][Remark.]{\begin{trivlist}
\item[\hskip \labelsep {\bfseries #1}]}{\end{trivlist}}

\newcommand{\piiso}{$\hat{\pi}_{\ast}$-isomorphism}

\newcommand{\pihat}{\hat{\pi}}
\newcommand{\sigt}{\tilde{\sigma}}
\newcommand{\lamt}{\tilde{\lambda}}
\newcommand{\Mit}{{$\mathcal{M}$}}
\newcommand{\Mi}{\mathcal{M}}
\newcommand{\Pbb}{\mathbb{P}}
\newcommand{\Abb}{\mathbb{A}}
\newcommand{\Nbb}{\mathbb{N}}
\newcommand{\Zbb}{\mathbb{Z}}
\newcommand{\Sbb}{\mathbb{S}}
\newcommand{\Rbb}{\mathbb{R}}
\newcommand{\Gbb}{\mathbb{G}}
\newcommand{\Hbb}{\mathbb{H}}

\newcommand{\Cst}{\mathcal{C}}

\newcommand{\Fst}{\mathcal{F}}
\newcommand{\Dst}{\mathcal{D}}

\newcommand{\Bst}{\mathcal{B}}
\newcommand{\Ist}{\mathcal{I}}

\newcommand{\Mst}{\mathcal{M}}

\newcommand{\colim}{\operatorname{colim}}
\newcommand{\we}{\wedge}
\newcommand{\dprime}{{\prime\prime}}
\newcommand{\sym}{\operatorname{sym}}
\newcommand{\adjunc}{\xymatrix{\ar@<0.4ex>[r] & \ar@<0.4ex>[l]}}

\begin{document}

\title{Semistable Symmetric Spectra \\ in $\Abb^1$-Homotopy Theory}

\author{Stephan Hähne and Jens Hornbostel}
\date{\today}

\maketitle

\begin{abstract}
We study semistable symmetric spectra
based on quite
general monoidal model categories, including
motivic examples. In particular, we establish
a generalization of Schwede's list of equivalent characterizations 
of semistability in the case of motivic symmetric spectra.
We also show that the motivic Eilenberg-MacLane spectrum and
the algebraic cobordism spectrum are semistable.
Finally, we show that semistability is preserved
under localization if some reasonable conditions
-- which often hold in practice --  are satisfied.
\end{abstract}


\section{Introduction}

A map between CW-spectra (or Bousfield-Friedlander-spectra)
is a stable weak equivalence if and only if it induces an isomorphism 
on stable homotopy groups. This is not true if we replace
spectra by symmetric spectra in general. However, there is 
a large class of symmetric spectra for which the stable homotopy
groups (sometimes called the ``naive stable homotopy groups''
as they ignore the action of the symmetric groups)
do coincide with the stable weak equivalences.
This leads to the notion of {\em semistable} symmetric spectra,
and these have been studied notably by Schwede
\cite{S1}, \cite{S4}, \cite{S6}. 
There are many equivalent ways to recognize them,
and there are indeed many examples of symmetric spectra which are
semistable (e.g. suspension spectra, Eilenberg MacLane spectra,
$K$-theory and various cobordism spectra). Any symmetric spectrum
is weakly equivalent to a semistable one, and semistable
spectra are very suitable both under theoretical and
computational aspects. 

\medskip

The goal of this article is to study semistability for symmetric spectra
based on other model categories than simplicial sets
or topological spaces. Our main interest here are symmetric
spectra based on motivic spaces as studied in \cite{H1}, \cite{J}, 
which model the motivic stable homotopy category \cite{V2}.
However, we state most results in greater generality so that
they may be applied to other settings as well.

The results of this article may be divided in three families.
First, we establish a long list of equivalent characterizations
of semistability. 
Second, we discuss examples of semistable motivic symmetric ring spectra.
Third, we show that semistable ring spectra are particularly well-behaved
under localization. Most of our results are generalizations
of known results for symmetric spectra bases on simplicial sets, 
but at least some proofs considerably differ.

One of our motivations to study semistability for motivic symmetric
ring spectra was our expectation that a motivic version of a Theorem 
of Snaith \cite{GS}, \cite{SO} should lead to a motivic symmetric 
commutative ring spectrum representing algebraic $K$-theory, which then
would fit in the framework of \cite{Hor2}.
Indeed, while the first author was writing \cite{H},
R\"ondigs, Spitzweck and {\O}stv{\ae}r were able to deduce
this result carrying out a small part of the general
theory established here, see the remark after Proposition 
\ref{precedingremark}.

\medskip

We now briefly recall the notion of semistability. For any symmetric 
spectrum $X$, the actions of $\Sigma_n$ on $X_n$ induce
an action of the {\em injection monoid }$\Mst$ (that is the monoid
of injective self-maps on $\Nbb$) on $\pi_*X$. We say that
$X$ is {\em semistable} if this action is trivial.
In general, the $\Mst$-action encodes additional information
of the symmetric spectrum. See \cite[Example 3.4]{S4}
for an example of symmetric spectra with isomorphic
stable homotopy groups but having different $\Mst$-action.

\medskip

The following Theorem of Schwede provides a list of equivalent ways of 
describing semistable symmetric spectra based on simplicial sets.
This is essentially \cite[Theorem I.4.44]{S1}, 
see also \cite[Theorem 4.1]{S4} and \cite{S6}.
\begin{theorem}
\label{th-intro}
\label{th-problem}
\label{thorg}
    For any symmetric spectra $X$ in simplicial sets, the follwing conditions
$(i) - (v)$ are equivalent. If $X$ is levelwise fibrant,
then these are also equivalent to conditions
$(vi) - (viii)$.
    \begin{enumerate}[(i)]
        \item There is a {\piiso} from $X$ to an $\Omega$-spectrum,
that is an isomorphism of naive stable homotopy groups.
        \item The tautological map $c: \hat{\pi}_k X \longrightarrow \pi_k X$ 
from naive to ``true'' homotopy groups is an isomorphism for all $k \in \Zbb$.
        \item The action of $\Mst$ is trivial on all homotopy groups of $X$.
        \item The cycle operator $d$ acts trivially
 on all homotopy groups of $X$.
        \item The morphism $\lambda_X: S^1 \we X \longrightarrow sh X$ is a \piiso.
        \item The morphism $\tilde{\lambda}_X: X \longrightarrow \Omega (sh X)$ is a \piiso.
        \item The morphism $\lambda_X^{\infty}: X \longrightarrow R^{\infty} X$ is a \piiso.
        \item The symmetric spectrum $R^{\infty}X$ is an $\Omega$-spectrum.
    \end{enumerate}
    
\end{theorem}

In order to generalize this Theorem to other model categories
$\Dst$, it seems natural to generalize 
the $\Mst$-action to appropriate stable homotopy groups
in $\Dst$.
However, in our first partial generalization Theorem \ref{thmin}
homotopy groups do not appear. They only do appear later in the
full generalization, namely in Theorem \ref{motsemi-theoremorg-th}.
To state and prove the latter, we need to axiomatize the properties of
the \emph{sign} $(-1)_{S^1}$ on $S^1$, see Definition \ref{vorzeichen-def}.
That is, we require that our circle object $T$ has an
automorphism $(-1)_T$ in $Ho(\Dst)$ satisfying the conditions 
of that Definition. For our applications, it is thus crucial that the 
pointed motivic space $T=\Pbb^1$ has 
a sign (see Proposition \ref{p1hassign}).
We are then able to prove the full generalization 
of Schwede's theorem.
The precise statement of this Main Theorem 
\ref{motsemi-theoremorg-th}
looks rather 
technical at first glance and can be appreciated only after having read
section 2, so we don't reproduce it here.

In section 3, we show that motivic Eilenberg-MacLane spectra
and the motivic cobordism spectrum of Voevodsky \cite{V2}
are semistable. The key here is that the $\Sigma_n$-actions
extend to $GL_n$-actions.

Section 4 generalizes  \cite[Corollary I.4.69]{S1}
about the localization of semistable
symmetric ring spectra with respect to central elements.
The following is a special case of our Theorem
\ref{ringspec-local-semist-th}:

\begin{theorem}
\label{ringspec-intro}
Let $R$ be a level fibrant semistable motivic symmetric ring spectrum and 
$x: T^l \rightarrow R_m$ a central map.
Then we can define a motivic symmetric ring spectrum $R[1/x]$ which is 
semistable,
and the ring homomorphism $\pi_{*,*}^{mot}(R) \xrightarrow{j_*} 
\pi_{*,*}^{mot}(R[1/x])$ is a localization with respect to $x$.
\end{theorem}

This article is based on the diploma thesis of the first
author \cite{H} written under the direction of the second author.
We thank Stefan Schwede for providing us with updates \cite{S6}
of his book project \cite{S1} on symmetric spectra.
As the structure and in particular the numbering
are still subject to change, we only provide precise references
to the version \cite{S1}. We provide details 
rather than refering to \cite{S6} when relying
on arguments not contained in the version \cite{S1}
or in \cite{S4}.

We assume that the reader is familiar with model categories
in general \cite{Hi}, \cite{H3}. For symmetric spectra,
we refer to \cite{H2}, \cite{H1} and \cite{S1}, \cite{S6}.
References for motivic spaces (that is simplicial
presheaves on $Sm/S$ for a noetherian base scheme $S$
of finite Krull dimension) and motivic symmetric
spectra include \cite{MV}, \cite{J} and \cite{DLORV}.
It will be useful for the reader to have a copy of
\cite{H1} and \cite{S1} at hand.

\section{Semistability}

In this section, we will generalize Theorem \ref{th-intro} in two ways.
The first generalization (Theorem \ref{thmin}) applies to symmetric spectra
based on a very general monoidal model category,
but covers only part of the list of equivalent properties
of Theorem \ref{th-intro}.
The second generalization (Theorem \ref{motsemi-theoremorg-th})
applies to a slightly more restricted
class of examples (in particular the motivic ones we
are mainly interested in) and provides the
``full'' analog of Theorem \ref{th-intro}.
We will always assume that $\Dst$ is a 
monoidal model category, and that $T$ is a cofibrant object
of $\Dst$.
If moreover $\Dst$ is cellular and left proper, 
then by \cite{H1}
(see also \cite{J}),
we have both a level and a stable projective model structure on $Sp(\Dst,T)$, 
and similarly on $Sp^\Sigma(\Dst,T)$. 
We refer to \cite[Definition 4.1]{H1} for the definition of 
``almost finitely generated''.

As usual, for any spectrum $X$ we define $sX$ by $(sX)_n=X_{n+1}$,
$\Omega=Hom(T,-)$, $\Theta:=\Omega \circ s$ and 
$\Theta^{\infty}:=colim \Theta^k$. We write $\tilde{\sigma}^X_n$
for the adjoints of the structure maps  $\sigma^X_n$ of $X$,
and $J$ for a fibrant replacement functor in $Sp(\Dst,T)$.
By definition, an $\Omega$-spectrum is level-wise fibrant.
 
For some almost finitely generalized model categories stable
weak equivalences may be characterized as follows \cite[Section 4]{H1}:

\begin{theorem}
\label{stabgenmod}
\label{basics-fgmc-stabeq-th}
Assume that $\Dst$ is almost finitely generated, and that sequential colimits 
commute with finite products and with $\Omega$.
Then for any  $A \in Sp(\Dst,T)$, the map $A \rightarrow \Theta^\infty JA$ 
is a stable equivalence into an $\Omega$-spectrum. Moreover,
for an $f$ in $Sp(\Dst,T)$ the following are equivalent:
\begin{itemize}
    \item $f$ is a stable equivalence.
    \item For any levelwise fibrant replacement $f^\prime$ of $f$ the map 
$\Theta^\infty f^\prime$ is a level equivalence.
    \item There is a levelwise fibrant replacement $f^\prime$ of $f$ such that
the map 
$\Theta^\infty f^\prime$ is a level equivalence.
\end{itemize}
\end{theorem}
\begin{proof}
This is a special case of \cite[Theorem 4.12]{H1} with $U = \Omega$.
\end{proof}

\subsection{The first generalization}

\label{first-gen}

We refer to \cite{H2} and \cite{S1} for standard definitions
and properties of symmetric spectra. We consider a closed symmetric monoidal
model category $(\Dst,\we,S^0)$ with internal Hom-objects $Hom$. 
As above, let $T$ be a cofibrant object in $\Dst$ and $\Omega = Hom(T,-)$. 
We will consider the category of symmetric $T$-spectra
$Sp^\Sigma(\Dst,T)$ with the projective stable model structure of \cite{H1}.
As usual, we define an endofunctor $sh$ on  $Sp^\Sigma(\Dst,T)$
by $shX_n=X_{1+n}$, where (following Schwede) the notation
$1+n$ emphasizes which action of $\Sigma_n$ on $X_{n+1}$ we consider.
We further set $R:=\Omega \circ sh$ and 
$R^{\infty}:=colim R^k$. Recall also that there
is a natural map $\lambda_X:X \we T \to shX$, which has an adjoint
$\tilde{\lambda}_X:X \to RX=\Omega \circ sh X$.

\begin{lemdef}
\label{omega-n-comp}
\begin{enumerate}
\item
Let $X$ be any object of $\Dst$.
We inductively define $ev^n_X: \Omega^n X \we T^n \rightarrow X$ by
$ev^1_X = ev$ and $ev^n_X = ev \cdot (ev^{n-1}_{\Omega X} \we T)$.
Then the adjoint $\delta_{n,X}: \Omega^n X \rightarrow Hom(T^n, X)$ of 
$ev^n_X$ is a natural isomorphism.\\
Using this, we define for any $\tau \in \Sigma_n$ a natural transformation $\Omega^{\tau}: \Omega^n \rightarrow \Omega^n$:\\
$\centerline{\xymatrix{
\Omega^n \ar[r]^(0.4){\cong}_(0.3){\delta_{n,X}} \ar[d]^{\Omega^{\tau}} & Hom(T^n,-) \ar[d]^{Hom(\tau^{-1},-)} \\
\Omega^n \ar[r]^(0.4){\cong}_(0.3){\delta_{n,X}} & Hom(T^n,-)
}}$
\item
If $(\tau_1, \tau_2) \in \Sigma_n \times \Sigma_m$ ($n, m\in \Nbb_0$), then
$\Omega^{\tau_1 + \tau_2}_X = \Omega^{\tau_1}_{\Omega^m X} \cdot \Omega^n \Omega^{\tau_2}_X$
\end{enumerate}
\end{lemdef}
\begin{proof}
\begin{enumerate}
\item
Obvious.
\item
Setting $f := \delta_{n+m,X} \cdot \delta_{n,\Omega^m X}^{-1} \cdot Hom(T^n,\delta_{m,X}^{-1}): Hom(T^n,Hom(T^m,X)) \rightarrow Hom(T^{n+m},X)$ we may 
identify $\Dst(A, f)$ using the following commutative diagramm:\\
$\centerline{\xymatrix{
\Dst(A, \Omega^{n+m} X) \ar@/_5pc/[dd]^(0.7){\Dst(A, \delta_{n,\Omega^m X})} \ar[r]_{\cong} \ar@/^.8pc/[rr]^{\Dst(A, \delta_{n+m,X})} \ar[d]^{\cong} & \Dst(A \we T^{n+m},X) \ar@{=}[d] \ar[r]_{\cong} & \Dst(A, Hom(T^{n+m},X)) \\
\Dst(A \we T^n, \Omega^m X) \ar[d]^{\cong} \ar[r]^{\cong} \ar@/_1pc/[rr]_{\Dst(A \we T^n,\delta_{m,X})} & \Dst((A \we T^n) \we T^m,X) \ar[r]^{\cong} & \Dst(A \we T^n,Hom(T^m,X)) \\
\Dst(A, Hom(T^n, \Omega^m X)) \ar[rr]_{\Dst(A,Hom(T^n, \delta_{m,X}))} && \Dst(A, Hom(T^n,Hom(T^m,X))) \ar[u]^{\cong} \ar@/_5pc/[uu]^(0.8){\Dst(A,f)}
}}$\\
Hence $f$ is compatible with $\tau_1^{-1}: T^n \rightarrow T^n$ and $\tau_2^{-1}: T^m \rightarrow T^m$. By naturality $\delta_{n+m,X} \cdot \delta_{n,\Omega^m X}^{-1}$ is then compatible with $\tau_1^{-1}$, and similarly (because
$f = \delta_{n+m,X} \cdot  \Omega^n \delta_{m,X}^{-1} \cdot \delta_{n,Hom(T^m,X)}^{-1}$) the map $\delta_{n+m,X} \cdot  \Omega^n \delta_{m,X}^{-1}$ 
is compatibe with $\tau_2^{-1}$. The first compatibility imples 
$\Omega^{\tau_1 + m}_X = \Omega^{\tau_1}_{\Omega^m X}$ and the second
$\Omega^{\tau_1 + m}_X = \Omega^n \Omega^{\tau_2}_X$, whence the claim.
\end{enumerate}
\end{proof}

\begin{lemma}
\label{lem-sigt-omega}
Let $X$ be a symmetric $T$-spectrum and $\chi_{l,m} \in \Sigma_{l+m}$
permuting the blocks of the first $l$ and the last $m$ elements.
Then for the structure maps of $\Omega^l X$, we have the equality
$\sigt^{\Omega^l X}_n = \Omega^{\chi_{l,1}}_{X_{n+1}} \cdot \Omega^l \sigt^X_n$.

For $R^\infty X$, we have $\sigt^{R^\infty X}_n = incl \cdot colim \sigt^{R^k X}_n$, with $incl$ being the map $colim(\Omega(R^k X)_{n+1}) \rightarrow \Omega (R^\infty X)_{n+1}$.
\end{lemma}
\begin{proof}
For $l = 1$, we have $\sigt^{\Omega X}_n = \Omega^{\chi_{1,1}}_{X_{n+1}} \cdot \Omega \sigt^X_n$, as by definition we have $ev^1_{X_{n+1}} \cdot (\sigma^{\Omega X}_n \we T) = \sigma_n^X \cdot (ev^1_{X_n} \we T) \cdot (\Omega X_n \we t_{T,T})$  and thus
$ev \cdot [(\delta_{2,X_{n+1}} \cdot \sigt^{\Omega X}_n)\we T^2] = ev^1_{X_{n+1}} \cdot (ev^1_{\Omega X_{n+1}} \we T) \cdot (\sigt^{\Omega X}_n \we T^2) = ev^1_{X_{n+1}} \cdot (\sigma^{\Omega X}_n \we T) = \sigma_n^X \cdot (ev^1_{X_n} \we T) \cdot (\Omega X_n \we t_{T,T}) = ev^1_{X_{n+1}} \cdot (\sigt_n^X \we T) \cdot (ev^1_{X_n} \we T) \cdot (\Omega X_n \we t_{T,T}) = ev^1_{X_{n+1}} \cdot (ev^1_{\Omega X_{n+1}} \we T) \cdot (\Omega \sigt_n^X \we t_{T,T}) = ev \cdot (\delta_{2,X_{n+1}} \we T^2) \cdot (\Omega \sigt_n^X \we t_{T,T}) = ev \cdot (Hom(T^2, X_{n+1}) \we t_{T,T}) \cdot [(\delta_{2,X_{n+1}} \Omega \sigt_n^X) \we T^2] = ev \cdot (Hom(t_{T,T}, X_{n+1}) \we T^2) \cdot [(\delta_{2,X_{n+1}} \Omega \sigt_n^X) \we T^2] = ev \cdot [(\delta_{2,X_{n+1}} \Omega^{\chi_{1,1}}_{X_{n+1}} \Omega \sigt_n^X)\we T^2]$.\\

Induction over $l$ then yields $\sigt_n^{\Omega^{l-1}\Omega X} = \Omega^{\chi_{l-1,1}}_{\Omega X_{n+1}} \cdot \Omega^{l-1} \sigt^{\Omega X}_n = \Omega^{\chi_{l-1,1}}_{\Omega X_{n+1}} \cdot \Omega^{l-1} (\Omega^{\chi_{1,1}}_{X_{n+1}} \cdot \Omega \sigt^X_n) = \Omega^{\chi_{l,1}}_{X_{n+1}} \cdot \Omega^l \sigt^X_n$, by Lemma \ref{omega-n-comp} and $\chi_{l,1} = (\chi_{l-1,1} +1) \cdot ((l-1) + \chi_{1,1})$.\\

The second claim follows as the adjoints of the maps already coincide on
$(R^l X)_n \we T$, where they are $\sigma_n^{R^l X} = ev \cdot(\sigt_n^{R^l X} 
\we T)$.
\end{proof}

In sections \ref{sect-defmact} and \ref{zweivier} below
(compare also \cite[Example I.4.17]{S1}), we will study
in detail the action of the injection monoid 
$\Mst$ on $\underline{X}(\omega) \cong (\Theta^{\infty}X)_0$.
In this section, we only need to know how the action
of the cycle operator $d$ relates
to the map $\lamt$ (generalizing a result of \cite{S6}).
\begin{lemma}
\label{general-lemmata-1}
\label{dlamt}
For any symmetric $T$-spectrum $X$, the following triangle commutes:\\
$\centerline{\xymatrix{
(\Theta^{\infty} X)_0 \ar[r]^{d} \ar[dr]_{(\Theta^{\infty} \lamt_X)_0} & (\Theta^{\infty} X)_0 \ar[d]^{\cong}\\
& (\Theta^{\infty} \Omega sh X)_0 \\
}}$\\
\end{lemma}
\begin{proof}
The isomorphism on the right hand side is induced by
$\Omega^{1+l} X_{1+l} \xrightarrow{\Omega^{\chi_{1,l}}_{X_{1+l}}} \Omega^{l+1} X_{1+l}$. In the diagram \\
$\centerline{\xymatrix{
\Omega^{1+l} X_{1+l} \ar[d]^{\Omega^{\chi_{1,l}}}\ar[r]^(.4){\Omega^{1+l} \sigt} & \Omega^{1+l+1} X_{1+l+1}  \ar[d]^{\Omega^{\chi_{1,l}}\Omega} \ar[dr]^*!/u2pt/{\labelstyle \Omega^{\chi_{1,l+1}}} & \\
\Omega^{l+1} X_{1+l} \ar[r]^(.4){\Omega^{l+1} \sigt} & \Omega^{l+2} X_{1+l+1} \ar[r]^{\Omega^l \Omega^{\chi_{1,1}}} & \Omega^{l+1+1} X_{1+l+1} \\
}}$\\
the lower composition equals $\Omega^l \sigt^{\Omega X}_{1+l}$ by Lemma 
\ref{lem-sigt-omega}. As $\Omega^{\chi_{1,l}} \Omega = \Omega^{\chi_{1,l}+1}$ and $\Omega^l \Omega^{\chi_{1,1}} = \Omega^{l + \chi_{1,1}}$ (Lemma \ref{omega-n-comp}) and $\chi_{1,l+1} = (l+\chi_{1,1}) \cdot (\chi_{1,l} + 1)$ 
everything commutes, hence the above maps are compatible with the structure 
maps. Finally, the diagram of the Lemma is induced by the 
following commutative diagram:\\
$\centerline{\xymatrix{
\Omega^l X_l \ar[r]^{\Omega^l \sigt} \ar[dr]^{\Omega^l \sigt} & \Omega^{l+1} X_{l+1} \ar[r]^{\Omega^{\chi_{l,1}} \chi_{l,1}} & \Omega^{1+l} X_{1+l} \ar[d]^{\Omega^{\chi_{1,l}}}\\
& \Omega^{l+1} X_{l+1} \ar[r]^{\Omega^{l+1} \chi_{l,1}} & \Omega^{l+1} X_{1+l} \\
}}$
\end{proof}

\begin{lemma}
\label{general-lemmata-4}
For any symmetric $T$-spectrum X, there is an isomorphism 
$\sym_{X,n}: (\Theta^\infty X)_n \cong (R^\infty X)_n$.
\end{lemma}
\begin{proof}
The isomorphism is induced by a sequence of compatible isomorphisms \\
        $\centerline{\xymatrix{
        X_n \ar[d]^{1 = \alpha_{0,n}} \ar[r]^{\sigt} & \Omega X_{n+1} \ar[d]^{\Omega \alpha_{1,n}} \ar[r]^{\Omega \sigt} & \Omega^2 X_{n+2} \ar[d]^{\Omega^2 \alpha_{2,n}} \ar[r]^{\Omega^2 \sigt} & \ldots \ar[r] & \Omega^l X_{n+l} \ar[d]^{\Omega^l \alpha_{l,n}} \ar[r]^{\Omega^l \sigt}  & \\
        X_n \ar[r]^{\lamt} & \Omega X_{1+n} \ar[r]^{\Omega\lamt} & \Omega^2 X_{2+n} \ar[r]^{\Omega^2\lamt} & \ldots \ar[r] & \Omega^l X_{l+n} \ar[r]^{\Omega^l\lamt} &
        }}$\\
        
        where $\alpha_{l,n}$ is a permutation which is inductively defined
by the following commutative diagram:\\
        $\centerline{\xymatrix{
        \Omega^l X_{n+l} \ar[d]^{\Omega^l \alpha_{l,n}} \ar[r]^{\Omega^l \sigt} & \Omega^{l+1} X_{n+l+1} \ar[dr]^*!/u3pt/{\labelstyle \Omega^{l+1} \alpha_{l+1,n}} \ar[d]^(0.75){\Omega^{l+1}(\alpha_{l,n}+1)}\\
        \Omega^l X_{l+n} \ar[r]_*!/d3pt/{\labelstyle \Omega^l \sigt} & \Omega^{l+1} X_{l+n+1} \ar[r]_*!/d3pt/{\labelstyle \Omega^{l+1} \chi_{l+n,1}} & \Omega^{l+1} X_{1+l+n}
        }}$\\
        
         Here we use the $\Sigma_{n+l}$-equivariance of $\sigt$ and set
$\alpha_{l+1,n} = \chi_{l+n,1} \cdot (\alpha_{l,n}+1)$. Then by induction, 
it follows that $\alpha_{l,n} = \chi_{n,l} \cdot (n + \beta_l)$,
        where $\beta_l \in \Sigma_l$ is the reflection  $\beta_l(i) = l + 1 - i$:
        \begin{itemize}
        \item
            $\chi_{n,0} \cdot (n + \beta_0) = id$
        \item
            $\chi_{l+n,1} \cdot (\alpha_{l,n}+1) = \chi_{l+n,1} \cdot ([\chi_{n,l} \cdot (n + \beta_l)]+1) = \chi_{n,l+1} \cdot (n + \beta_{l+1})$
        \end{itemize}
\end{proof}

\begin{corollary}
\label{general-cor-1}
\label{Rinftomeg}
Assume that $\Omega$ commutes with sequential colimits.
Then for any $X$ in $Sp^\Sigma(\Dst,T)$, the following diagram commutes:\\
$\centerline{\xymatrix{
(\Theta^\infty sh^n X)_0 \ar[d]^{d} \ar@{}[r]|-{=} & (\Theta^\infty X)_n \ar[rr]^{\sym_{X,n}} & & (R^\infty X)_n \ar[d]^{\sigt^{R^\infty X}_n}\\
(\Theta^\infty sh^n X)_0 \ar@{}[r]|-{=} & (\Theta^\infty X)_n \ar@{}[r]|-{\cong} & \Omega(\Theta^\infty X)_{n+1} \ar[r]^*!/u3pt/{\labelstyle \Omega \sym_{X,n+1}} & \Omega(R^\infty X)_{n+1}\\
}}$
\end{corollary}

\begin{proof}
Using Lemma \ref{lem-sigt-omega} it suffices to show
that the following diagram commutes: \\
$\centerline{\xymatrix{
\Omega^l X_{n+l} \ar[r]^(0.45){\Omega^l \sigt} \ar[d]^{\Omega^l \alpha_{l,n}} & \Omega^{l+1} X_{n+l+1} \ar[r]^(0.37){\Omega^{\chi_{l,1}} (n + \chi_{l,1})} \ar[d]^{\Omega^l (\alpha_{l,n}+1)} & \Omega^{1+l} X_{n+1+l} \ar[d]^{\Omega^{1+l} \alpha_{l,n+1}}\\
\Omega^l X_{l+n} \ar[r]^(0.45){\Omega^l \sigt} & \Omega^{l+1} X_{l+n+1} \ar[r]^{\Omega^{\chi_{l,1}}} & \Omega^{1+l} X_{l+n+1}\\
}}$\\
This is the case as we have $\alpha_{l,n+1} \cdot (n+\chi_{l,1}) = \chi_{n+1,l}\cdot(n+1+\beta_l)\cdot (n+\chi_{l,1}) = \chi_{n+1,l}\cdot (n+\chi_{l,1})\cdot(n+\beta_l+1) = [\chi_{n,l} \cdot (n+\beta_l)]+1 = \alpha_{l,n} + 1$.
\end{proof}

\begin{lemma}
\label{general-lemmata-5}
\label{lamtsh}
Let $X \in Sp^\Sigma(\Dst,T)$. Then the maps $\lamt_{sh X}$ and $sh \lamt_X$ 
are equal in $Sp^\Sigma(\Dst,T)$ up to a canonical isomorphism
of the targets.
\end{lemma}
\begin{proof}
We use the isomorphism $\Omega sh(sh X) \xrightarrow{\cong} sh(\Omega sh X)$ 
which is levelwise given by the  $\Sigma_n$-equivariant map $\Omega X_{1+1+n} \xrightarrow{\Omega (\chi_{1,1}+n)}  \Omega X_{1+1+n}$. 
This really is a map in $Sp(\Dst,T)$, as\\
$\centerline{\xymatrix{
\Omega X_{1+1+n}\ar[d]^{\Omega (\chi_{1,1}+n)} \ar[r]^(0.45){\Omega \sigt} & \Omega^2 X_{1+1+n+1} \ar[d]^{\Omega^2 (\chi_{1,1}+n+1)} \ar[r]^{\Omega^{\chi_{1,1}}} & \Omega^2 X_{1+1+n+1} \ar[d]^{\Omega^2 (\chi_{1,1}+(n+1))}\\
\Omega X_{1+1+n}\ar[r]^{\Omega \sigt} & \Omega^2 X_{1+1+n+1} \ar[r]^(0.45){\Omega^{\chi_{1,1}}} & \Omega^2 X_{1+1+n+1}\\
}}$\\

commutes by Lemma \ref{lem-sigt-omega}.
This yields a commutative diagram\\
$\centerline{\xymatrix{
sh X \ar[r]^(0.4){\lamt_{sh X}} \ar[rd]^{sh \lamt_X} & \Omega sh( sh X) \ar[d]^{\cong} \\
& sh (\Omega sh X)\\
}}$\\

as we have levelwise\\
$\centerline{\xymatrix{
X_{1+n} \ar@{=}[d] \ar[r]^(0.45){\sigt} & \Omega X_{1+n+1} \ar@{=}[d] \ar[r]^(0.45){\Omega (1+\chi_{n,1})} & \Omega X_{1+1+n} \ar[d]^{\Omega (\chi_{1,1}+n)} \\
X_{1+n} \ar[r]^(0.45){\sigt}& \Omega X_{1+n+1} \ar[r]^{\Omega \chi_{1+n,1}} & \Omega X_{1+1+n}\\
}}$\\
\end{proof}

\begin{lemma}
\label{general-lemma-6}
\label{Rstab}
Let $X \in Sp^\Sigma(\Dst,T)$. Then we have a natural isomorphism
$(\Theta^\infty R X)_n \cong (\Theta^\infty X)_n$.
\end{lemma}
\begin{proof}
The isomorphismus is induces by the following chain of compatible isomorphisms:\\
        $\centerline{\xymatrix{
        \Omega X_{n+1} \ar[d]^{1} \ar[r]^{\Omega \sigt_{n+1}^X} & \Omega^2 X_{n+2} \ar[d]^{\Omega^{\chi_{1,1}}} \ar[r]^{\Omega^2 \sigt_{n+2}^X} & \ldots \ar[r] & \Omega^l X_{n+l} \ar[d]^{\Omega^{\chi_{1,l-1}}} \ar[r]^{\Omega^l \sigt_{n+l}^X}  & \Omega^{l+1} X_{n+l+1} \ar[r]^{\Omega^{l+1} \sigt_{n+l+1}^X} \ar[d]^{\Omega^{\chi_{1,l}}} &\\
\Omega X_{1+n} \ar[r]^{\sigt_n^{RX}} & \Omega^2 X_{1+n+1} \ar[r]^{\Omega \sigt_{n+1}^{RX}} & \ldots \ar[r] & \Omega^l X_{1+n+l-1} \ar[r]^(0.4){\Omega^{l-1} \sigt_{n+l-1}^{RX}}  & \Omega^{l+1} X_{1+n+l} \ar[r]^{\Omega^l \sigt_{n+l}^{RX}} &
}}$\\

    This diagram commutes as the following does and we have $\Omega^{l-1+\chi_{1,1}} = \Omega^{l-1} \Omega^{\chi_{1,1}}$ (see Lemma-Definition \ref{omega-n-comp}):\\
$\centerline{\xymatrix{
\Omega^l X_{n+l} \ar[d]^{\Omega^{\chi_{1,l-1}}} \ar[rr]^{\Omega^l \sigt_{n+l}^X} & & \Omega^{l+1} X_{n+l+1} \ar[d]^{\Omega^{\chi_{1,l}}} \ar[dl]^{\Omega^{\chi_{1,l-1}+1}}\\
\Omega^l X_{1+n+l-1} \ar[r]_*!/d3pt/{\labelstyle \Omega^l \sigt_{1+n+l-1}^{X}} & \Omega^{l+1} X_{1+n+l} \ar[r]_*!/d3pt/{\labelstyle \Omega^{l-1+\chi_{1,1}}} & \Omega^{l+1} X_{1+n+l}
}}$
\end{proof}

\begin{proposition}
\label{basics-fgmc-stabeq-R}
Let $(\Dst,\we,S^0)$ be as in Theorem
\ref{stabgenmod}.
Then the endofunctor $R$ preserves stable weak equivalences
in  $Sp(\Dst,T)$ between level fibrant objects in $Sp^\Sigma(\Dst,T)$.
\end{proposition}
\begin{proof}
Let $f: X \rightarrow Y$ be a map in $Sp^\Sigma(\Dst,T)$ between level fibrant
objects which is a stable weak equivalence in $Sp(\Dst,T)$. Then by assumption
$\Theta^\infty f$ is a level equivalence. By Lemma \ref{Rstab}, we have $(\Theta^\infty Rf)_l \cong (\Theta^\infty f)_l$ for all $l \in \Nbb_0$. Hence $\Theta^\infty Rf$ is a level equivalence and $RX, RY$ are level fibrant objects 
($\Omega$ preserves fibrant objects), and consequently $Rf$ is a stable 
weak equivalence again by assumption.
\end{proof}

We now establish a first incomplete generalization 
of Schwede's Theorem \ref{thorg}. Then we provide an example for
$\Dst$ which satisfies the hypotheses.
\begin{theorem}
    \label{general-theorem-th}
    \label{thmin}
    Let $(\Dst,\we,S^0)$ be a symmetric monoidal model category
and $T$ a cofibrant object. Assume that for $Sp(\Dst,T)$ the 
projective level model structure (see e.g. \cite[Theorem 1.13]{H1}) exists.
Assume further that:
    \begin{enumerate}[(a)]
    \item
    For any map $f$ in $Sp(\Dst,T)$ the following are equivalent
    (compare also Theorem \ref{basics-fgmc-stabeq-th}):
    \begin{itemize}
    \item $f$ is a stable equivalence.
    \item For any level fibrant replacement $f^\prime$ of $f$, we have that 
    $\Theta^\infty f^\prime$ is a level equivalence.
    \item There is a level fibrant replacement $f^\prime$ of $f$ such that 
    $\Theta^\infty f^\prime$ is a level equivalence.
    \end{itemize}
    \item Countable compositions of stable equivalences in $Sp(\Dst,T)$ 
between level fibrant objects are stable equivalences in $Sp(\Dst,T)$.
    \item $\Omega$ commutes with sequential colimits in
    $\Dst$ (see also Theorem \ref{basics-fgmc-stabeq-th}).
    \item Sequential colimits of fibrant objects in $\Dst$ are fibrant.
    \end{enumerate}
    Let $X$ be a symmetric spectrum in $Sp^{\Sigma}(\Dst,T)$ which is levelwise
 fibrant. Then (i) to (iv) below are equivalent, and (v) follows from these.
    \begin{enumerate}[(i)]
        \item There is a map in $Sp^{\Sigma}(\Dst,T)$ from $X$ to an $\Omega$-spectrum which is a stable equivalence in $Sp(\Dst,T)$.
        \item The morphism $\tilde{\lambda}_X: X \longrightarrow RX$ is a stable equivalence in $Sp(\Dst,T)$.
        \item For all $n\in \Nbb_0$, the cycle operator $d_{sh^n X}: (\Theta^{\infty} sh^n X)_0 \rightarrow (\Theta^{\infty} sh^n X)_0$ is a weak equivalence.
        \item The symmetric spectrum $R^{\infty}X$ is an $\Omega$-spectrum.
        \item The morphism $\lambda_X^{\infty}: X \longrightarrow R^{\infty} X$ in $Sp^\Sigma(\Dst,K)$ is a stable equivalence in $Sp(\Dst,T)$.
    \end{enumerate}
\end{theorem}

\begin{proof}
\begin{itemize}
    \item $(i) \Longrightarrow (ii)$
   Let $f: X \rightarrow Y$ be a map in $Sp^\Sigma(\Dst,T)$ with $Y$ being an
$\Omega$-spectrum, and such that $f$ is a stable equivalence in 
$Sp(\Dst,T)$. Then by Proposition \ref{basics-fgmc-stabeq-R}, $Rf$ is also a stable equivalence in $Sp(\Dst,T)$. This implies that $\lamt_Y$ ($(\lamt_Y)_l = \chi_{l,1}^Y \cdot \sigt_l^Y$) is a level equivalence, and hence 
a stable equivalence in $Sp(\Dst,T)$, follows by naturality of $\lamt$
that \\
    $\centerline{
        \xymatrix{X \ar[r]^{f} \ar[d]^{\lamt_X} & Y \ar[d]^{\lamt_Y} \\
        RX \ar[r]^{Rf} & RY}
    }$
    commutes, hence by the 2-out-of-3 axiom 
$\lamt_X$ is a stable equivalence in $Sp(\Dst,T)$.

    \item $(ii) \Leftrightarrow (iii)$    
    We have $d_{sh^n X} \cong (\Theta^\infty \lamt_{sh^n X})_0$ by Lemma \ref{dlamt} and  $\lamt_{sh^n X} \cong sh^n \lamt_X$ by Lemma \ref{lamtsh}, and 
furthermore $(\Theta^\infty sh^n)_0 \cong (\Theta^\infty)_n$, hence $d_{sh^n X} \cong (\Theta^\infty \lamt_X)_n$. As $\Omega$ is a right Quillen functor on
$\Dst$, both $X$ (by assumption) and $RX$ are level fibrant. Using $(a)$
and the above isomorphism, we deduce that $\lamt_X$ is a stable
equivalence in $Sp(\Dst,T)$ if for every $n \in \Nbb_0$ 
the map $d_{sh^n X}$ is a weak equivalence in $\Dst$.
    \item $(iii) \Leftrightarrow (iv)$    
        By Proposition \ref{basics-fgmc-stabeq-R},  the maps $R^s \lamt_X$ 
are stable equivalences in $Sp(\Dst,T)$ for all $s \in \Nbb_0$ between
level fibrant objects ($\Omega$ is right Quillen). By $(b)$, the
inclusion $\lambda_X^{\infty}$ is then a stable equivalence in $Sp(\Dst, T)$.
    \item $(ii) \Longrightarrow (i)$    
        This follows from  $(ii) \Rightarrow (iv), (v)$.
\end{itemize}
\end{proof}    

An important class of examples is given by almost finitely generated model 
categories:
\begin{proposition}
    \label{general-theorem-applcase-fgmc}
Let $\Dst$ be a symmetric monoidal model category which is almost finitely 
generated, and let $T$ be a cofibrant object of $\Dst$.
Assume that sequential colimits commute with
finite products, weak equivalences and $\Omega$,
and that the projective level model structure on
$(\Dst, T)$ exists. Then the couple $(\Dst, T)$ 
satisfies the hypotheses of Proposition \ref{thmin}.
\end{proposition}

\begin{proof}
\begin{enumerate}[(a)]
\item holds by \ref{stabgenmod}.
\item 
We show more generally that stable equivalences in
$Sp(\Dst,T)$ are closed under sequential colimits.
Using a standard reduction, it suffices to show 
that sequential colimits of stable equivalences between stably fibrant objects
in $Sp(\Dst,T)$ are stable equivalences. But the stable model structure
is a left Bousfield localization of the projective level model structure,
hence stable equivalences between stably fibrant objects
are level equivalences \cite[Theorem 3.2.13, Prop. 3.4.1]{Hi}. 
By assumption, those are preserved by sequential colimits
(as these are defined level-wise), hence are stable equivalences again.
\item holds by assumption.
\item holds by \cite[Lemma 4.3]{H1}.
\end{enumerate}
\end{proof}

We now consider the category $M.(S)$ of pointed simplicial presheaves 
on $Sm/S$ for a given noetherian base scheme $S$ (sometimes called
the category of {\it motivic spaces}). Besides the injective
\cite{MV} and the projective motivic model structure,
there is a third model structure introduced in \cite[Section A.3]
{P1} and denoted by $M.^{cm}(S)$
which is convenient for our purposes. (Recall also \cite{MV}, \cite{J}
that there is a model structures on pointed simplicial sheaves
$sShv(S).$ which is -- via the sheafification $a$ as a left Quillen functor
-- Quillen equivalent to the injective model structure on $M.(S)$.)

\begin{corollary}\label{cmisgood}
\label{propmotivic}
The assumptions of Proposition \ref{thmin} are satisfied for the model category
$\Dst=M.^{cm}(S)$ and for all cofibrant objects $T$ for which $Hom(T,-)$ 
commutes with sequential colimits (in particular for $T = \Pbb^1$).
\end{corollary}
\begin{proof}
The projective level model structure exists by \cite[Theorem 1.13]{H1}.
The model category $M.^{cm}(S)$ is symmetric monoidal by 
\cite[Theorem A.17]{P1} and weak equivalences are stable under
sequential colimits by \cite[Lemma A.18]{P1}.
The claims about $T$ and $\Pbb^1$ follow from \cite[Lemma A.10]{P1}
and \cite[Lemma 2.5]{DRO}. To show that $M.^{cm}(S)$
is almost finitely generated, one shows that the model category
$M.^{cs}(S)$ (see \cite[Section A.3]{P1}) is almost fintely
generated, left proper and cellular. From this, one deduces
that the left Bousfield-Hirschhorn localization 
$M.^{cm}(S)$ exists and is still almost finitely
generated. See \cite{NS} or 
\cite[Propositions 2.20, 2.44 and 2.49]{H}
for further details.
\end{proof}

The model category $sSet_*$ together with $T=S^1$ also satisfies the
assumptions of  \ref{thmin}. By Lemma \ref{lem-dshn} below,
the map $d_{sh^n X}$ is a weak equivalence for all $n \geq 0$ 
if and only if the cycle operator $d$ induces bijections on all
stable homotopy groups $\pihat_k(X), k \in \Zbb$.
Moreover, the stable equivalences in $Sp(sSet_*,S^1)$ 
are precisely the $\pihat_*$-equivalences. Hence Proposition \ref{thmin} 
really is a partial generalization of Theorem \ref{thorg}.

\begin{lemma}
\label{lem-dshn} Let $X \in Sp^{\Sigma}(sSet_*,S^1)$.
Then $d_{sh^n X}: (\Theta^{\infty} sh^n X)_0 \rightarrow (\Theta^{\infty} sh^n X)_0$ is a weak equivalence if and only if the cycle operator
$d$ induces bijections on all stable homotopy groups
$\pihat_{k-n}(X), k \in \Nbb_0$.
\end{lemma}
\begin{proof}
Lemma \ref{dlamt} shows that $d_{sh^n X}$ is a weak equivalence if and only
if $\pi_k(d_{sh^n X})$ is a bijection for all $k \geq 0$.
Using \cite[Construction I.4.12]{S1} and section \ref{sect-signcircle},
we see that $\pi_k(d_{sh^n X})$ is isomorphic to the action
of $d$ on $\pihat_k(sh^n X)$. We also have isomorphisms
of $\Mst$-modules $\pihat_k(sh^n X) \cong \pihat_{k-n}(X)(n)$ 
(see Propostion \ref{app-motstgroups-isos-prop}, the remark after
Definition \ref{motsemi-mactionorig-def} and the example after 
\ref{app-motstgroups-def1}). By tameness, $d$ acts as an automorphism
on $\pihat_{k-n}(X)(n)$ if and only if the $\Mst$-action on $\pihat_{k-n}(X)(n)$ is trivial. Again by tameness, this in turn holds if and only if
the $\Mst$-action on $\pihat_{k-n}(X)$ is trivial, because then 
the filtration is bounded (see Lemma \ref{mit-lem-tame}). 
This is also equivalent to $d$ acting trivially on $\pihat_{k-n}(X)$.
\end{proof}

We now state a first version of our definition
of semistability (see also Definition \ref{def-semistab-new} 
and the remark thereafter):
\begin{definition}
\label{general-theorem-semi-def}
Assume that the assumptions of Proposition \ref{thmin} are satisfied
and the projective level structure on $Sp^\Sigma(\Dst,T)$ exists, 
in particular the functorial fibrant approximation $J^\Sigma$. 
Then in this section, a symetric spectrum $X \in Sp^\Sigma(\Dst, T)$ 
is called \emph{semistable} if $J^\Sigma X$ satisfies one
(and hence all) of the above properties $(i)-(iv)$.
\end{definition}

Using this definition, we have (compare also \cite[Proposition 5.6.5]{H2}):
\begin{proposition}
Assume that the assumptions of Proposition \ref{thmin} are satisfied
and the projective level structure on $Sp^\Sigma(\Dst,T)$ exists.
Let $f: X \rightarrow Y$ be a morphism in $Sp^\Sigma(\Dst,T)$ 
between semistable symmetric spectra, and assume that
the forgetful functor $U: Sp^\Sigma(\Dst,T) \rightarrow Sp(\Dst,T)$ 
reflects stable equivalences. Then if $f$ is a stable equivalence
in $Sp^\Sigma(\Dst,T)$, then so is $U(f)$ in $Sp(\Dst,T)$.
\end{proposition}
\begin{proof}
It is enough to show the claim for $J^\Sigma f$.
Namely,  $Z \rightarrow J^\Sigma Z$ 
is a natural level equivalence, hence we may replace 
$f$ by $J^\Sigma f$ and assume that $X$ and $Y$ are level fibrant
and the hypotheses of Proposition \ref{thmin} hold for $X$ and $Y$.
In the commutative diagram in $Sp^\Sigma(\Dst,T)$\\
$\centerline{\xymatrix{
X \ar[d]^f \ar[r]^{\lamt_X^\infty} & R^\infty X \ar[d]^{R^\infty f} \\
Y \ar[r]^{\lamt_Y^\infty} & R^\infty Y
}}$
$R^\infty X$ and $R^\infty Y$ are $\Omega$-spectra by assumption,
and hence fibrant objects for the stable model structure on 
$Sp^\Sigma(\Dst,T)$. Also, $U(\lamt_X^\infty)$ and $U(\lamt_Y^\infty)$ 
are stable equivalences. Using the assumptions on $U$, we see that
$\lamt_X^\infty$ and $\lamt_Y^\infty$ are stable equivalences
in $Sp^{\Sigma}(\Dst,T)$. But $f$ is a stable equivalence in 
$Sp^\Sigma(\Dst,T)$, hence by  \cite[Theorem 3.2.13]{Hi}
$R^\infty f$ is a level equivalence. Therefore $U(R^\infty f)$ 
(and thus $U(f)$) is a stable equivalence.
\end{proof}
The condition that $U$ reflects stable equivalences is satisfied 
for $\Dst = M_\cdot^{cm}(S)$, because by
\cite[Theorem A.5.6 and Theorem A.6.4]{P1} the stable equivalences
for $Sp(\Dst,T)$ and $Sp^{\Sigma}(\Dst,T)$ in \cite{J} resp. \cite{P1} 
coincide and for the stable equivalences in \cite{J} the condition
is satisfied by \cite[Prop. 4.8]{J}.

Comparing Proposition \ref{thmin} with Theorem \ref{thorg}, one notices
that several things are missing. We will provide what is missing
below (see Theorem \ref{motsemi-theoremorg-th}). 

\subsection{The sign $(-1)_T$ and the action of the symmetric group} 
We now axiomatize some properties
of the topological circle, in a way which is convenient
for studying the $\Mst$-action on generalized stable
homotopy groups. The following two subsections then discuss
the two key examples, namely $T=S^1$ in pointed simplicial sets
and $T=\Pbb ^1$ in pointed motivic spaces.

Let $(\Dst, \we, S^0)$ be a symmetric monoidal model category.
Fix a cofibrant object $T$ in $\Dst$ and set $T^n := T^{\we n}$.

\begin{definition}
\label{vorzeichen-def}
    A \emph{sign} of $T$ in $\Dst$ is an automorphism $(-1)_T$ of $T$ in $Ho(\Dst)$ of order $2$ with the following properties:

\begin{enumerate}
    \item For any $\tau \in \Sigma_n$, the permutation of factors 
 $T^n \xrightarrow{\tau} T^n$ coincides with $|\tau|_T \we T^{n-1}$ 
in $Ho(\Dst)$ (the latter map is defined as $T$ is cofibrant), where we
set $|\tau|_T = (-1)_T$ if $\tau$ is an odd permutation and 
$|\tau|_T = 1$ otherwise. We call $|\tau|_T$ the \emph{sign of the 
permutation} $\tau$.
    \item $T^2 \xrightarrow{(-1)_T\we 1_T} T^2$ coincides with
$T^2 \xrightarrow{1_T\we (-1)_T} T^2$ in $Ho(\Dst)$.
\end{enumerate}
\end{definition}

\subsubsection{The sign of the simplicial circle}
\label{sect-signcircle}
Let $\Dst = sSet_*$ with the usual smash product.

\begin{definition}
Fix a homeomorphism $h: |S^1| \cong S^1$.
This yields a weak equivalence $\nu: S^1 \xrightarrow{\sim} Sing(|S^1|) \xrightarrow{h} Sing(\Rbb^+)$ in $sSet_*$. The map $(-1)_{\Rbb} :\Rbb \rightarrow \Rbb, t \mapsto -t$ then induces the automorphism 
$$(-1)_{S^1} = \nu^{-1} \cdot Sing((-1)_{\Rbb}^+) \cdot \nu$$
of $S^1$ in $Ho(sSet_*)$, which we call the \emph{sign of $S^1$},
and which is obviously of order $2$.
\end{definition}

In particular $(-1)_{\Rbb^+}$ has degree $-1$.

\begin{lemma}
The above automorphism $(-1)_{S^1}$ is a sign of $S^1$.
\end{lemma}
\begin{proof}
It is enough to check the properties of Definition \ref{vorzeichen-def}
in $Ho(Top_*)$, that is after geometric realization.
It also suffices to check the equalities after conjugation
with the canonical isomorphism $$|(S^1)^{\we n}| \rightarrow |S^1|^{\we n} \xrightarrow{h^{\we n}} (\Rbb^+)^{\we n} \rightarrow \Rbb^{n+}$$ 
(here we used that $-^+$ is strictly monoidal)
Conjugation of $\tau: (S^1)^{\we n} \rightarrow (S^1)^{\we n}$ 
then yields the map $\Rbb^{n +} \xrightarrow{\tau^+} \Rbb^{n+}$
because 
$$\xymatrix{
|(S^1)^{\we n}| \ar@{}[r]|{\cong} \ar[d]^{|\tau|} & |S^1|^{\we n} \ar@{}[r]|{\cong} \ar[d]^{\tau} & (\Rbb^+)^{\we n} \ar@{}[r]|{\cong} \ar[d]^{\tau} & \Rbb^{n+} \ar[d]^{\tau^+} \\
|(S^1)^{\we n}| \ar@{}[r]|{\cong} & |S^1|^{\we n} \ar@{}[r]|{\cong} & (\Rbb^+)^{\we n} \ar@{}[r]|{\cong} & \Rbb^{n+} \\
}$$ commutes ($-^+$ is symmetric monoidal).

After conjugation, the map $|(-1)_{S^1} \we (S^1)^{\we n-1}|$ 
yields $\Rbb^{n+} \xrightarrow{diag(-1,1,\ldots,1)^+} \Rbb^{n+}$
because the following diagram commutes
(here we use relations between units and counit):\\
\begin{equation*}\centerline{\xymatrix{
|S^1 \we (S^1)^{\we n-1}| \ar[d]^{|\nu \we 1|} \ar[r]^{\cong} &  |S^1| \we |S^1|^{\we n-1} \ar[d]^{|\nu| \we 1} \ar[r]^(0.6){\cong} & (\Rbb^+)^{\we n} \ar[r]^{\cong} \ar@{=}[d] & \Rbb^{n+} \ar[ddd]^{((-1)_{\Rbb}\times 1_{\Rbb^{n-1}})^+}\\
|Sing(\Rbb^+) \we (S^1)^{n-1}| \ar[d]^{|Sing((-1)_{\Rbb}^+) \we 1|} \ar[r]^{\cong} & |Sing(\Rbb^+)| \we |S^1|^{\we n-1} \ar[d]^{|Sing((-1)_{\Rbb}^+)| \we 1} \ar[r]^(0.6){\sim} &  (\Rbb^+)^{\we n} \ar[d]^{(-1)_{\Rbb}^+ \we 1}\\
|Sing(\Rbb^+) \we (S^1)^{\we n-1}| \ar[r]^{\cong} & |Sing(\Rbb^+)| \we |S^1|^{\we n-1}  \ar[r]^(0.6){\sim} &  (\Rbb^+)^{\we n} \\
|S^1 \we (S^1)^{\we n-1}| \ar[u]_{|\nu \we 1|} \ar[r]^{\cong} &  |S^1| \we |S^1|^{\we n-1} \ar[u]_{|\nu| \we 1} \ar[r]^(0.6){\cong} & (\Rbb^+)^{\we n} \ar[r]^{\cong} \ar@{=}[u] & \Rbb^{n+}
}}\end{equation*}
and $\Rbb^{2+} \xrightarrow{diag(1,-1)^+} \Rbb^{2+}$ is the conjugated
map of $|S^1 \we (-1)_{S^1}|$.\\
Now let $\tau \in \Sigma_n$ and $P_\tau \in GL_n(\Rbb)$ 
the permutation matrix corresponding to $\tau$.
If $\tau$ is odd, then $\det P_\tau = -1 = \det diag(-1,1,...,1)$. 
Lemma \ref{gln-top} then implies that the maps  $\tau^+: \Rbb^{n+} \rightarrow \Rbb^{n+}$ and $diag(-1,1,...,1)^+: \Rbb^{n+} \rightarrow \Rbb^{n+}$ 
are equal in $Ho(Top_*)$, hence $\tau: (S^1)^{\we n} \rightarrow (S^1)^{\we n}$ and $(-1)_{S^1} \we (S^1)^{\we n-1}$ are also equal in $Ho(sSet_*)$.
If $\tau$ is even , then $\det P_\tau = 1 = det E_n$,
and the maps $\tau: (S^1)^{\we n} \rightarrow (S^1)^{\we n}$ 
equals the identity on $(S^1)^{\we n}$ in $Ho(sSet_*)$.
For the second condition, note that the diagonal matrices
$diag(-1,1)$ and $diag(1,-1)$ have the same determinant,
so by Lemma \ref{gln-top} the maps $\Rbb^{2+} \xrightarrow{diag(1,-1)^+} \Rbb^{2+}$ and $\Rbb^{2+} \xrightarrow{diag(-1,1)^+} \Rbb^{2+}$ are equal in $Ho(Top_*)$ and therefore also $(-1)_{S^1} \we S^1$ and $S^1 \we (-1)_{S^1}$.
\end{proof}

We have just used the following:
\begin{lemma}
\label{gln-top}
The topological group
$GL_n(\Rbb)$ has two path components
(corresponding to the sign of the determinant).
If $A,B \in GL_n(\Rbb)$ have determinants with the same sign,
then the two pointed maps  $\Rbb^{n+} \xrightarrow{A^+, B^+} \Rbb^{n+}$
are equal in  $Ho(Top_*)$.
\end{lemma}
\begin{proof}
Well-known.
\end{proof}

\subsubsection{The sign $(-1)_{\Pbb^1}$ of the projective
line}
\label{subsect-sign-proj}

We have a pushout diagram (both in $Sm/S$ and in $sShv(S).$)
$$
\xymatrix{
\Gbb_{m\,S}= D_+(T_0 T_1) \times S \ar[r]^{i_1^\prime} \ar[d]^{i_0^\prime} & \Abb^1_S = D_+(T_1) \times S \ar[d]^{i_1} \\
\Abb^1_S = D_+(T_0) \times S \ar[r]^{i_0} & \Pbb^1_S\\
}
$$
The base point of $\Pbb^1_\Zbb$ is the closed immersion
$Spec(\Zbb) \xrightarrow{0} \Abb^1_\Zbb \xrightarrow{i_1} \Pbb^1_\Zbb$
and its base change 
$S \rightarrow \Pbb^1_S$ is the base point of $\Pbb^1_S$.
The latter map induces a base point map $\Pbb^1_S$: $\infty: \ast = S \rightarrow \Pbb^1_S$ which is closed for the $cm$-model structure above
(see Corollary \ref{cmisgood}). For that model structure,
$(\Pbb^1_S,\infty)$ is a cofibrant pointed motivic space
which we denote by $\Pbb^1$ from now on.
Similarly, we write $\Gbb_m$ for the $cm$-cofibrant pointed
motivic space $(\Gbb_{m\,S},1)$.

Now we define the sign of $\Pbb^1$.
(See also \cite[6.1 The element $\epsilon$]{Mo}
for the sign of $\Pbb^1$ and its behaviour with respect
to $\Pbb^1 \cong S^1 \we \Gbb_m$.)

\begin{definition}\label{signp1}
The automorphism $\Pbb^1_\Zbb \rightarrow \Pbb^1_\Zbb$
given by the graded isomorphism $\Zbb[T_0,T_1] \rightarrow \Zbb[T_0,T_1], T_0 \mapsto -T_0, T_1\mapsto T_1$ is denoted by $(-1)_{\Pbb^1_\Zbb}$, and similarly
$(-1)_{\Pbb^1_S} = (-1)_{\Pbb^1_\Zbb} \times S$ for the base change of
that automorphism to $Sm/S$.
The following Lemma shows that $(-1)_{\Pbb^1_S}$ induces
an automorphism $(-1)_{\Pbb^1_S}$ on $\Pbb^1$, which we 
denote by $(-1)_{\Pbb^1}$, and call it
the \emph{sign of $\Pbb^1$}. 
\end{definition}

\begin{lemma}
\label{p1a1-lem}
The automorphism $(-1)_{\Pbb^1_\Zbb}$ is the morphism induced by 
(the push-outs of) the following diagram:
$$
\xymatrix{
D_+(T_0) \ar[d]^{\frac{T_1}{T_0} \mapsto -\frac{T_1}{T_0}} & \ar[l] D_+(T_0 T_1) \ar[r] \ar[d]^{\frac{T_1}{T_0} \mapsto -\frac{T_1}{T_0}} & D_+(T_1) \ar[d]^{\frac{T_0}{T_1} \mapsto -\frac{T_0}{T_1}}\\
D_+(T_0) & \ar[l] D_+(T_0 T_1) \ar[r] & D_+(T_1)
}
$$
Consequently, the diagram
$$\xymatrix{
\Abb^1_S \cup_{\Gbb_{m\,S}} \Abb^1_S \ar[d]^{(-1)\cup_{(-1)} (-1)} \ar[r] & \Pbb^1_S \ar[d]^{(-1)_{\Pbb^1_S}}\\
\Abb^1_S \cup_{\Gbb_{m\,S}} \Abb^1_S \ar[r] & \Pbb^1_S
}$$ commutes where $(-1)$ on coordinates is given by
$T \mapsto -T$. Hence $(-1)_{\Pbb^1_S}$ respects the base point $\infty$.
\end{lemma}
\begin{proof}
Straightforward. For the last claim, use the first one and that
$Spec(\Zbb[T]) \rightarrow Spec(\Zbb[T]), T \mapsto -T$ maps the point $T = 0$ 
to itself.
\end{proof}

From now on, we replace the motivic space $(\Pbb^1)^{\we n}$ by the weakly equivalent $\Abb^n_S/((\Abb^n_\Zbb - 0) \times S)$. On the latter, we consider
the usual $GL_{n\,S}$-action and 
relate it to the sign of $\Pbb^1$. 
\begin{lemma}
\label{app-permutations-sign-lem}
\begin{enumerate}
\item
There is a zig-zag of weak equivalences in
$M_\cdot(S)$ between the pointed spaces $\Pbb^1$ 
and $\Abb^1_S/_{i_0^\prime}\Gbb_{m\,S}$.
\item
Via this zig-zag, the pointed map $(-1)_{\Pbb^1_S}$ corresponds to 
the map $(-1)_{\Abb^1_S}/(-1)_{\Gbb_{m\,S}}$.
\end{enumerate}
\end{lemma}
\begin{proof}
\begin{enumerate}
\item
We have a commutative diagram $$
\xymatrix{
\Abb^1_S \ar[d]^{1} & \ar[l]^{i_0^\prime} \Gbb_{m\,S} \ar[d]^{1} \ar[r]^{i_1^\prime} & \Abb^1_S \ar[d]^{\sim} \\
\Abb^1_S & \ar[l]^{i_0^\prime} \Gbb_{m\,S} \ar[r] & S
}
$$
The map $\Gbb_{m\,S} \xrightarrow{i_0^\prime} \Abb^1_S$ is a monomorphism,
and the vertical maps are weak equivalences. As the injective model structure
is left proper, the induces map
$f: \Abb^1_S \amalg_{\Gbb_{m\,S}} \Abb^1_S \rightarrow \Abb^1_S \amalg_{\Gbb_{m\,S}} * = \Abb^1_S/_{i_0^\prime}\Gbb_{m\,S}$ 
is a weak equivalence, too. The motivic space
$\Abb^1_S \amalg_{\Gbb_{m\,S}} \Abb^1_S$ is pointed by $S \xrightarrow{0} \Abb^1_S \xrightarrow{incl_1} \Abb^1_S \amalg_{\Gbb_{m\,S}} \Abb^1_S$,
and with this choice $f$ is a pointed map.
The induced map $\Abb^1_S \amalg_{\Gbb_{m\,S}} \Abb^1_S \xrightarrow{(i_0,i_1)} \Pbb^1_S$ is a motivic weak equivalence, as it is an isomorphism after 
sheafification \cite[Lemma 2.1.13]{Mo}. It is pointed as 
$(i_0,i_1) \cdot incl_1 \cdot 0 = i_1 \cdot 0$.

\item
The squares 
$$\xymatrix{
\Abb^1_S \amalg_{\Gbb_{m\,S}} \Abb^1_S \ar[r] \ar[d]_{(-1)_{\Abb^1_S} \amalg_{(-1)_{\Gbb_{m\,S}}} (-1)_{\Abb^1_S}} & \Abb^1_S \amalg_{\Gbb_{m\,S}} * \ar[d]^{(-1)_{\Abb^1_S} \amalg_{(-1)_{\Gbb_{m\,S}}} 1_*} \\
\Abb^1_S \amalg_{\Gbb_{m\,S}} \Abb^1_S \ar[r] & \Abb^1_S \amalg_{\Gbb_{m\,S}} *
}$$
and
$$\xymatrix{
\Abb^1_S \amalg_{\Gbb_{m\,S}} \Abb^1_S \ar[r] \ar[d]_{(-1)_{\Abb^1_S} \amalg_{(-1)_{\Gbb_{m\,S}}} (-1)_{\Abb^1_S}} & \Pbb^1_S \ar[d]^{(-1)_{\Pbb^1_S}} \\
h(\Abb^1_S) \amalg_{\Gbb_{m\,S}} \Abb^1_S \ar[r] & \Pbb^1_S
}$$
commute by Lemma \ref{p1a1-lem}.

\end{enumerate}
\end{proof}

For any $S \to Spec({\Zbb})$, we consider the usual actions $\mu: GL_{n\,S} \times_S \Abb^n_S \rightarrow \Abb^n_S$ (on the open subscheme $(\Abb^n - 0) \times S$ as well) and homomorphisms $GL_n(\Zbb) \rightarrow Aut_{Sch_S}(\Abb^n_S)$ 
and $GL_n(\Zbb) \rightarrow Aut_{M_\cdot(S)}(\Abb^n_S/(\Abb^n - 0) \times S)$.

Above, as usual, we have identified smooth varieties and the associated
simplicially constant (pre-)sheaf given by the Yoneda embedding. However, 
to avoid confusion when it comes to base points, we will write
$h_\cdot:Sm/S \to M.(S)$ for the composition of the Yoneda embedding 
with adding a disjoint base point.

The above induces a map $h_\cdot(GL_{n\,S}) \we [\Abb^n_S/((\Abb^n-0) \times S)] \rightarrow [\Abb^n_S/((\Abb^n-0) \times S)]$, and for any $A \in
GL_n(\Zbb)$ the diagram\\
$\centerline{\xymatrix{
[\Abb^n_S/((\Abb^n-0) \times S)] \ar[d]^{\cong} \ar[ddr]^{A} \\
h_\cdot(S) \we [\Abb^n_S/((\Abb^n-0) \times S)] \ar[d]^{h_\cdot(A) \we 1} \\
h_\cdot(GL_{n\,S}) \we [\Abb^n_S/((\Abb^n-0) \times S)] \ar[r] & [\Abb^n_S/((\Abb^n-0) \times S)]\\
}}$\\
commutes. Precomposition with the monomorphism $\Sigma_n \rightarrow GL_n(\Zbb)$ yields the above  $\Sigma_n$-actions on $\Abb^n_S$ and on $\Abb^n_S/((\Abb^n - 0) \times S)$.

\begin{lemma}
\label{an-eqv-lem}
\begin{enumerate}
\item
There is a $\Sigma_n$-equivariant map 
$$f: [\Abb^1_S/\Gbb_{m\,S} ]^{\we n} \rightarrow \Abb^n_S/((\Abb^n_\Zbb - 0) \times S) $$
in $M_\cdot(S)$ which is a motivic equivalence.
\item
The diagram
$$\xymatrix{
\Abb^1_S/\Gbb_{m\,S} \we [\Abb^1_S/\Gbb_{m\,S}]^{\we n-1} \ar[r] \ar[d]^{ (-1)_{\Abb^1_S/\Gbb_{m\,S}} \we 1} & \Abb^n_S/((\Abb^n_\Zbb - 0) \times S) \ar[d]^{diag(-1,1,\ldots,1)} \\
\Abb^1_S/\Gbb_{m\,S} \we [\Abb^1_S/\Gbb_{m\,S} ]^{\we n-1} \ar[r] & \Abb^n_S/((\Abb^n_\Zbb - 0)\times S)
}$$
commutes, and similarly for $diag(1,\ldots, 1, -1)$.
\end{enumerate}
\end{lemma}
\begin{proof}
\begin{enumerate}
\item
We have a commutative diagram
$$\xymatrix{
\coprod_{i=0}^{n-1} (\Abb^1_S)^{\times i} \times \Gbb_{m\,S} \times (\Abb^1_S)^{\times n-(i+1)} \ar[d]^{\cong} \ar[r] & (\Abb^1_S)^{\times n} \ar[d]^{\cong} \\
\coprod_{i=0}^{n-1} \Abb^{1\,\times_S i}_S \times_S \Gbb_{m\,S} \times_S \Abb^{1\, \times_S n-(i+1)}_S\ar[d] \ar[r] & {\Abb^1_S}^{\times_S n} \ar[d]^{\cong} \\
(\Abb^n - 0) \times S \ar[r] & \Abb^n_S\\
}$$
in which the vertical maps are $\Sigma_n$-equivariant. 
The horizontal maps induce
the desired map $f$ on the quotients. To see that $f$ is a weak equivalence,
it suffices to show \cite[Lemma 2.6]{J2} that its sheafification is an 
isomorphism. Using the adjunction $a: M_\cdot(S) \adjunc sShv_\cdot(S): i$,
this reduces to show that for all $\Fst \in sShv_\cdot(S)$ 
the induced map $M_\cdot(S)(f,i(\Fst))$ is a bijection.
The familiy of open immersions $\{\Abb^{1\,\times_S i}_S \times_S \Gbb_{m\,S} \times_S \Abb^{1\, \times_S n-(i+1)}_S \hookrightarrow ((\Abb^n - 0) \times S); 0 \leq i \leq n-1\}$ is a Zariski covering, hence a Nisnevich covering.
Therefore, in the diagram
$$\xymatrix{
\Fst(\Abb^n_S) \ar[d]^{\cong} \ar[r] & \Fst((\Abb^n - 0) \times S) \ar[d] & \ar[l] \Fst(S) \ar[d]^{1}\\
\Fst({\Abb^1_S}^{\times_S n}) \ar[r] & \prod_{i = 0}^{n-1} \Fst(\Abb^{1\,\times_S i}_S \times_S \Gbb_{m\,S} \times_S \Abb^{1\, \times_S n-(i+1)}_S) & \ar[l] \Fst(S)
}$$
the middle vertical map is injective. It follows that the induced map on
pull-backs is bijective, and that one coincides with $M_\cdot(S)(f,i(\Fst))$.
\item
The first diagram above is compatible
with the corresponding maps for $diag(-1,1,\ldots,1)$.
(Apply the monomorphism $((\Abb^n - 0) \times S) \hookrightarrow \Abb^n_S$
to see this for the lower left map.)
\end{enumerate}
\end{proof}

The above together with Lemma \ref{gln-top} below
leads to the main result of this subsection:
\begin{proposition}\label{p1hassign}
The automorphism $(-1)_{\Pbb^1}$ is a sign of $\Pbb^1$ in 
$M_\cdot^{cm}(S)$.
\end{proposition}
\begin{proof}
By Definition \ref{signp1}, the automorphism $(-1)_{\Pbb^1}$
has order $2$. Using that the smash product of weak equivalences
in $M_\cdot(S)$ is again a weak equivalence, as well as Lemmas 
\ref{app-permutations-sign-lem} and \ref{an-eqv-lem}, 
the required properties of Definition \ref{vorzeichen-def} 
follow from the following:
\begin{enumerate}
\item
Let $\tau \in \Sigma_n$ be a permutation. Then, in $Ho(M_\cdot^{cm}(S))$
the automorphism induced by $\tau$ on
$\Abb^n_S/((\Abb^n_\Zbb - 0)\times S)$
equals $diag(-1,1,\ldots,1)$ if $\tau$ is an odd permutation,
and the identity if $\tau$ is even.
\item
The automorphisms $diag(-1,1)$ and $diag(1,-1)$ of $\Abb^2_S/(\Abb^2_\Zbb - 0)\times S$ are equal in $Ho(M_\cdot^{cm}(S))$.
\end{enumerate}
Using Lemma  \ref{app-permutations-connect-lem}-$2.$,
these in turn follow from
\begin{enumerate}
\item
$\det P_\tau =  \left\{
\begin{array}{l l}
  \det diag (-1,1,\ldots,1) & \quad \text{$\tau$ is odd}\\
  \det diag(1,\ldots,1) & \quad \text{$\tau$ is even}\\
\end{array} \right.$
\item
$\det diag(-1,1) = \det diag(1,-1)$
\end{enumerate}
\end{proof}

\begin{lemma}
\label{app-permutations-connect-lem}
Let $A_0, A_1 \in GL_n(\Zbb)$ two matrices with $A_1 A_0^{-1} \in SL_n(\Zbb)$. 
Via the inclusion $GL_n(\mathbb{Z}) \hookrightarrow GL_n(\mathcal{O}_S(S)) 
\cong Sch_S(S, GL_{n,S})$, these matrices induce morphisms
$\underline{A_0},\underline{A_1}: S \rightarrow GL_{n,S}$ in $Sm/S$.
\begin{enumerate}
\item
There is a map $f: \Abb^1_S \rightarrow GL_{n,S}$ in $Sm/S$
with $f\cdot i_l = \underline{A_l}$ for $l=0,1$, where $i_l: S \rightarrow 
\Abb^1_S$ are the morphisms represented
by $0$ and $1$ in $\mathcal{O}_S(S)$:
\begin{equation*}\centerline{\xymatrix{
    S \amalg S \ar[r]^{(\underline{A_0}, \underline{A_1})} \ar[d]_{(i_0, i_1)} & GL_{n,S}\\
    \Abb^1_S \ar@{.>}[ru]^{f} &
}}\end{equation*}
\item
For any pointed motivic space $E$ and $\mu: h_\cdot(GL_{n\,S}) \we E \rightarrow E$ a map in $M_\cdot(S)$,
the endomorphisms on $E$ induced by
$\underline{A_0}$ and $\underline{A_1}$ are equal in
$Ho(M_\cdot^{cm}(S))$.
\end{enumerate}
\end{lemma}
\begin{proof}
\begin{enumerate}
\item By adjunction, a map $f: \Abb^1_S \rightarrow GL_{n,S}$ in $Sm/S$ 
corresponds uniquely to a matrix $\tilde{A} = \tilde{A}(f) \in 
GL_n(\mathcal{O}_S(S)[T])$.
On global sections, $i_l$ is given by $\mathcal{O}_S(S)[T] \rightarrow 
\mathcal{O}_S(S), T \mapsto l$. Therefore, the condition that
a lift $f$ exists corresponds to the equalities $\tilde{A}(l) = A_l$
for $l=0,1$, where $\tilde{A}(l)$ is the image of
$\tilde{A}$ under $GL_n(\mathcal{O}_S(S)[T]) \rightarrow 
GL_n(\mathcal{O}_S(S)), T \mapsto l$.
We may assume that $A_0 = E$ is the unit matrix and
$A_1 \in SL_n(\Zbb)$. (If the couple ($E$, $A_1 \cdot A_0^{-1}$) 
allows for a lift $\tilde{A} \in GL_n(\mathcal{O}_S(S)[T])$,
then $\tilde{A} \cdot A_0 \in GL_n(\mathcal{O}_S(S)[T])$ is a lift
for ($A_0$, $A_1$) with $A_0$ constant with respect to $T$.)
We may further assume that $A_1$ is an elementary matrix,
as $T \mapsto l$ is multiplicative. Namely, if $\tilde{A}$ is a lift
of $(A_0, A_1)$ and $\tilde{B}$ is a lift of $(B_0, B_1)$, then
$\tilde{A}\tilde{B}$ is a lift of $(A_0 B_0, A_1 B_1)$.
Finally, for $A_0 = E$ and $A_1 = E_{k,l}(a)$ an elementary matrix
with $a \in \mathbb{Z}$, we may choose $\tilde{A} := E_{k,l}(a T) 
\in GL_n(\mathcal{O}_S(S)[T])$ as a lift.
\item
If $pr: \Abb^1_S \rightarrow S$ is the projection, we have
$h_\cdot(pr) \cdot h_\cdot(i_l) = h_\cdot(pr \cdot i_l) = h_\cdot(1_S) = 
1_{h_\cdot(S)}, l=0,1$. As $h_\cdot(pr)$ is a motivic weak equivalence,
$h_\cdot(i_0)$ and $h_\cdot(i_1)$ are isomorphic
in the motivic homotopy category and hence \cite[Lemma 3.2.13]{MV}
so are $h_\cdot(i_l) \we E, l=0,1$.
Now the claim follows by 
$\underline{A_l} = f \cdot i_l$.
\end{enumerate}
\end{proof}

\subsection{Definition of the \Mit-action on stable homotopy
groups}
\label{sect-defmact}

{\em From now on, we will make the following standard assumptions:
Let $(\Dst, \we, S^0)$ be a pointed symmetric monoidal model
category. There is a monoidal left Quillen functor
\cite[Def. 4.2.16]{H3} $i: sSet_* \hookrightarrow \Dst$ with right 
adjoint $j: \Dst \rightarrow sSet_*$. 
We choose a cofibrant object $T$ in $\Dst$ such that $- \we T$ 
preserves weak equivalences. Moreover, we assume that $T$ is a cogroup object
in $Ho(\Dst)$. 
(This is the case if e.g. $T \simeq S^1 \we B$ for some 
object $B$ of $\Dst$.)
Finally, we fix a class $\Bst$ of cofibrant objects in
$\Dst$.}

\medskip
For the category $M_\cdot(S)$, we will take $i$ to be the functor
mapping a simplicial set to a constant simplicial presheaf,
and $j$ the evaluation on the terminal object $S \in Sm/S$.
The condition that is $T$ cofibrant is equivalent to
require that the functor $- \we T$ preserves cofibrations,
as then $i(S^0) \we T \cong S^0 \we T \cong T$ is also cofibrant.
The functor $- \we T$ induces a functor on $Ho(\Dst)$.

\begin{definition}
\label{app-motstgroups-def1}
Let E be a  $T$-spectrum in $\Dst$. Then for all $q \in \Zbb, V\in \Bst$,
the abelian
groups (see also  Lemma \ref{app-motstgroups-groupstructure-lem})
$$\colim_{m\geq 0, q+m \geq 1}(\cdots \rightarrow [V \we T^{q+m},E_m] \xrightarrow{\sigma_* (- \we T)} [V \we T^{q+m+1},E_{m+1}] \rightarrow \cdots)
$$
are called the \emph{stable homotopy groups} of $E$, and will be denoted
by $\pi_q^V(E)$. They are functors $Sp(\Dst,T) \rightarrow Ab$.
\end{definition}

\begin{example}
\begin{itemize}
\item
For $\Dst = sSet_*$, $T = S^1$ and $\Bst = \{S^0\}$, one recovers
the definition of the usual 
(naive, that is forgetting the $\Sigma_n$-action) stable homotopy
groups (denoted by $\pihat_*$ in \cite{S6}): $\pi_q^{S^0}(E) \cong \pihat_{q}(E)$.
\item
For $\Dst = M_\cdot^{cm}(S), T = \Pbb^1$ and $\Bst = \{S^r \we h_\cdot(U) \we \Gbb_m^{\we s}| r, s\in \Nbb_0, U \in Sm/S\}$, the groups
$\pi_q^V(E)$ are the motivic stable homotopy
groups of $E$. In particular, $\pi_q^{S^r \we \Gbb_m^{\we s}}(E) \cong \pi^{mot}_{q+r+s,q+s}(E)(U)$ (note that $\Bst$ consists of $cm$-cofibrant objects).
\end{itemize}
\end{example}

\begin{lemma}
\label{app-motstgroups-groupstructure-lem}
Consider two objects $A$ and $X$ in $\Dst$ with $A$ cofibrant, and
$V \in \Bst$.
Then $V \we T^2 \we A$ has an abelian cogroup structure, 
and the corresponding group structure on $[V \we T^2 \we A, X]$ is
compatible with $- \we T$.
\end{lemma}
\begin{proof}
As $T$ is a cogroup object by assumption, $T^2$ and more generally
$A^\prime := V \we T^2 \we A$ is an abelian cogroup object 
with comultiplication $V \we T^2 \we A \xrightarrow{V \we \mu \we T \we A} V \we (T \vee T) \we T \we A \cong [V \we T \we T \we A] \we [V \we T \we T \we A]$. As the comultiplikation on $A^\prime \we T$ is given by $A^\prime \we T \xrightarrow{\mu_{A^\prime} \we T} (A^\prime \vee A^\prime) \we T \cong (A^\prime \we T) \vee (A^\prime \we T)$, the compatibility with $- \we T$ follows.
\end{proof}

\begin{definition}\label{defpiB}
Let $f: E\rightarrow F$ be a map of $T$-spectra in $\Dst$.
Then $f$ is called a $\pi^\Bst$-stable
equivalence if the induced maps
$$\pi_q^V(f): \pi_q^V(E) \rightarrow \pi_q^V(F)$$
are isomorphisms for all $q \in \Zbb, V \in \Bst$.
\end{definition}

\medskip

We now turn to the ${\Mst}$-action.
Let $\Ist$ be the category of finite sets and injective maps,
and $\Mst$ the ``injection monoid'' (see \cite{S1}, \cite{S4}
and Definition \ref{mst-def} below).
Recall (see \cite[section 4.2]{S1}, \cite[section 1.2]{S4}) that there
are functors from symmetric spectra to $Ab$-valued $\Ist$-functors
and from $\Ist$-functors to (tame) $\Mst$-modules,
mapping $X$ to $\underline{X}$ and further to 
$\underline{X}(\omega)$.
 
We still make the above assumptions, and also assume that
$T$ has a sign. The following definition generalizes
\cite[1.2 Construction, Step 1]{S4}.
\begin{propdef}
\label{motsemi-mactionorig-def}
Let $q \in \Zbb$ and $V \in \Bst$. For any symmetric spectrum $X$
in $\Dst$, we define a functor  $\underline{X} : \Ist \rightarrow Ab$ 
for any symmetric $T$-spectrum $X$ in $\Dst$ and then obtain
(see above) an $\Mst$-action on its evaluation
at $\omega$, $\pi_q^V(X)$, which is precisely the group
$\pi_q^V(X)$ of Definition \ref{app-motstgroups-def1}.
In more detail, any $m$ in $\Ist$ is mapped to
$[V \we T^{q+m},X_m]$ (see Lemma \ref{app-motstgroups-groupstructure-lem}) 
if  $q+m \geq 2$, and to $0$ otherwise. 
For $f: m \rightarrow n$ a morphism in $\Ist$
(hence $n\geq m$) we choose a permutation $\gamma \in \Sigma_n$ 
with $f = \gamma_{|m}$. Then $\underline{X}(f)$
is the composition $[V \we T^{q+m}, X_m] \xrightarrow{\sigma^{n-m}_*(-\we T^{n-m})} [V \we T^{q+n}, X_n] \xrightarrow{(V \we |\gamma|_T \we T^{q+n-1})^* \gamma_*} [V \we T^{q+n}, X_n]$ if $q+m \geq 2$, and $0$ otherwise.
\end{propdef}
\begin{proof}
The map $V \we |\gamma|_T \we T^{q+n-1}$ is defined as $V$ and $T$ 
are cofibrant. The above composition is a group homomorphism
as the group structure is compatible with
$- \we T$ (Lemma \ref{app-motstgroups-groupstructure-lem}),
and we have $V \we |\gamma|_T \we T^{q+n-1} = V \we T \we |\gamma|_T \we T^{q+n-2}$ by Definition \ref{vorzeichen-def}.\\
The functor $\underline{X}$ is well-defined
on morphisms: Consider $\gamma, \gamma' \in \Sigma_n$ with 
$\gamma_{|m} = \gamma'_{|m}$. Then there is a $\tau \in \Sigma_{n-m}$ 
with $\gamma'^{-1} \gamma = m + \tau$ and the claim
$\underline{X}(\gamma) = \underline{X}(\gamma')$
is equivalent to showing that the two compositions
$$[V \we T^{q+m}, X_m] \xrightarrow{\sigma^{n-m}_*(-\we T^{n-m})} [V \we T^{q+n}, X_n] \xrightarrow[(V \we |m + \tau|_T \we T^{q+n-1})^*]{(1_m \times \tau)_*} [V \we T^{q+n}, X_n]$$
are equal.
Let $n \geq m$ (otherwise there is nothing to prove).
By Definition \ref{vorzeichen-def}, we have $|m + \tau|_T \we T^{q+n-1} \\= T^{q+m} \we |\tau|_T \we T^{n-m-1} = T^{q+m} \we \tau_T$ in $Ho(\Dst)$.
Applying $V \we -$ and using the equivariance of
$(m + \tau)\cdot \sigma^{n-m} = \sigma^{n-m} \cdot (X_m\we \tau)$,
the equality follows 
from the equality of the following two composizions: \\
\centerline{$[V \we T^{q+m}, X_m] \xrightarrow{(-\we T^{n-m})} [V \we T^{q+m} \we T^{n-m}, X_m \we T^{n-m}] \xrightarrow[(V \we T^{q+m} \we \tau)^*]{(X_m \we \tau)_*} [V \we T^{q+m} \we T^{n-m}, X_m \we T^{n-m}]$}.\\
A straighforward computation involving that
$sgn(\delta \cdot (\gamma + (n'-n))) = sgn(\delta) \cdot sgn(\gamma)$
shows that $\underline{X}$ is indeed a functor.
Finally, as the inclusion $m \rightarrow m+1$ corresponds to
$\sigma_* (-\we T)$, $\underline{X}(\omega)$ is indeed
$\pi_q^V(X)$ as claimed.
\end{proof}

For $\Dst = sSet_*, T =S^1$, this coincides with the definition of \cite{S4},
because $|(-1)_{S^1}|$ is isomorphic to a self-map on $S^1$ 
of degree $-1$.
For $\Dst = M_\cdot^{cm}(S)$ and $T = \Pbb^1$, 
note that being
semistable does not depend on the ${\Abb}^1$-local model structure
(projective, injective, cm...), but only on
the motivic homotopy category $Ho(\Dst)$.

We are now able to state our key definition. 

\begin{definition}(compare \cite[Theorem 4.1]{S4})
\label{def-semistab-new}
Let $\Dst$ be as above and fix a class $\Bst$ of cofibrant objects.
A symmetric $T$-spectrum $X$ is called \emph{semistable}, 
if the \Mit-action (see Definition \ref{motsemi-mactionorig-def}) 
is trivial on all homotopy
groups of $X$ appearing in Definition \ref{defpiB}.
\end{definition}

\begin{remark}
Note that this definition heavily depends on the choice of $\Bst$.
If the $\pi^\Bst$-stable equivalences coincide with the stable
equivalences in $Sp(\Dst,T)$, then under the assumptions
of Theorem \ref{motsemi-theoremorg-th} the two definitions
of semistability coincide. This holds in particular
for $\Dst = M_\cdot^{cm}(S)$ (see above and Proposition \ref{thmain-ex1} 
below), and $\Bst$ as in the example above.
\end{remark}

\begin{lemma}
\label{general-th-appl-1}
Let $f: X \rightarrow Y$ be a $\pi^\Bst$-stable equivalence in 
$Sp^\Sigma(\Dst,T)$. Then $\pi_q^V(f)$ is an isomorphism of $\Mst$-objects. 
In particular: $X$ is semistable if and only if $Y$ is semistable.
\end{lemma}
\begin{proof}
By Definition, the map  $\pi_q^V(f)$ commutes with the  $\Mst$-action
and by assumption the map is an isomorphism.
\end{proof}

\subsection{Some \Mit-isomorphisms between stable homotopy groups}\label{zweivier}
We keep the assumptions of the previous section,
and assume that $T$ has a sign.
Recall \cite{S1}, \cite{S4} the definition of the cycle operator and
of tameness:

\begin{definition}
\label{mst-def}
\begin{itemize}
\item
Let $\Mst$ be the set of all self injections of $\Nbb$.
This is a monoid under composition, the so-called
\emph{injection monoid}.
\item
The injective map  $d: \Nbb \rightarrow \Nbb$
given by $x \mapsto x+1$ is called the \emph{cycle operator}.
\item
As usual, we sometimes
consider $\Mst$ as a category with a single object.
A $\Mst$-object $W$ in $\Dst$ is a functor
$W: \Mst \rightarrow \Dst$, and we have the category
$Func(\Mst,\Dst)$ of $\Mst$-objetcs in $\Dst$.
If $\Dst$ is the category of sets resp- abelian groups,
we call these objects $\Mst$-modules resp. $\Mst$-sets.
\item
Let $n \in \Nbb_0$. The injective map $\Mst \rightarrow \Mst$, 
given by mapping $f$ to the map $x \mapsto \left\{ \begin{array}{ll} x & x\leq n \\ f(x-n) & x > n\end{array} \right.$, is denoted by $n + -$ or $-(n)$.
For $W$ any $\Mst$-object, note that $W(n)$ is the $\Mst$-object 
with underlying object $W$ and the $\Mst$-action
restricted along $n+ -$ .
\item
Now assume further that $\Dst$ has a forgetful
functor to the category of sets. Let $\phi$ be an $\Mst$-action
on an object $W$ in $\Dst$. Then we sometimes write $fx$ for 
$[\phi(f)](x)$ if the $\Mst$-action is understood.
For any $f \in \Mi$ let $|f| := \min\{i\geq 0; f(i+1) \neq i+1\}$. 
An element $x \in W$ \textit{has filtration} $n$ if for all $f \in \Mi$ 
with $|f| \geq n$ we have $fx = x$. We write
$W^{(n)}$ for the subset of all elements
of \emph{filtration} $n$. The $\Mst$-action on $W$ is  \emph{tame}
if $W = \bigcup_{n \geq 0} W^{(n)}$.\\
If $\Dst$ fas a forgetful functor to abelian groups,
then $W^{(n)}, n\geq 0$ are abelian groups as well.
\end{itemize}
\end{definition}

The stable homotopy groups of $sh X, T \we X$ and $\Omega X$ 
may be expressed through the stable
homotopy groups of $X$. The following
generalizes \cite[Examples 3.10 and 3.11]{S4}. 

\begin{proposition}
\label{app-motstgroups-isos-prop}
\label{homcomp}
Let $X$ be a $T$-spectrum in $\Dst$ and $q \in \Zbb, V \in \Bst$. 
Then we have the following isomorphisms of groups.
They are compatible with the sign of $T$, 
and if $X$ is a symmetric spectrum they also respect the $\Mst$-action:
\begin{enumerate}[(i)]
    \item $\pi_{q+1}^V(sh X) \cong \pi_q^V(X)(1)$,
    \item $\pi_q^V(\Omega X) \cong \pi_{q+1}^V(X)$, if $X$ is level-fibrant
and $T$ is cofibrant, and
    \item $\pi_q^V(X) \xrightarrow{T \we -} \pi_{1+q}^V(T \we X)$.
\end{enumerate}
\end{proposition}

\begin{proof}
We first establish the isomorphisms.
\begin{enumerate}[(i)] 
    \item Easy.
    \item As  $X_m$ is fibrant and $V \we T^{q+m}$ is cofibrant, we have isomorphisms:\\
    $\centerline{\xymatrix{[V \we T^{q+m}, \Omega X_m] \xrightarrow{\alpha_{V\we T^{q+m},X_m}} [V\we T^{q+m} \we T,X_m] \xrightarrow{(V \we \chi_{1,q+m})^*} [V \we T \we T^{q+m}, X_m],}}$ compatible with the structure maps, that
is the diagram \\
    $\centerline{\xymatrix{
    [V \we T^{q+m}, \Omega X_m] \ar[d]^{\sigma^{\Omega X}_\ast(- \we T)} \ar[r]^(0.47){\alpha_{V\we T^{q+m},X_m}} & [V\we T^{q+m} \we T,X_m] \ar[r]^{(V \we \chi_{1,q+m})^*} & [V \we T \we T^{q+m}, X_m] \ar[d]^{\sigma^X_\ast(- \we T)} \\
    [V \we T^{q+m+1}, \Omega X_{m+1}] \ar[r]^*!/u3pt/{\labelstyle \alpha_{V\we T^{q+m+1},X_{m+1}}} & [V\we T^{q+m+1} \we T,X_{m+1}] \ar[r]^*!/u3pt/{\labelstyle (V \we \chi_{1,q+m+1})^*} & [V \we T \we T^{q+m+1}, X_{m+1}]
    }}$
    
commutes. Now for any $f: V \we T^{q+m} \rightarrow \Omega X_m$ in $Ho(\Dst)$,
we have\\ $\alpha_{V \we T^{q+m+1}, X_{m+1}}(\sigma^{\Omega X} \cdot (f \we T)) = ev \cdot ([\sigma^{\Omega X} \cdot (f \we T)] \we T) = \sigma^X \cdot (ev_X \we T) \cdot (1 \we \chi_{1,1}) \cdot (f \we T^2) = \sigma^X \cdot (ev_X \we T) \cdot (f \we \chi_{1,1})$. Thus under the lower left composition, $f$ maps to\\ $\sigma^X \cdot (ev_X \we T) \cdot (f \we \chi_{1,1}) \cdot (V \we \chi_{1, q+m+1}) = \sigma^X \cdot (ev_X \we T) \cdot (f \we T^2) \cdot (V \we \chi_{1, q+m} \we T)$, and to $\sigma^X \cdot ([\alpha_{V\we T^{q+m},X_m}(f) \cdot (V \we \chi_{1,q+m})] \we T) = \sigma^X \cdot ([ev \cdot (f \we T) \cdot (V \we \chi_{1,q+m})] \we T)$ under the upper right composition. This yields 
the claimed bijection. Using Lemma \ref{app-motstgroups-groupstructure-lem}
resp. Definition \ref{vorzeichen-def}, we see that $\alpha_{V \we T^{q+m},X_m}$
resp. $(V \we \chi_{1,q+m})^*$ is a group homomorphism.
    \item
    As $T \we -$ preserves weak equivalences in $\Dst$, it induces maps\\
    $\centerline{\xymatrix{[V \we T^{q+m},X_m] \xrightarrow{T\we -} [T \we V \we T^{q+m},T\we X_m] \xrightarrow{(t_{V,T}\we T^{q+m})^*} [V \we T \we T^{q+m},T\we X_m],}}$
    which are obviously compatible with the structure maps.
    For any $f: V \we T^{q+m} \rightarrow X_m$ in $Ho(\Dst)$, the diagram\\
    $\centerline{\xymatrix{
    V \we T \we T^{q+m} \ar[r]^{t_{V,T} \we T^{q+m}} \ar[dr]_{t_{T,T^{q+m}}} & T \we V \we T^{q+m} \ar[r]^{T \we f} & T \we X_m \\
    & V \we T^{q+m} \we T \ar[r]^{f \we T} \ar[u]_{t_{V \we T^{q+m},T}} & X_m \we T \ar[u]^{t_{T,X_m}}
    }}$\\
commutes, therefore the map above equals the composition\\
    $\centerline{\xymatrix{
    [V \we T^{q+m},X_m] \xrightarrow{-\we T} [V \we T^{q+m} \we T, X_m \we T] \xrightarrow{(V \we \chi_{1,q+m})^\ast t_{T,X_m\ast}} [V \we T \we T^{q+m},T \we X_m]
    }}$\\
Arguing as in (ii), we see that this is a group homomorphism. Passing to the 
colimit yields the desired map $T \we (-)= (T \we (-))_X$. By naturality,
any level equivalence $X^c \rightarrow X$ in $Sp(\Dst,T)$ induces
an isomorphism between the maps  $(T \we -)_X$ and $(T \we -)_{X^c}$.
Choosing $X^c$ to be level cofibrant, we may assume that $X$ is level 
cofibrant itself when showing that $(T \we -)_X$ is an isomorphism.

To see injectivity, assume that there is some $f$ in the kernel,
and that $f$ is represented by some element in $[V \we T^{q+m},X_m]$.
Then contemplating the commutative diagram \\
    $\centerline{\xymatrix{
    [V \we T^{q+m},X_m] \ar[r]^{-\we T} \ar[d]_{T \we -} & [V \we T^{q+m} \we T,X_m \we T] \ar[r]^{\sigma_*} & [V \we T^{q+m} \we T, X_{m+1}] \\
     [T\we V \we T^{q+m},T \we X_m] \ar[ur]_(.65)*!/r14pt/{\labelstyle t_{V\we T^{q+m},T}^* \cdot t_{T,X_m*}} \ar[r]_*!/d4pt/{\labelstyle (t_{V,T}\we T^{q+m})^*} & [V \we T \we T^{q+m}, T \we X_m] \ar[r]_*!/d4pt/{\labelstyle (V \we \chi_{q+m,1})^* \cdot t_{T,X_m*}} & [V \we T^{q+m} \we T, X_m \we T] \ar[u]^{\sigma_*} \ar@{=}[ul]
    }}$
we see that it has to be zero in the upper right corner,
showing injectivity as claimed.

To obtain inverse images, consider the composition \\
    $\centerline{\xymatrix{[V \we T^{1+q+m}, T \we X_m] \xrightarrow{(V \we \chi_{q+m,1})^* \cdot t_{T,X_m*}} [V \we T^{(q+m)+1} ,X_m \we T] \xrightarrow{\sigma_*} [V \we T^{q+m+1} ,X_{m+1}].}}$ 
It remains to show that composing this with the map above
equals $\sigma^{T \we X}_* (- \we T)$. This will rely on the existence
of the sign on $T$. Let $f: V \we T^{1+q+m} \rightarrow T \we X_m$ be a map
in $Ho(\Dst)$. Then we have\\
    $[(t_{V,T}\we 1)^*\cdot (T\we -)] \cdot [\sigma_* \cdot t_{T,X_m*}\cdot (V\we \chi_{q+m,1})^*](f) = [(t_{V,T}\we 1)^*\cdot (T\we -)](\sigma \cdot t_{T,X_m} \cdot f \cdot (V\we \chi_{q+m,1})) = T \we (\sigma \cdot t_{T,X_m} \cdot f \cdot (V\we \chi_{q+m,1})) \cdot (t_{V,T}\we T^{q+m+1}) = \sigma^{T\we X} \cdot (T \we t_{T,X_m}) \cdot (T\we f) \cdot (t_{V,T}\we \chi_{q+m,1})$\\
Let us first consider
    $(T \we t_{T,X_m}) \cdot (T\we f) = (T \we t_{T,X_m}) \cdot (t_{T,T}\we X_m)^2 \cdot (T\we f) \\= [(T \we t_{T,X_m}) \cdot (t_{T,T}\we X_m)] \cdot [((-1)_T\we T \we X_m) \cdot (T\we f)] = t_{T,T\we X_m}  \cdot (T\we f) \cdot ((-1)_T \we V \we T^{1+q+m}) = (f \we T) \cdot t_{T,V \we T^{1+q+m}}  \cdot ((-1)_T \we V \we T^{1+q+m}).$\\
Because 
    $t_{T,V\we T^{1+q+m}} \cdot ((-1)_T\we V \we T^{1+q+m}) \cdot (t_{V,T}\we \chi_{q+m,1})
    \\= t_{T,V\we T^{1+q+m}} \cdot (t_{V,T}\we \chi_{q+m,1}) \cdot (V \we (-1)_T \we T^{1+q+m})
    \\= (V \we \tau_{1, 1+q+m+1}) \cdot (V \we (-1)_T \we T^{1+q+m})
    = (V \we (-1)_T \we T^{1+q+m}) \cdot (V \we (-1)_T \we T^{1+q+m})
    = 1$\\
    we finally obtain $[(t_{V,T}\we 1)^*\cdot (T\we -)] \cdot [\sigma_* \cdot t_{T,X_m*}\cdot (V\we \chi_{q+m,1})^*](f) = \sigma^{T \we X} \cdot (f \we T)$. 
Here $\tau_{1, 1+q+m+1} \in \Sigma_{1+q+m+1}$ is the permutation interchanging 
$1$ and $1+q+m+1$.
\end{enumerate}

We now turn to the \Mit-action. Let $f: \Nbb \rightarrow \Nbb$ be injective, $\max(f(m)) = n$ and $\gamma \in \Sigma_n$ mit $\gamma_{|m} = f|_m^n$. \\
Concerning (i), for $1+\gamma \in \Sigma_{1+m}$ we have $(1+\gamma)_{|1+n} = (1+f)|_{1+m}^{1+n}$ and the diagram\\
$\centerline{\xymatrix{
[V \we T^{(q+1)+m}, (sh X)_m] \ar@{=}[r] \ar[d]^{\sigma^{n-m}_* \cdot (-\we T^{n-m})}     &         [V \we T^{q+(1+m)}, X_{1+m}] \ar[d]^{\sigma^{(n+1)-(m+1)}_* \cdot (-\we T^{(n+1)-(m+1)})} \\
[V \we T^{(q+1)+n}, (sh X)_n] \ar@{=}[r] \ar[d]^{(V \we |\gamma|_T\we 1)^* \cdot \gamma_*}     &         [V \we T^{q+(1+n)}, X_{1+n}] \ar[d]^{(V \we |1+\gamma|_T \we 1)^* \cdot (1+\gamma)_*} \\
[V \we T^{(q+1)+n}, (sh X)_n] \ar@{=}[r]         &         [V \we T^{q+(1+n)}, X_{1+n}]
}}$
commutes as $sgn(\gamma) = sgn(1 + \gamma)$. But the right hand side is
precisely the \Mit-action on $\pi_{q}^V(X)(1)$.\\

As the maps in (ii) and (iii) commute levelwise with $\sigma^{n-m}_* \cdot (-\we T^{n-m})$, it remains to show that they also commute with maps of the form
 $(V \we |\gamma|_T\we 1)^* \cdot \gamma_*$. For (ii), consider the diagram \\
$\centerline{\xymatrix{
[V \we T^{q+m}, \Omega X_n] \ar[d]^{(1\we |\gamma|_T \we 1)^* \cdot(\Omega \gamma)_*} \ar[r]^(0.43)*!/u0pt/{\labelstyle -\we T}         &         [V \we T^{q+m}\we T, \Omega X_n \we T] \ar[d]^{(1\we |\gamma|_T \we 1\we T)^* \cdot(\Omega \gamma \we T)_*} \ar[r]^(0.55)*!/u3pt/{\labelstyle ev}         &         [V \we T^{q+m}\we T, X_n] \ar[d]^{(1\we|\gamma|_T \we 1 \we T)^* \cdot \gamma_*} \ar[r]^(0.53)*!/u3pt/{\labelstyle (V\we \chi_{1,q+m})^*}         &         [V \we T\we T^{q+m}, X_n] \ar[d]^{(1\we T\we |\gamma|_T \we 1)^* \cdot \gamma_*} \\
[V \we T^{q+m}, \Omega X_n] \ar[r]_(0.43)*!/d1pt/{\labelstyle -\we T}         &         [V \we T^{q+m}\we T, \Omega X_n \we T] \ar[r]_(0.55)*!/d3pt/{\labelstyle ev}         &         [V \we T^{q+m}\we T, X_n] \ar[r]_(0.53)*!/d3pt/{\labelstyle (V\we \chi_{1,q+m})^*}         &         [V \we T\we T^{q+m}, X_n] \\
}}$
which commutes by the naturality of $ev$ and $t_{-,-}$).
In the last column, we have $1\we T\we |\gamma|_T \we 1 = 1\we |\gamma|_T \we 1$ by Definition \ref{vorzeichen-def}. As $\alpha = ev \cdot (-\we T)$ 
the compatibility with the \Mit-action follows. For (iii),
we consider the commutative diagram\\
$\centerline{\xymatrix{
[V \we T^{q+m}, X_n] \ar[d]^{(1\we |\gamma|_T \we 1)^* \cdot \gamma_*} \ar[r]^(0.42)*!/u2pt/{\labelstyle (T\we -)}         &         [T \we V \we T^{q+m}, T\we X_n] \ar[d]^{(T\we V\we |\gamma|_T \we 1)^* \cdot (T \we \gamma)_*} \ar[r]^(.51)*!/u2pt/{\labelstyle (t_{V,T}\we 1)^*}         &         [V \we T \we T^{q+m}, T\we X_n] \ar[d]^{(V\we T\we |\gamma|_T \we 1)^* \cdot  (T \we \gamma)_*} \\
[V \we T^{q+m}, X_n] \ar[r]_(0.42)*!/d3pt/{\labelstyle (T\we -)}         &         [T \we V \we T^{q+m}, T\we X_n] \ar[r]_(0.51)*!/d3pt/{\labelstyle (t_{V,T}\we 1)^*}         &         [V \we T \we T^{q+m}, T\we X_n]
}}$
Here for the last column we have $V\we T\we |\gamma|_T \we T^{q+m-1} = V\we |\gamma|_T \we T^{q+m}$ by Definition \ref{vorzeichen-def},
hence the third isomorphism also respects the \Mit-action.    

The compatibility with the sign is shown by a similar argument.
\end{proof}

The Proposition implies that the class of semistable spectra
is stable under various operations (compare \cite[section 4]{S4}, \cite{S6}):
\begin{corollary}
\label{app-motstgroups-isos-cor}
Assume that  $Sp^\Sigma(\Dst,T)$ has a levelwise fibrant replacement
functor. Then for any symmetric $T$-spectrum in $\Dst$, the following
are equivalent:
\begin{itemize}
    \item $X$ is semistable
    \item $T \we X$ is semistable
    \item $\Omega J^\Sigma X$ is semistable
    \item sh $X$ is semistable
\end{itemize}
\end{corollary}

\begin{proof}
Most of this follows directly from Proposition \ref{app-motstgroups-isos-prop}. Concerning $sh X$ it remains to show that for a tame  $\Mst$-modul $W$ 
the $\Mst$-action is trivial if and only if it is trivial on $W(1)$.
But if the $\Mst$-action is trivial on $W(1)$, then $W$ has filtration
$\leq 1$ and thus by Lemma \ref{mit-lem-tame} below the $\Mst$-action 
is trivial.
\end{proof}

\begin{definition}
Let $X$ be a levelwise fibrant symmetric $T$-spectrum. We denote the
composition of the $\Mst$-isomorphisms $(i)$ and $(ii)$ of Proposition
\ref{homcomp} by $\alpha: \pi_q^V(RX) \cong \pi_q^V(X)(1)$.
\end{definition}

The following will be used when proving Theorem \ref{motsemi-theoremorg-th}:
\begin{proposition}
\label{app-motstgroups-isos-comp-prop}
\label{semistcomp}
Let $X$ be a symmetric $T$-spectrum. The action of
$d \in \mathcal{M}$ is isomorphic to the action of
$\lambda_X$ on stable homotopy groups, i.e.
the square\\
$\centerline{\xymatrix{
\pi_q^V(X) \ar[r]^{d_*} \ar[d]_{(-1)_T^q \cdot (T\we -)}^{\cong} & \pi_q^V(X)(1) \ar[d]^{\cong} \\
\pi_{1+q}^V(T \we X) \ar[r]^*!/u2pt/{\labelstyle \pi_{1+q}^V(\lambda_X)} & \pi_{1+q}^V(sh X)
}}$\\
commutes. If $X$ is levelwise fibrant,
the for all $n \in \Nbb_0$ the squares\\
$\centerline{\xymatrix{
\pi_q^V(X) \ar[r]^{d_*} \ar[d]_{\lamt_{X\,\ast} \cdot (-1)_T^q}^{\cong} & \pi_q^V(X)(1) \ar@{=}[d] & \textit{and} & \pi_q^V(R^n X) \ar[r]^{\alpha^n} \ar[d]_{\pi_q^V(\Omega^n \lamt_{sh^n X}) \cdot (-1)_T^q} & \pi_q^V(X)(n) \ar[d]^{d_*} \\
\pi_q^V(RX) \ar[r]^{\alpha} & \pi_q^V(X)(1) & & \pi_q^V(R^{n+1} X) \ar[r]^{\alpha^{n+1}} & [\pi_q^V(X)(n)](1)
}}$\\
commute as well, the right $d_*$ is the action of $d(n)$
on the underlying sets $\pi_q^V(X)$ (see Definition \ref{mst-def}).
In particular, the action of $d(n)$ on $\pi^V_q(X)$
is isomorphic to the map $\pi^V_q(R^n \lamt_X)$.
\end{proposition}

\begin{proof}
Let $f: V \we T^{q+m} \rightarrow X_m$ be a morphism in $Ho(\Dst)$.
The first square commutes because
\par
\begingroup
\leftskip=3cm 
\noindent
$[\lambda_{X,m*} \cdot(V \we (-1)_T^q \we 1)^* \cdot [(t_{V,T} \we T^{q+m})^* \cdot(T \we -)]](f)\\
= (\chi_{m,1} \cdot \sigma_m^X \cdot t_{T,X_m}) \cdot(T \we f) \cdot (t_{V,T} \we T^{q+m})\cdot(V \we (-1)_T^q \we 1)\\
= \chi_{m,1} \cdot \sigma_m^X  \cdot(f \we T) \cdot t_{T,V \we T^{q+m}} \cdot (t_{V,T} \we T^{q+m})\cdot(V \we (-1)_T^q \we 1)\\
= \chi_{m,1} \cdot \sigma_m^X  \cdot(f \we T) \cdot (V \we \chi_{1,q+m}) \cdot(V \we (-1)_T^q \we 1)\\
= \chi_{m,1} \cdot \sigma_m^X  \cdot(f \we T) \cdot (V \we (-1)_T^{q+m} \we 1) \cdot(V \we (-1)_T^q \we 1)\\
= \chi_{m,1} \cdot \sigma_m^X  \cdot(f \we T) \cdot (V \we (-1)_T^m \we 1)\\
= [\chi_{m,1*} \cdot (V \we (-1)_T^m \we 1)^* \cdot \sigma^X_* \cdot(- \we T)](f) = d_*(f).$\\
\par
\endgroup
and similarly for the second square
\par
\begingroup
\leftskip=3cm 
\noindent
$[[(V \we \chi_{1,q+m})^* \cdot \alpha] \cdot \lamt_{X,m\,\ast} \cdot (V\we (-1)_T^q \we 1)^*](f)\\
= ev \cdot ([\lamt_{X,m} \cdot f \cdot (V\we (-1)_T^q \we 1)] \we T) \cdot (V \we \chi_{1,q+m})\\
= \chi_{m,1} \cdot \sigma^X_m  \cdot(f \we T) \cdot (V \we \chi_{1,q+m}) \cdot(V \we (-1)_T^q \we 1) = d_*(f).$\\
\par
\endgroup
Finally, following Schwede we observe that the commutativity of 
the third square follows from the second. To see this, consider
the large commutative (note that the isomorphisms are compatible with the sign
by Proposition \ref{homcomp})
diagram \\
$\centerline{\xymatrix{
\ar[rr]^{\cong} \ar@/^1.5pc/[rrrr]^{\alpha^n} \pi_q^V(\Omega^n sh^n X) \ar[d]^{(\Omega^n \lamt_{sh^n X})_\ast \cdot (-1)_T^q} & & \pi_{q+n}^V(sh^n X) \ar[dl]^(0.37)*!/r10pt/{\labelstyle (\lamt_{sh^n X})_* \cdot (-1)_T^q} \ar[dr]^{d_\ast} \ar[rr]^{\cong} & & \pi_q^V(X) \ar[d]^{d(n)}\\
\ar[r]^{\cong} \ar@/_1.5pc/[rrrr]_{\alpha^{n+1}} \pi_q^V(\Omega^n sh^n RX) & \ar[rr]^{\alpha}_{\cong} \pi_{q+n}^V(R sh^n X) & & \pi_{q+n}^V(sh^n X) \ar[r]^(0.55){\cong} & \pi_q^V(X)
}}$
The last claim follows from Lemma \ref{lamtsh} by which
the  morphisms $\lamt_{sh^n X}$ and $sh^n \lamt_X$ are isomorphic
in $Sp^\Sigma(\Dst,T)$.
\end{proof}


\subsection{Generalities concerning the $\Mst$-action}

\begin{lemma}(Schwede)
\label{mit-lem-tame}
Let $W$ be a tame \Mit-module.
\begin{enumerate}[(i)]
\item
Any element of $\Mst$ acts injectively on $W$.
\item
If the filtration on $W$ is bounded,
then $W$ is a trivial \Mit-module.
\item
If $d\in\Mst$ acts surjectively on $W$, then $W$ is a trivial
$\Mst$-module.
\item
If $W$ is a finitely generated abelian group, 
then $W$ is a trivial $\Mst$-module.
\end{enumerate}
\end{lemma}
\begin{proof}
See \cite[Lemma 2.3]{S4}.
\end{proof}

\begin{lemma}
\label{mact-ifuncsurj}
Let $F: \Ist \rightarrow \Dst$ be a functor and assume
that $\Dst$
has sequential colimits and a forgetful functor
to the category of sets. Then, if any element of $F(\omega)$ 
is in the image of some inclusion map $incl_m^{F(\omega)}$, $F(\omega)$ 
is tame. 
\end{lemma}
\begin{proof}
It suffices to show that any $x \in F(\omega)$ arising via 
$y \in F(\textbf{l})$ $x = incl_l^{F(\omega)}(y)$ has filtration
$\leq l$. Consider $f \in \Mst$ with $|f| = l$. By definition
of $F(f)$, we have $F(f) \cdot incl_l^{F(\omega)} = F(f_{|\textbf{l}}) = incl_l^{F(\omega)} \cdot F(1_\textbf{l}) = incl_l^{F(\omega)}$ as $f$ 
restrics to $1_\textbf{l}$. This yields  $F(f)(x) = x$, so $x$ 
has filtration $\leq l$.
\end{proof}

The next result describes several general properties
of the construction which \cite{S6} applies to the functors
$\pihat_k$.
\begin{propdef}
\label{truehogroups}
Let $\Dst$ be a category and $\Fst$ a class of functors from $\Dst$ 
to the category of $\Mst$-sets. Let $\Cst = \Dst_\Fst$ be the full subcategory
of $\Dst$ of those $X$ for which the $\Mst$-action on $F(X)$ is trivial
for all $f \in \Fst$.
\begin{enumerate}[(i)]
\item
For any $X \in \Dst, F \in \Fst$, consider the set $\tilde{F}(X)$ 
of natural transformations of functors $\Cst \rightarrow Set$ 
from $\Dst(X,-)$ to $F$. Then $\tilde{F}$ is a functor from
$\Dst$ to $\Mst$-sets.
\item
$\Mst$ acts trivially on $\tilde{F}(X)$.
\item
There is a natural map $c_X: F(X) \rightarrow \tilde{F}(X)$ of
$\Mst$-sets.
\item
An object $X$ of $\Dst$ is in $\Cst$ if and only if $c_X: F(X) \rightarrow \tilde{F}(X)$ is bijective (or equivalently injective).
\end{enumerate}
\end{propdef}
\begin{proof}
\begin{enumerate}[(i)]
\item
Let $f: X_1 \rightarrow X_2$ be a map in $\Dst$, $g \in \tilde{F}(X_1)$ and $k: X_2 \rightarrow Y$ a map in $\Dst$ with $Y$ in $\Cst$. The natural transformation $g$ maps $kf$ to an element $g^\prime_Y(k) := g_Y(kf) \in F(Y)$.
By naturality of $g$ the assignment $k \mapsto g^\prime_Y(k)$ is natural in
$Y$. Hence we obtain a map $\tilde{F}(f): \tilde{F}(X_1) \rightarrow \tilde{F}(X_2), g \mapsto g^\prime$, and one easily verifies that
$\tilde{F}$ is a functor.
Now let $w \in \Mst$ and $g \in \tilde{F}(X_1)$. Then the composition $F_{|\Cst}(w) \cdot g \in \tilde{F}(X_1)$ is a natural transformation, thus defining
an $\Mst$-action on $\tilde{F}(X_1)$. For any $f: X_1 \rightarrow X_2$ in 
$\Dst$ we then have $[w_* \tilde{F}(f)](g) = F_{|\Cst}(w) \cdot g(-\cdot f) = [F_{|\Cst}(w) \cdot g](-\cdot f) = [\tilde{F}(f) w_*](g)$. 
Therefore $\tilde{F}(f)$ respects the $\Mst$-action.
\item Let $w \in \Mst, g \in \tilde{F}(X)$ and $k: X \rightarrow Y$ with
$Y$ in $\Mst$. Then $\Mst$ acts trivially on $F(Y)$ and we have
$[(F_{|\Cst}(w) \cdot g)(k)] = w_\ast [g(k)] = g(k)$, so
$\Mst$ acts trivially on $\tilde{F}(X)$ as well.
\item
The map $c_X$ sents $x \in F(X)$ to the natural transformation $k \mapsto [F(k)](x)$ ($k: X \rightarrow Y$, $Y$ in $\Cst$), which is natural in $X$.
For $w \in \Mst$ we have $[w_\ast(c_X(x))](k) = [F_{|\Cst}(w) \cdot c_X(x)](k) = F_{|\Cst}(w)([F(k)](x)) = [F(k)](w_\ast(x)) = [c_X(w_\ast(x))](k)$ 
as $F$ is compatible with $\Mst$. Hence $c_X$ is a map of $\Mst$-sets.
\item
Now let $X$ in $\Cst$. By Yoneda $ev_{1_X}: \tilde{F}(X) \rightarrow F(X)$ an $1_X$ is bijective with inverse $c_X$ ($ev_{1_X} \cdot c_X = 1_{F(X)}$).
Conversely, if $c_X$ is injective, then by $(ii)$ and $(iii)$ the action 
of $\Mst$ on $F(X)$ trivial, hence $X$ is in $\Bst$.
\end{enumerate}
\end{proof}

One can show that for $\Dst$ the category of symmetric spectra
based on simplicial sets and $\Fst$ the set of stable homotopy groups
$\pihat_k, k \in \Zbb$ the above definition of $\widetilde{\pihat_k}$
is isomorphic to the definition of the ``true'' stable homotopy groups.
Later we will also need the follwing standard result.

\begin{corollary}
\label{app-motstgroups-Rinfty-cor}
\label{corRinfty}
Assume that fibrant objects $\Dst$ are closed under
sequential colimits, and the functors
$j$, $Hom(T,-)$ and $Hom(A,-)$ for all $A \in \Bst$ 
preserve sequential colimits.
Then for any sequential diagram $X^\bullet$ in $Sp(\Dst,T)$ the map
$$ \colim_{n\geq 0} \pi_q^V(X^n) \xrightarrow{incl_*} 
\pi_q^V(colim X^\bullet)$$
is an isomorphism for all $q\in \Zbb, V \in \Bst$.
\end{corollary}

\subsection{Criterions for semistability: the generalized theorem}

We keep the hypotheses of section \ref{sect-defmact}.
We now extend Theorem \ref{thmin} (under additional 
assumptions), which simultaneously generalizes
Theorem \ref{th-problem} of Schwede.

\begin{theorem}
\label{motsemi-theoremorg-th}
Let $(\Dst,\we,S^0)$ be a pointed symmetric monoidal model category with a cofibrant object $T$, such that $- \we T$ preserves weak equivalences
and $T$ has a sign. Let $i: sSet_* \rightarrow \Dst$ be a monoidal left
Quillen functor with adjoint $j$. Let $\Bst$ be a class
of cofibrant objects in $\Dst$. Moreover, assume that
fibrant objects in $\Dst$ are closed under sequential colimits 
and that $j$, $Hom(T,-)$ and $Hom(A,-)$ for all $A \in \Bst$ 
preserve sequential colimits. Then for any
$T$-spectrum $X$ in $\Dst$ the following are equivalent:
\begin{enumerate}[(i)]
\item[(i)] $X$ is semistable (see Definition \ref{def-semistab-new}).
\item[(ii)] The cycle operator $d$ (see Definition \ref{mst-def}) 
acts surjectively on all stable homotopy groups.
\item[(iii)] The map $\lambda_X: T\we X \rightarrow sh X$ is a $\pi^\Bst$-
stable equivalence.
\end{enumerate}
If the class $\{\pi_q^V; q \in \Zbb, V \in \Bst \}$ of functors from
$Sp^\Sigma(\Dst,T)$ to $\Mst$-sets satisfies the assumptions of Proposition-Definition 
\ref{truehogroups}, then $(i)$ is also equivalent to\\
$(i^\prime)$: The map $c_X: \pi_q^V(X) \rightarrow \widetilde{\pi_q^V}(X)$ (Definition \ref{truehogroups}) is a bijection for all $q \in \Zbb, V \in \Bst$.\\

If $X$ is level fibrant, then  $(i)-(iii)$ are also equivalent to
\begin{enumerate}
\item[(iv)] The map $\lamt_X: X \rightarrow RX$ is a $\pi^\Bst$-stable 
equivalence.
\item[(v)] The map $\lamt_X^\infty: X \rightarrow R^\infty X$ is a
$\pi^\Bst$-stable equivalence.
\item[(vi)] The symmetric spectrum $R^\infty X$ is semistable.
\end{enumerate}

\noindent If moreover the following holds
\begin{itemize}
\item
the projective level model structure on $Sp(\Dst,T)$ exists
and the conditions $(a)$ and $(b)$ of Theorem \ref{thmin} are satisfied,
\item the projektive level model structure on $Sp^\Sigma(\Dst,T)$ exists 
(in particular there is a level fibrant replacement functor
$1 \rightarrow J^\Sigma$, and
\item $\pi^\Bst$-stable equivalences coincide with
stable equivalences in $Sp(\Dst,T)$,
\end{itemize}
then the above conditions $(i)-(iii)$ are equivalent
to $(viii)$, and if $X$ is also level fibrant all above
conditions are equivalent to $(vii)$:
\begin{enumerate}
\item[(vii)] The symmetric spectrum $R^\infty X$ is an $\Omega$-spectrum.
\item[(viii)] There is a $\pi^\Bst$-stable equivalence
$X$ to an $\Omega$-spectrum.
\end{enumerate}
In any case, we always have the implications
$(viii) \Rightarrow (i)$ and $(vii) \Rightarrow (vi)$.
\end{theorem}
\begin{proof}
\begin{itemize}
\item
$(i) \Leftrightarrow (ii)$
By definition  $(ii)$ follows from $(i)$. Because of tameness (see 
Lemma \ref{mact-ifuncsurj} $(i)$), Lemma \ref{mit-lem-tame} $(iii)$ 
shows the converse.
\item
$(ii) \Leftrightarrow (iii)$
This follows from the first commutative diagram in
Proposition \ref{semistcomp}.
\item
$(i) \Leftrightarrow (i^\prime)$ follows from Proposition \ref{truehogroups} 
and Definition \ref{def-semistab-new}.
\item
$(viii) \Rightarrow (ii)$
For any $\Omega$-spectrum $Z$, $\lamt_Z$ is a level equivalence
and hence a $\pi^\Bst$-stable equivalence. By $(iv) \Rightarrow (ii) \Rightarrow (i)$ it follows that $\Omega$-spectra are semistable.
Lemma \ref{general-th-appl-1} then shows that $X$ is semistable.
\item
$(vii) \Rightarrow(vi)$: We saw in $(viii) \Rightarrow (ii)$
that $\Omega$-spectra are semistable.
\end{itemize}

Now assume that $X$ is level fibrant.
\begin{itemize}
\item
$(ii) \Leftrightarrow (iv)$
by the second commutative square in Proposition \ref{semistcomp},
$(iv)$ equivalent to $d$ acting bijectively on all $\pi^\Bst$-stable 
homotopy groups of $X$. Now use $(i) \Leftrightarrow (ii)$.
\item
$(iv)\Rightarrow(v)$:
As $\lamt_X$ is a $\pi^\Bst$-stable equivalence,
so are $R^n \lamt_X, n \in \Nbb_0$ as $\Omega$ 
and $sh$ preserve $\pi^\Bst$-stable equivalences
according to Proposition \ref{homcomp} $(i), (ii)$,
By Corollary \ref{corRinfty}, the map $\pi_q^V(\lamt_X^\infty)$ is isomorphic
to the inclusion $\pi_q^V(X) \xrightarrow{incl_0} \colim_{n \geq 0} \pi_q^V(R^n X)$. But all the maps $\pi_q^V(R^n \lamt_X), n \in \Nbb_0$ 
are isomorphisms, hence so is the inclusion and
thus $\lamt_X^\infty$ is a $\pi^\Bst$-stable equivalence.
\item
$(v)\Rightarrow (ii)$: The maps $\pi_q^V(R^n \lamt_X), n \in \Nbb_0$ are 
injective, because by Proposition \ref{semistcomp} they are isomorphic to the 
action of $d(n)$ on $\pi_q^V(X)$, which again by Lemma \ref{mit-lem-tame} and 
\ref{mact-ifuncsurj} is injective. The inclusion $\pi_q^V(X) \xrightarrow{incl} \colim_{n \geq 0} \pi_q^V(R^n X)$ is an isomorphism, as it is
isomorphic to $\pi_q^V(\lamt_X^\infty)$ (Corollary \ref{corRinfty}). 
As all maps in the sequential diagram
$\pi_q^V(R^\bullet X)$ are injective, they must be surjective.
Hence $d$ acts surjectively on $\pi_q^V(X)$.
\item
$(iv) \Rightarrow (vi)$: As $(iv)$ implies $(v)$ and $(ii)$, hence
also $(i)$, Lemma \ref{general-th-appl-1} shows that
$R^\infty X$ is semistable.
\item
$(vi) \Rightarrow (i)$: We saw above ($(v) \Rightarrow (ii)$) that
$\pi_q^V(\lamt_X^\infty): \pi_q^V(X) \rightarrow \pi_q^V(R^\infty X)$ 
is injective and compatible with the \Mit-action.
As the $\Mst$-action on $\pi_q^V(R^\infty X)$ is 
trivial, so is its restriction to $\pi_q^V(X)$.
\end{itemize}

Finally, we assume that the hypotheses in the last part of the theorem
are satisfied.
\begin{itemize}
\item
$(iv)\Rightarrow(vii)$: By hypothesis $\lamt_X$ is a stable equivalence in 
$Sp(\Dst,T)$. The implication $(ii) \Rightarrow (iv)$ in Theorem \ref{thmin} 
then yields the claim.
\item
$(i) \Rightarrow (viii)$
We have a $\pi^\Bst$-stable equivalence $X \rightarrow J^\Sigma X =: Y$ in $Sp^\Sigma(\Dst,T)$. Using Lemma \ref{general-th-appl-1} we see that
is $J^\Sigma X$ semistable, so the implikations $(i) \Rightarrow (v),(vii)$
show that $\lamt^\infty_Y: Y \rightarrow R^\infty Y$ is a $\pi^\Bst$-stable 
equivalence and $R^\infty Y$ an $\Omega$-spectrum.
\end{itemize}
\end{proof}

\begin{example}
\label{general-susp-semist}
For suspension spectra $\Sigma^\infty L$ the map
$\lambda_{\Sigma^\infty L}$ is already levelwise an 
isomorphism, as the structure maps
$\sigma$ are identities. Hence
suspension spectra are semistable.
\end{example}

The above Theorem \ref{motsemi-theoremorg-th} is designed
to apply notably to the motivic model category
$M_\cdot^{cm}(S)$:
\begin{proposition}
\label{thmain-ex1}
All assumptions (except for those preceding
$(i^\prime)$) of Theorem \ref{motsemi-theoremorg-th} are satisfied
for $\Dst = M^{cm}_\cdot(S), T = \Pbb^1$, $\Bst = \{S^r \we \Gbb_m^s 
\we U_+| r,s\geq 0, U \in Sm/S\}$.
\end{proposition}
\begin{proof}
Most of this has been proved in Corollary \ref{propmotivic} already.
Subsection \ref{subsect-sign-proj} shows that $\Pbb^1$ has a sign, 
and the projective level
model structure on $Sp^{\Sigma}(\Dst,T)$ is established in 
\cite[Theorem 8.2]{H1}.
The $\pi^\Bst$-equivalences coincide with
the stable equivalences in  $Sp(\Dst,T)$
by \cite[section 3.2]{J}.
\end{proof}

Sometimes sequential colimits preserve semistability:
\begin{proposition}
Let $X^\bullet$ be a sequential diagram in $Sp^\Sigma(\Dst,T)$
and assume that the hypotheses of Corollary \ref{corRinfty} hold. 
If all $X^n, n\in \Nbb_0$ are semistable, then so is $colim X^\bullet$.
\end{proposition}
\begin{proof}
Following Corollary \ref{corRinfty}, we have an isomorphism
$colim \pi_q^V(X^\bullet) \rightarrow \pi_q^V(colim X^\bullet)$. 
Now the maps $\pi_q^V(X^n) \xrightarrow{incl_*} \pi_q^V(colim X^\bullet)$ 
respect the $\Mst$-action and the sets $\pi_q^V(X^n),$ $n\in \Nbb_0$ have
trivial $\Mst$-action. As colimits preserve identities,
$\Mst$ acts trivially on $\pi_q^V(colim X^\bullet)$ as well.
\end{proof}

For $\Dst = Top_*, T = S^1$ a special class of
semistable spectra is given by orthogonal spectra
(see \cite[Example 3.2]{S4}.) These include
not only suspension spectra, but also various
Thom spectra. This is related to the
following criteria:
\begin{proposition}\label{precedingremark}
\label{motsemi-mactioninduced-appl-prop}
A symmetric spectrum $X$ is semistable if one of the following conditions hold:
\begin{enumerate}
\item For any $q \in \Zbb$ and $V \in \Bst$ there is an $l \geq 0$
such that the inclusion map $[V \we T^{q+l}, X_l] \rightarrow \pi^V_q(X)$ 
is surjective. This holds, in particular, if the stable homotopy groups
stabilize, i.e. $[V \we T^{q+n}, X_n] \rightarrow [V \we T^{q+n+1}, X_{n+1}]$ 
is an isomorphism for $n \gg 0$.
\item Even permutations on $X_l$ induce identities in $Ho(\Dst)$.
\item The stable homotopy groups $\pi^V_q(X)$ are finitely generated abelian
groups for all $q \in \Zbb$ and $V \in \Bst$.
\end{enumerate}
\end{proposition}
\begin{proof}
\begin{enumerate}
\item According to Lemma \ref{mact-ifuncsurj} the filtration on
$\pi^V_q(X)$ is bounded, hence by Lemma \ref{mit-lem-tame} $(ii)$ 
the $\Mst$-action on $\pi^V_q(X)$ is trivial.
\item
We show that $d$ acts trivially on $\pi^V_q(X)$.
The following observation is crucial:
For any even $n \in \Nbb_0$ the map  $$[V \we T^{q+n+1},T_{n+1}] 
\xrightarrow{\chi_{n,1*} (V \we |\chi_{n,1}|_T \we 1)^*} 
[V \we T^{q+n+1},T_{n+1}]$$ is the identity. This is because
$\chi_{n,1}$ is even, hence $|\chi_{n,1}|_T = 1$ (Definition 
\ref{vorzeichen-def}), and $\chi_{n,1*}$ is the identity by assumption.
Any element in $\pi^V_q(X)$ is (stably) represented
by some $f \in [V \we T^{q+n},T_n]$ with $n\in \Nbb_0$ even.
Therefore $d[f] = [\chi_{n,1*} (V \we |\chi_{n,1}|_T \we 1)^* \cdot \iota_*(f)] = [\iota_*(f)] = [f]$. Thus $d$ acts trivially. Following
Lemma \ref{mit-lem-tame}, the $\Mst$-action on $\pi_q^V(X)$ is trivial.
\item
By the tameness of $\pi_q^V(X)$ (use Lemma \ref{mact-ifuncsurj}),
this follows from Lemma \ref{mit-lem-tame} $(iv)$.
\end{enumerate}
\end{proof}

\begin{remark}\label{RSOsemi}
The result in \cite[Proposition 3.2]{RSO} provides exactly the same criterion
as Propostion \ref{motsemi-mactioninduced-appl-prop}-2.
\end{remark}

The motivic stable homotopy category contains various spectra
$X$ which come with a natural action of the general linear group.
If this action is compatible with the action of the symmetric group,
then $X$ is semistable:

\begin{corollary}
\label{motsemi-mactioninduced-appl-cor}
Let $E$ be a symmetric $T$-spectrum. Assume that for any $n \in \Nbb_0$ 
there is an $E^\prime_n$ in $M_\cdot(S)$ with $\Sigma_n$-action,
a zig-zag of  $\Sigma_n$-equivariant maps between  $E_n$ and $E^\prime_n$
which are motivic weak equivalences and a map
$$h_\cdot(GL_{n\,S}) \we E^\prime_n \rightarrow E^\prime_n$$
in $M_\cdot(S)$ such that this linear action restrics to the
given $\Sigma_n$-action on $E^\prime_n$.
Then $E$ is semistable.
\end{corollary}
\begin{proof}
Let $\Dst = M_\cdot^{cm}(S)$ and $\tau \in \Sigma_n$ even 
with permutation matrix $P_\tau$. By Lemma \ref{app-permutations-connect-lem}
we know that $\underline{P_\tau}$ and $\underline{id}$ induce the same
endomorphism on $E^\prime_n$ in $Ho(\Dst)$, and the latter is the identity 
by assumption. Hence any even permutation acts trivially
on $E^\prime_n$ (in  $Ho(\Dst)$) as it is conjugated to the 
action on $E_n$. Now apply  Proposition 
\ref{motsemi-mactioninduced-appl-prop} $(ii)$.
\end{proof}

\begin{remark}
In fact, one may define the notion of
a motivic linear spectrum, using the canonical
action of $GL_n$ on ${\Abb}^n$ and the canonical
isomorphisms $({\Abb}^n/({\Abb^n-0})) \we 
({\Abb}^m/( {\Abb}^m-0)) \cong ({\Abb}^{n+m}/({\Abb^{n+m}-0}))$ 
(see \cite[Proposition 3.2.17]{MV}). Then the forgetful functor
from motivic linear spectra to motivic spectra with the projective,
flat,... model structure should create a projective, flat...
monoidal model structure on motivic linear spectra.
Moreover, this forgetful functor has a right adjoint
for formal reasons (see e.g. \cite[Proposition 3.2]{MMSS}),
and this Quillen adjunction is expected to be a Quillen equivalence. 
Motivic linear spectra will be a convenient framework for 
equivariant stable motivic homotopy theory.  
\end{remark}

\section{Examples of semistable motivic symmetric spectra}

In \cite{RSO} it is shown that algebraic $K$-theory may
be represented by an explicit semistable motivic spectrum. 
In this section, we discuss two further examples.
{\em In the following section, we only consider the motivic case,
that is $\Dst = M^{cm}_\cdot(S), T = \Pbb^1$, $\Bst = \{S^r \we \Gbb_m^s 
\we U_+| r,s\geq 0, U \in Sm/S\}$ as
in Proposition \ref{thmain-ex1}.}

\subsection{The motivic Eilenberg-Mac Lane spectrum} 
In \cite[Example 3.4]{DRO}, the motivic Eilenberg-Mac Lane spectrum
is defined as the evaluation of a certain motivic functor
on smash powers of $T$ (see \cite[Abschnitt 3]{DRO}). 
According to \cite[Lemma 4.6]{DRO} 
this represents integral motivic cohomology,
and this is the description we will use.

In general, consider a functor $H: M_\cdot(S) \rightarrow M_\cdot(S)$ 
with the following properties: First, there are natural functors
$H_{A,B}: Hom(A, B) \rightarrow Hom(H(A), H(B))$ compatible with
the composition and such that restriction to
$S$ and zero-simplices is just $H$ on morphisms. Second, $H$ maps
motivic weak equivalences between projective cofibrant objects
(see \cite[section 2.1]{DRO}) to motivic weak equivalences.
We will see below that these two properties are sufficient to define
a semistable motivic symmetric spectrum. To obtain the
motivic Eilenberg-Mac Lane spectrum as in  
\cite[Example 3.4]{DRO}, we must take $H = u \circ \Zbb_{tr}$
where $u$ denotes forgetting the transfers,
and the second property holds by \cite[S. 524]{DRO}.

Let $\tilde{T}$ be a projective cofibrant replacement of
$\Gbb_m \we S^1$. 

\begin{definition}
The \emph{motivic Eilenberg-Mac Lane spectrum} $\Hbb$ 
is the symmetric $\tilde{T}$-spectrum with $\Hbb_n := H(\tilde{T}^n)$,
$\Sigma_n$ acting by permutation of the smash functors
and structure maps $\Hbb_n \we \tilde{T} \rightarrow \Hbb_{n+1}$
adjoint to
$$\tilde{T} \xrightarrow{\text{unit}} Hom(\tilde{T}^n, \tilde{T}^n \we \tilde{T}) \xrightarrow{H_{\tilde{T}^n, \tilde{T}^n \we \tilde{T}}} Hom(H(\tilde{T}^n), H(\tilde{T}^n \we \tilde{T})).
$$
\end{definition}

Note that the compositions $\sigma^{\Hbb\,l}_n: \Hbb_n \we \tilde{T}^l 
\rightarrow \Hbb_{n+l}$ of the structure maps are adjoint to
$\tilde{T}^l \xrightarrow{\text{unit}} Hom(\tilde{T}^n, \tilde{T}^n \we \tilde{T}^l) \xrightarrow{H} Hom(H(\tilde{T}^n), H(\tilde{T}^{n+l}))$
because $H$ is compatible with compositions on $Hom$, hence
$\Sigma_n \times \Sigma_l$-equivariant.

The following Lemmas show that $\Hbb$ satisfies the assumptions
of Corollary \ref{motsemi-mactioninduced-appl-cor}.

\begin{lemma}
\label{motEilenberg-lemma1}
There is a zigzag of  $\Sigma_n$-equivariant maps
between $\tilde{T}^{\we n}$ and\\ $T^n := h_\cdot(\Abb^n_S) // h_\cdot((\Abb^n-0)_S)$, and this is a zigzag of motivic weak equivalences between
projectively cofibrant pointed objects.
\end{lemma}
\begin{proof}
Using Lemma \ref{app-permutations-sign-lem}, Lemma \ref{an-eqv-lem} 
and \cite[Lemma 3.2.13]{MV} we obtain the desired
zigzag \\
$\centerline{\xymatrix{h_\cdot(\Abb^n_S) // h_\cdot((\Abb^n - 0)_S) \xrightarrow{\sim} h_\cdot(\Abb^n_S) / h_\cdot((\Abb^n - 0)_S) \xleftarrow{\sim} (\Abb^1/\Gbb_m)^{\we n} \simeq (\Gbb_m \we S^1)^{\we n} \xleftarrow{\sim} \tilde{T}^n.}}$
if we replace everything projectively cofibrant. Choosing a functorial 
replacement, it is $\Sigma_n$-equivariant as well.
\end{proof}

\begin{lemma}
\label{motEilenberg-lemma2}
There is a zigzag of motivic weak equivalences which are
$\Sigma_n$-invariant between $\Hbb_n$ and 
$H(h_\cdot(\Abb^n_S) // h_\cdot((\Abb^n-0)_S))$.
\end{lemma}
\begin{proof}
The zigzag of weak equivalences follows from Lemma \ref{motEilenberg-lemma1} 
and the second above property of $H$, and equivariance follows
from the first property.
\end{proof}

\begin{lemma}
\label{motEilenberg-lem3}
\begin{enumerate}
\item
There is a map $h_\cdot(GL_{n\,S}) \we T^n \rightarrow T^n$
extending the $\Sigma_n$-action on $T^n$.
\item
There is a map $h_\cdot(GL_{n\,S}) \we H(T^n) \rightarrow H(T^n)$
extending the $\Sigma_n$-action on $H(T^n)$.
\end{enumerate}
\end{lemma}
\begin{proof}
\begin{itemize}
\item
We have a commutative diagram \\
$\centerline{\xymatrix{
h_\cdot(GL_{n\,S}) \we h_\cdot((\Abb^n-0)_S) \ar[d]^{1 \we h_\cdot(incl)} \ar[r]^(0.6){\mu} & h_\cdot((\Abb^n-0)_S) \ar[d]^{h_\cdot(incl)} \\
h_\cdot(GL_{n\,S}) \we h_\cdot(\Abb^n_S) \ar[r]^(0.6){\mu} & h_\cdot(\Abb^n_S)
}}$
where the maps $\mu$ extend the $\Sigma_n$-action.
As the smash product commutes with colimits, the diagram
induces a map
$h_\cdot(GL_{n\,S}) \we T^n \rightarrow T^n$
extending the $\Sigma_n$-action.
\item
The map in the first part is adjoint to a map 
$h_\cdot(GL_{n\,S}) \rightarrow Hom(T^n,T^n)$
whose composition with $H_{T^n,T^n}$ is adjoint
to a map $h_\cdot(GL_{n\,S}) \we H(T_2^n) \rightarrow H(T_2^n)$.
The latter extends the $\Sigma_n$-action because
$H_{T^n,T^n}(S)$ is the map $M_\cdot(S)(T^n,T^n) \rightarrow M_\cdot(S)(H(T^n),H(T^n))$ and the $\Sigma_n$-action on $H(T^n)$ 
is induced by the one on $T^n$.
\end{itemize}
\end{proof}

\begin{corollary}
\label{motEilenberg-cor1}
The motivic Eilenberg-MacLane spectrum $\Hbb_n$
is semistable.
\end{corollary}
\begin{proof}
This follows from Lemma \ref{motEilenberg-lemma2}, Lemma
\ref{motEilenberg-lem3} and 
Corollary \ref{motsemi-mactioninduced-appl-cor}.
\end{proof}

\subsection{The algebraic cobordism spectrum}

The algebraic cobordism spectrum was first
defined in \cite[Abschnitt 6.3]{V2}. In
\cite[section 6.5]{PY} (see also \cite[section 2.1]{P2})
it is shown how to construct it as 
a motivic symmetric commutative ring spectrum.
We only care about the underlying motivic symmetric 
spectrum $\mathbb{MGL}$ (see Definition \ref{defMGL} below)
and will show that it is semistable.

Recall the following definition of \cite{MV}.
Let $X$ be an $S$-scheme and $\xi: E \rightarrow X$ a vector bundle.
Then the zero section $z(\xi): X \rightarrow E$ of $\xi$ is
a closed immersion, and the \emph{Thom space} $Th(\xi)$ of
$\xi$ is the pointed motivic space $a[h_\cdot(E)/(h_\cdot(U(\xi))]$.

\begin{lemma}
\label{UThlemma}
\begin{enumerate}
\item
Let $A$ be an $S$-scheme. Then $U(1_A) = \emptyset$, and there is
a natural motivic pointed weak equivalence $h_\cdot(A) \rightarrow Th(1_A)$.
\item
Let $X, X^\prime$ be two $S$-schemes with vector bundles
$\xi: V \rightarrow X, \xi^\prime: V^\prime \rightarrow X^\prime$.
Then $U(\xi \times_S \xi^\prime) = pr_1^{-1}(\xi) \cup pr_2^{-1}(\xi^\prime)$.
Furthermore, we have a motivic pointed weak equivalence
$Th(\xi) \we Th(\xi^\prime) \stackrel{\simeq}{\rightarrow} Th(\xi \times_S \xi^\prime)$ which is associative and commutes with the permutation
of $\xi$ and $\xi^\prime$. The composition $h_\cdot(A) \we Th(\xi) 
\rightarrow Th(1_A) \we Th(\xi) \rightarrow Th(1_A \times_S \xi)$ 
is denoted by $Th_{A,\xi}$. Then the following diagram commutes:\\
$\centerline{\xymatrix@=14pt{
Th(\xi) \ar[d]^{\cong} \ar[rd]^{Th(\cong)} \\
h_\cdot(S) \we Th(\xi) \ar[r]^{Th_{S,\xi}} & Th(1_S \times_S \xi)
}}$\\
\end{enumerate}
\end{lemma}
\begin{proof}
Straightforward.
\end{proof}

Considering schemes as functors on commutative rings
\cite[4.4 Comparison Theorem in I, §1]{DG},
we define Grassmannian schemes $Gr(d, n)$ in the usual way
(see \cite[I, § 1, 3.4 and I, §2, 4.4]{DG}).
The tautological bundle is denoted by 
$\xi_{n,d}: \tau(d, n) \rightarrow Gr(d, n)$.

\begin{lemdef}
\label{xi-upsilon}
For $m,n \geq 0$ there is a commutative diagramm of $GL_n$-equivariant
maps \\
$\centerline{\xymatrix{
\tau(n,nm) \ar[r] \ar[d]^{\xi_{n,nm}} &  \tau(n,n(m+1)) \ar[d]^{\xi_{n,n(m+1)}} \\
Gr(n,nm) \ar[r] & Gr(n,n(m+1))
}}$
The induced morphism $\xi_{n,nm} \rightarrow \xi_{n,n(m+1)}$ will be denoted by
$\upsilon_{n,m}$. Then $U(\xi_{n,nm})$ is mapped to $U(\xi_{n,n(m+1)})$.
\end{lemdef}
\begin{proof}
Straightforward.
\end{proof}

As before, we may restrict the  $GL_n$-action
to a $\Sigma_n$-action. Then we are ready for the
definition of $\mathbb{MGL}$. Recall that $\overline{T}$
is the Thom space of the trivial line bundle on $S$.

\begin{definition}\label{defMGL}
The symmetric \emph{algebraic cobordism spectrum} $\mathbb{MGL}$ 
is the underlying $\overline{T}$-Spectrum of the following motivic 
commutative ring spectrum:
\begin{itemize}
\item
The sequence of motivic spaces $\mathbb{MGL}_n := colim_{m\geq 1} (\ldots \rightarrow Th(\xi^S_{n,nm}) \xrightarrow{Th(\upsilon_{n,m})} Th(\xi^S_{n,n(m+1)})\rightarrow \ldots), n\geq 0$ with the induced $\Sigma_n$-action,
\item
$\Sigma_n\times \Sigma_p$-equivariant multiplication maps $\mu_{n,p}: \mathbb{MGL}_n \we \mathbb{MGL}_p \rightarrow \mathbb{MGL}_{n+p}, n,p\geq 0$
induced by $Th(\xi_{n,nm}^S) \we Th(\xi_{p,pm}^S) \rightarrow Th(\xi_{n,nm}^S \times_S \xi_{p,pm}^S) \xrightarrow{Th(\mu_{n,p,m})} Th(\xi_{n+p,(n+p)m}^S)$ 
\item
$\Sigma_n$-equivariant unit maps $\iota_n: \overline{T}^n \rightarrow \mathbb{MGL}_n, n \geq 0$ which for $n \geq 1$ are given by the compositions
$\overline{T}^n \cong Th(\xi^S_{1,1})^{\we n} \rightarrow Th(\xi^{S\,\times_S n}_{1,1}) \rightarrow Th(\xi^S_{n,n}) \rightarrow \mathbb{MGL}_n$ (and for 
$n = 0$ by $S^0 = h_\cdot(S) \rightarrow Th(1_S) \cong Th(\xi_{0,0}^S) \xrightarrow{\cong} \mathbb{MGL}_0$).
\end{itemize}
\end{definition}

Now the semistability of $\mathbb{MGL}$ follows from the
above discussion and (again)
Corollary \ref{motsemi-mactioninduced-appl-cor}
\begin{corollary}
\label{motCobord-CobordGLaction-cor2}
The motivic symmetric spectrum $\mathbb{MGL}$ is semistable.
\end{corollary}
\begin{proof}
We have a morphism $a_{\mathbb{MGL}_n}: h_\cdot(GL_n^S) \we \mathbb{MGL}_n \rightarrow \mathbb{MGL}_n$ in $M_\cdot(S)$ induced by the following commutative diagram: \\
$\centerline{\xymatrix{
h_\cdot(GL_n^S) \we Th(\xi_{n,nm}^S) \ar[d]^{1 \we Th(\upsilon_{n,m})} \ar[r]^{Th_{GL_n^S,\xi_{n,nm}^S}} & Th(GL_n^S \times_S \xi_{n,nm}^S) \ar[d]^{Th(GL_n^S \times_S \upsilon_{n,m})} \ar[r]^(0.63){Th(a_{n,m}^S)} & Th(\xi_{n,nm}^S) \ar[d]^{Th(\upsilon_{n,m})} \\
h_\cdot(GL_n^S) \we Th(\xi_{n,n(m+1)}^S) \ar[r]^{GL_n^S, \xi_{n,n(m+1)}^S} & Th(GL_n^S \times_S \xi_{n,n(m+1)}^S) \ar[r]_(0.63){Th(a_{n,m+1}^S)} & Th(\xi_{n,n(m+1)}^S)
}}$
Here the left hand square commutes by naturality (see Lemma \ref{UThlemma}) 
and the right hand square by functoriality of Thom spaces and
the $GL_n$-equivariance in Lemma \ref{xi-upsilon}.
Now for $\tau \in \Sigma_n$ and $S \xrightarrow{f_\tau} GL_n^S$ the associated
matrix, the following square commutes (see Lemma \ref{UThlemma}): \\
$\centerline{\xymatrix@=15pt{
Th(\xi^S_{n,nm}) \ar[rd]^{Th(\cong)} \ar[d]_{\cong} \ar@/^2.2pc/[rrdd]^(0.6){Th(\tau_*)} \\
h_\cdot(S) \we Th(\xi^S_{n,nm}) \ar[r]^{Th_{S,\xi^S_{n,nm}}} \ar[d]_{h_\cdot(f_\tau) \we 1} & Th(S \times_S \xi^S_{n,nm}) \ar[d]^{Th(f_\tau \times_S 1)}\\
h_\cdot(GL_n^S) \we Th(\xi^S_{n,nm}) \ar[r]_*!/d5pt/{\labelstyle Th_{GL_n^S,\xi^S_{n,nm}}} & Th(GL_n^S \times_S \xi^S_{n,nm}) \ar[r]_(0.63){Th(a_{n,m}^S)} & Th(\xi^S_{n,nm})
}}$
Thus $h_\cdot(GL_n^S) \we \mathbb{MGL}_n \rightarrow \mathbb{MGL}_n$ extends
the $\Sigma_n$-action on $\mathbb{MGL}_n$, and the semistability
follows from Corollary \ref{motsemi-mactioninduced-appl-cor}.
\end{proof}

\section{The multiplicative structure on stable homotopy groups
of symmetric ring spectra and its localizations}
\label{motstabmult}
In this section, we will prove a generalization of \cite[Corollary I.4.69]{S1}.
More precisely, we will show that the localization
$R[1/x]$ (see below) of a semistable symmetric ring spectrum $R$ 
with respect to a suitable $x$ is again semistable and the map
$j: R \rightarrow R[1/x]$ behaves as expected on stable
homotopy groups (see section \ref{sect-locringspec}).

{\em Throughout this section, we assume the following:
The assumptions of section \ref{sect-defmact} hold
and $T$ has a sign. The smash product in $\Dst$ 
preserves weak equivalences, which is the case for simplicial sets and motivic 
spaces by \cite[Lemma 3.2.13]{MV}.
We also assume that there is a commutative monoid
$N$  with zero, for any $r \in N$ a cofibrant object $\Sbb^r$ 
and isomorphisms $s_{r_1,r_2}: \Sbb^{r_1 + r_2} \rightarrow \Sbb^{r_1} 
\we \Sbb^{r_2}$ in $\Dst$ for all $r_1, r_2 \in N$ such that 
the following holds:
\begin{itemize}
\item There is an isomorphism $\cong^{\Sbb^0}: \Sbb^0 \cong S^0$,
\item $s_{-,-}$ is associative
\item There are isomorphisms $s_{0,r} \cong l_{\Sbb^r}^{-1}$ and 
$s_{r,0} \cong \rho_{\Sbb^r}^{-1}$ (via $\Sbb^0 \cong S^0$) (here $l$ and 
$\rho$ are the obvious structure morphisms, see \cite[chapter 4]{H3}).
\end{itemize}
Finally, we assume that there is  a class of cofibrant objects
$\Bst^\prime$ in $\Dst$ with
$\Bst =\{\Sbb^r \we U| r \in N, U \in \Bst^\prime \}$.}

\begin{example}
The standard example is, of course, $N = \Nbb_0$ and $\Sbb^r = S^r = 
(S^1)^{\we r}$ together with the identities $S^{r_1 + r_2} = S^{r_1} 
\we S^{r_2}$ (recall that the simplicial spheres are in $\Dst$
via $i$ by assumption). If $\Dst = M_\cdot(S)$ and
$\Bst = \{S^r \we \Gbb_m^s \we U_+| r,s\geq 0, U \in Sm/S\}$ 
as above, we may also
consider $N = \Nbb_0^2$ and $\Sbb^r = S^{r^\prime} \we \Gbb_m^{\we r^\dprime}$ 
with $r = (r^\prime, r^\dprime)$ and the isomorphisms given 
by the obvious permutations. Note that in general
$\Sbb$ and $T$ may be completely unrelated, but in the 
motivic case that we care about they are the same.
\end{example}

\begin{definition}
For any symmetric $T$-spectrum $X$ we set\\
$\centerline{\xymatrix{\pi_{r,q}^U(X) := \pi_q^{\Sbb^r \we U}(X),}}$
for alle $r \in N, U \in \Bst^\prime, q\in \Zbb$.
We further set $\Sbb^{t_{r,r^\prime}} = s_{r,r^\prime}^{-1} t_{\Sbb^r,\Sbb^{r^\prime}} s_{r,r^\prime}$ and obtain maps $t_{r^\prime,r}: 
\pi_{r^\prime+r,q}^U(X) \rightarrow \pi_{r+r^\prime,q}^U(X)$ induced
by the maps\\
$\centerline{\xymatrix{
[\Sbb^{r\prime} \we \Sbb^r \we U \we T^{q+m}, X_m] \xrightarrow{(\Sbb^{t_{r,r^\prime}} \we T \we T^{q+m})^*} [\Sbb^r \we \Sbb^{r\prime} \we U \we T^{q+m}, X_m]
}}$
In particular, we have  $t_{0,r} = t_{r,0} = id$ as $l_{\Sbb^r} \circ t_{\Sbb^r,S^0} = \rho_{\Sbb^r}$.
\end{definition}

In the motivic case, one of the indices is of course redundant.
Namely, if $\Sbb^r = S^{r^\prime} \we \Gbb_m^{\we r^\dprime}$ 
(hence $r=(r',r'')$ and $U=S^0$, we have $\pi_{r,q}^U(X) \cong 
\pi^{mot}_{q+r'+r'',q+r'}(X)$, where we used Voevodsky's
indexing on the right hand side.

\subsection{The multiplication on stable homotopy groups}

The following generalizes the multiplication of 
stable homotopy groups for usual symmetric ring spectra
(see e.g. \cite[section I.4.6]{S1}).
The sign $(-1)_T^{q^\prime n}$ below will be used to show
that the product is compatible with stabilization.
See \cite[Definition I.1.3]{S1} (resp. its obvious generalization)
for the definition of a (commutative) symmetric ring spectra.
In particular, for any symmetric ring spectrum $R$ we have maps
$\mu_{n,m}: R_n \times R_m \to R_{n+m}$.  
Recall also the definition of central elements
$x:T^{l+m} \to R_m $ of \cite[Proposition I.4.61(i)]{S1}.
Those are stable under smash multiplication: if $y:T^{k+n} \to R_n$
is another central element, then $\mu \circ (x \we y)$ is also central.
If $R$ is commutative, then of course all maps $T^{l+m} \to R_m $ 
are central.
\begin{lemma}
\label{ringspec-outprod-def}
Let $R$ be a semistable symmetric $T$-ring spectrum. 
Then for any cofibrant objects $U, V$ in $\Dst$ and $r,r^\prime \in N, q,q^\prime \in \Zbb$, there is a natural (in $R, U, V$) biadditive map
$$m_{r,q,r^\prime,q^\prime}^{U,V,R}: \pi_{r, q}^U(R) \times \pi_{r^\prime, q^\prime}^V(R) \xrightarrow{} \pi_{r+r^\prime, q+q^\prime}^{U\we V}(R)$$
induced by 
$$\cdot: [\Sbb^r \we U \we T^{q+n}, R_n] \times [\Sbb^{r^\prime} \we V \we T^{q^\prime+n^\prime}, R_{n^\prime}] \rightarrow [\Sbb^{r+r^\prime} \we U \we V \we T^{q+q^\prime+n+ n^\prime}, R_{n + n^\prime}].$$
This pairing maps $(f,g)$ to the composition
$\Sbb^{r+r^\prime} \we U \we V \we T^{q+q^\prime+n+ n^\prime} \xrightarrow{s_{r,r^\prime} \we U \we V \we (-1)_T^{q^\prime n} \we 1} \Sbb^r\we \Sbb^{r^\prime} \we U \we V \we T^{q+n} \we T^{q^\prime+n^\prime} \xrightarrow{\eta^{r,r^\prime,U,V}_{q+n,q^\prime + n^\prime}} (\Sbb^r \we U \we T^{q+n}) \we (\Sbb^{r^\prime} \we V \we T^{q^\prime+n^\prime}) \xrightarrow{f \we g} R_n \we R_{n^\prime} \xrightarrow{\mu_{n,n^\prime}} R_{n + n^\prime}$
with
$\eta^{r,r^\prime,U,V}_{q+n,q^\prime + n^\prime}$ 
being the obvious permutation of smash functors.

The product is associative, that is the square\\
$\centerline{\xymatrix{
\pi_{r, q}^U(R) \times \pi_{r^\prime, q^\prime}^V(R) \times \pi_{r^\dprime, q^\dprime}^W(R) \ar[r]^(0.53){m \times 1} \ar[d]^{1 \times m} &  \pi_{r+r^\prime, q+q^\prime}^{U\we V}(R) \times \pi_{r^\dprime, q^\dprime}^W(R) \ar[d]^{m} \\
\pi_{r, q}^U(R) \times \pi_{r^\prime + r^\dprime, q^\prime + q^\dprime}^{V\we W}(R) \ar[r]^(0.53){1 \times m} & \pi_{r+r^\prime + r^\dprime, q+q^\prime + q^\dprime}^{U \we V\we W}(R)
}}$\\
commutes. It is compatible with the sign $(-1)_T$ 
in both variables, namely we have
$$(-1)_T(f \cdot g) = ((-1)_T f) \cdot g = f \cdot ((-1)_T g)$$
If the ring spectrum $R$ is commutative, then the multiplication
on stable homotopy groups is commutative, that is the square\\
$\centerline{\xymatrix{
\pi^U_{r,q}(R) \times \pi^V_{r^\prime,q^\prime}(R) \ar[r]^(0.575){m^{U,V}} \ar[d]^{t} & \pi^{U \we V}_{r+r^\prime,q+q^\prime}(R) \ar[rd]^{(-1)_T^{q^\prime q} t_{r^\prime,r}} \\
\pi^V_{r^\prime,q^\prime}(R) \times \pi^U_{r,q}(R) \ar[r]^(.575){m^{V,U}} & \pi^{V \we U}_{r^\prime+r,q^\prime+q}(R) \ar[r]^{t_{U,V}^\ast} & \pi^{U \we V}_{r^\prime+r,q^\prime+q}(R)
}}$\\
also commutes. Finally, if $f: \Sbb^r \we U \we T^{q+n} \rightarrow R_n$ is 
central in $\Dst$, then $t_{U,V}^\ast \circ (g\cdot [f]) = (-1)_T^{q^\prime q}
t_{r^\prime,r} ([f] \cdot g)$.

\end{lemma}
\begin{proof}
\textbf{0. Biadditivity}:
The  product is biadditive already before
stabilization.
This is a long, but straightforward verification.

\textbf{1. Associativity}:
We show that the  product is associative already before
stabilization. (The symbol $\eta$ below denotes various obvious isomorphisms.)
Let $f \in [\Sbb^r \we U \we T^{q+n}, R_n]$, $g \in [\Sbb^{r^\prime} \we V \we T^{q^\prime+n^\prime}, R_{n^\prime}]$ and $h \in [\Sbb^{r^\dprime} \we W \we T^{q^\dprime+n^\dprime}, R_{n^\dprime}]$. Then we have\\
$f\cdot(g \cdot h) = \mu_{n,n^\prime + n^\dprime} \circ (f \we (g \cdot h)) \circ \eta^{r,r^\prime+r^\dprime,U,V\we W}_{q+n,q^\prime+q^\dprime+n^\prime+n^\dprime} \circ (s_{r,r^\prime+r^\dprime} \we 1 \we (-1)_T^{(q^\prime + q^\dprime)n} \we 1) \\
= \mu_{n,n^\prime + n^\dprime} \circ (f \we [\mu_{n^\prime,n^\dprime} \circ (g \we h) \circ \eta^{r^\prime,r^\dprime,V,W}_{q^\prime + n^\prime,q^\dprime + n^\dprime} \circ (s_{r^\prime,r^\dprime} \we 1 \we (-1)_T^{q^\dprime n^\prime} \we 1)]) \circ \eta^{r,r^\prime+r^\dprime,U,V\we W}_{q+n,q^\prime+q^\dprime+n^\prime+n^\dprime} \circ (s_{r,r^\prime+r^\dprime} \we 1 \we (-1)_T^{(q^\prime + q^\dprime)n} \we 1) \\
= [\mu_{n,n^\prime + n^\dprime} \circ (1 \we \mu_{n^\prime,n^\dprime})] \circ (f \we g\we h) \circ (1 \we [\eta^{r^\prime,r^\dprime,V,W}_{q^\prime + n^\prime,q^\dprime + n^\dprime} \circ (s_{r^\prime,r^\dprime} \we 1 \we (-1)_T^{q^\dprime n^\prime} \we 1)]) \circ \eta^{r,r^\prime+r^\dprime, U, V\we W}_{,q+n,q^\prime+ q^\dprime+n^\prime +n^\dprime} \circ (s_{r,r^\prime+r^\dprime} \we 1 \we (-1)_T^{(q^\prime + q^\dprime)n} \we 1)\\
= [\mu_{n,n^\prime + n^\dprime} \circ (1 \we \mu_{n^\prime,n^\dprime})] \circ (f \we g\we h) \circ [(1 \we \eta^{r^\prime,r^\dprime,V,W}_{q^\prime + n^\prime,q^\dprime + n^\dprime}) \circ \eta^{r,r^\prime+r^\dprime,U,V\we W}_{q+n,q^\prime+q^\dprime+n^\prime+n^\dprime}] \circ [1 \we (s_{r^\prime,r^\dprime} \we 1 \we T^{q+n} \we (-1)_T^{q^\dprime n^\prime} \we 1)\circ (s_{r,r^\prime+r^\dprime} \we 1 \we (-1)_T^{(q^\prime + q^\dprime)n} \we 1)]\\
= [\mu_{n,n^\prime + n^\dprime} \circ (1 \we \mu_{n^\prime,n^\dprime})] \circ (f \we g\we h) \circ [(1 \we \eta^{r^\prime,r^\dprime,V,W}_{q^\prime + n^\prime,q^\dprime + n^\dprime}) \circ \eta^{r,r^\prime+r^\dprime,U,V\we W}_{q+n,q^\prime+q^\dprime+n^\prime+n^\dprime}] \circ [1 \we (s_{r^\prime, r^\dprime} \we 1 \we (-1)_T^{q^\dprime n^\prime} \we 1)\circ (s_{r,r^\prime+r^\dprime} \we 1 \we (-1)_T^{(q^\prime + q^\dprime)n} \we 1)]\\
= [\mu_{n,n^\prime + n^\dprime} \circ (1 \we \mu_{n^\prime,n^\dprime})] \circ (f \we g\we h) \circ [(1 \we \eta^{r^\prime,r^\dprime,V,W}_{q^\prime + n^\prime,q^\dprime + n^\dprime}) \circ \eta^{r,r^\prime+r^\dprime,U,V\we W}_{q+n,q^\prime+q^\dprime+n^\prime+n^\dprime}] \circ  (((1 \we s_{r^\prime,r^\dprime}) s_{r,r^\prime+r^\dprime}) \we 1 \we (-1)_T^{q^\dprime n^\prime + (q^\prime + q^\dprime)n} \we 1)\\
$\\
Here the second last equality uses Definition \ref{vorzeichen-def},
which yields 
$\Sbb^{r+r^\prime + r^\dprime} \we U \we V \we W \we T^{q+n} \we (-1)_T^{q^\dprime n^\prime} \we T^{q^\prime+q^\dprime + n^\prime + n^\dprime -1}  
=  \Sbb^{r+r^\prime+r^\dprime} \we U \we V \we W \we (-1)_T^{q^\dprime n^\prime} \we T^{q+q^\prime +q^\dprime + n+ n^\prime + n^\dprime -1}$.

A similar computation (slightly easier, Definition \ref{vorzeichen-def}
is not used here) shows that
$(f\cdot g)\cdot h \\
= [\mu_{n+n^\prime,n^\dprime} \circ (\mu_{n,n^\prime} \we 1)] \circ (f \we g \we h) \circ [(\eta^{r,r^\prime,U,V}_{q+n,q^\prime + n^\prime} \we 1) \circ \eta^{r+r^\prime,r^\dprime,U\we V,W}_{q+q^\prime +n+n^\prime,q^\dprime + n^\dprime}] \circ (((s_{r,r^\prime} \we 1) s_{r+r^\prime,r^\dprime}) \we 1 \we (-1)_T^{q^\prime n + q^\dprime (n + n^\prime)} \we 1)
$.\\

As $R$ is associative, we have $\mu_{n,n^\prime + n^\dprime} \circ 
(1 \we \mu_{n^\prime,n^\dprime}) = \mu_{n+n^\prime,n^\dprime} \circ 
(\mu_{n,n^\prime} \we 1)$.
Moreover $(1 \we \eta^{r^\prime,r^\dprime,V,W}_{q^\prime + n^\prime,q^\dprime + n^\dprime}) \circ \eta^{r,r^\prime+r^\dprime,U,V\we W}_{q+n,q^\prime+q^\dprime+n^\prime+n^\dprime}
= (\eta^{r,r^\prime,U,V}_{q+n,q^\prime + n^\prime} \we 1) \circ \eta^{r+r^\prime,r^\dprime,U\we V,W}_{q+q^\prime +n+n^\prime,q^\dprime + n^\dprime}$
as both sides are induced by permutations, and finally 
$q^\prime n + q^\dprime (n + n^\prime) = q^\dprime n^\prime + 
(q^\prime + q^\dprime)n$ and $(s_{r,r^\prime} \we 1) s_{r+r^\prime,r^\dprime} 
= (1 \we s_{r^\prime,r^\dprime}) s_{r,r^\prime+r^\dprime}$.

\textbf{2. Compatibility with stabilization}:
We show that the unstable product above is compatible
with the stabilization $\iota_* := \sigma_* \cdot (-\we T)$ 
in both variables. For the second variable, we must show that\\
$\centerline{\xymatrix@=16pt{
[\Sbb^r \we U \we T^{q+n}, R_n] \times [\Sbb^{r^\prime} \we V \we T^{q^\prime+n^\prime}, R_{n^\prime}] \ar[d]^{1 \times \iota_*} \ar[r]^(.535){\cdot} & [\Sbb^{r+r^\prime} \we U \we V \we T^{q+q^\prime+n+ n^\prime}, R_{n + n^\prime}] \ar[d]^{\iota_*}\\
[\Sbb^r \we U \we T^{q+n}, R_n] \times [\Sbb^{r^\prime} \we V \we T^{q^\prime+n^\prime+1}, R_{n^\prime+1}] \ar[r]^(0.535){\cdot} & [\Sbb^{r+r^\prime} \we U \we V \we T^{q+q^\prime+n+ n^\prime+1}, R_{n + n^\prime+1}]
}}$
commutes. For
$f \in [\Sbb^r \we U \we T^{q+n}, R_n]$, and $c = l_T \circ l_{S^0 \we T} \circ (\cong \we 1): \Sbb^0 \we S^0 \we T \rightarrow T$, we have\\
$[f \cdot (\iota_1 \circ c)] \circ (\Sbb^r \we \rho^{-1}_U \we T^{q+n+1}) = \mu_{n,1} \circ (f \we (\iota_1 \circ c)) \circ \eta^{r,0,U,S^0}_{q+n,0+1} \circ (s_{r,0} \we 1) \circ (\Sbb^r \we \rho^{-1}_U \we T^{q+n+1})
= \mu_{n,1} \circ (f \we \iota_1) = \sigma_n \circ (f \we T) = \iota_*(f)$ 
because of\\
$(1 \we ([l_T l_{S^0 \we T} \circ ((\Sbb^0 \rightarrow S^0) \we 1)]) \circ (\Sbb^r \we [(t_{\Sbb^0, U \we T^{q+n}} \we S^0)(\Sbb^0 \we U \we t_{S^0,T^{q+n}})] \we T)  \circ (s_{r,0} \we \rho^{-1}_U \we T^{q+n+1})
= (1 \we l_T l_{S^0 \we T}) \circ (\Sbb^r \we [(t_{S^0, U \we T^{q+n}} \we S^0)(S^0 \we U \we t_{S^0,T^{q+n}})] \we T)  \circ (\rho_{\Sbb^r}^{-1} \we \rho^{-1}_U \we T^{q+n+1})
= 1$.
Applying this to $g$ and $f \cdot g$, together with associativity and 
naturality we obtain \\ $\iota_*(f \cdot g) = [(f \cdot g) \cdot (\iota_1 c)] \circ (1 \we \rho_{U \we V}^{-1} \we 1) = [f \cdot (g \cdot (\iota_1 c))] \circ (1 \we U \we \rho_{V}^{-1} \we 1) = f \cdot [(g \cdot (\iota_1 c)) \circ (1 \we \rho_{V}^{-1} \we 1)] = f \cdot \iota_*(g)$. This yields a map $[\Sbb^r \we U \we T^{q+n}, R_n] \times \pi_{r^\prime,q^\prime}^U(R) \rightarrow \pi_{r+r^\prime,q+q^\prime}^U(R)$.

The first variable is more subtle. By
$\iota_*(f) \cdot g = [(f \cdot \iota_1 c) \circ (\Sbb^r \we \rho^{-1}_U \we T^{q+n+1})] \cdot g = [(f \cdot \iota_1 c) \cdot g] \circ (1 \we \rho^{-1}_U \we 1) = [f \cdot (\iota_1 c \cdot g)] \circ (1 \we l^{-1}_V \we 1) = f \cdot [(\iota_1 c \cdot g)\circ (\Sbb^{r^\prime} \we l^{-1}_V \we T^{1 + q^\prime + n^\prime})]$
and the above it suffices to show that $[(\iota_1 c \cdot g)\circ (\Sbb^{r^\prime} \we l^{-1}_V \we T^{1 + q^\prime + n^\prime})] = [g]$ in $\pi_{r^\prime,q^\prime}^V(R)$. For this, we first note that $(\iota_1 c \cdot g)\circ (\Sbb^{r^\prime} \we l^{-1}_V \we T^{1 + q^\prime + n^\prime}) = \chi_{n^\prime,1} \circ \iota_*(g) \circ (1\we(-1)_T^{n^\prime}\we 1)$ by the following computation:

$(\iota_1 c \cdot g)\circ (\Sbb^{r^\prime} \we l^{-1}_V \we T^{1 + q^\prime + n^\prime})
= \mu_{1,n^\prime} \circ (\iota_1 \we 1) \circ (c \we g) \circ \eta^{0,r^\prime,S^0,V}_{0+1,q^\prime + n^\prime} \circ (s_{0,r^\prime} \we 1 \we (-1)_T^{q^\prime} \we 1) \circ (\Sbb^{r^\prime} \we l^{-1}_V \we T^{1 + q^\prime + n^\prime})\\
\underset{central}{=} [\chi_{n^\prime,1} \circ \mu_{n^\prime,1} \circ (1 \we \iota_1) \circ t_{T, R_{n^\prime}}] \circ (c \we g) \circ \eta^{0,r^\prime,S^0,V}_{0+1,q^\prime + n^\prime} \circ (s_{0,r^\prime} \we l^{-1}_V \we T^{1 + q^\prime + n^\prime}) \circ (1 \we (-1)_T^{q^\prime} \we 1)\\
= \chi_{n^\prime,1} \circ \mu_{n^\prime,1} \circ (g \we \iota_1) \circ [t_{T, \Sbb^{r^\prime} \we V \we T^{q^\prime + n^\prime}} \circ (c \we 1) \circ \eta^{0,r^\prime,S^0,V}_{0+1,q^\prime + n^\prime} \circ (s_{0,r^\prime} \we l^{-1}_V \we T^{1 + q^\prime + n^\prime})] \circ (1 \we (-1)_T^{q^\prime} \we 1)\\
= \chi_{n^\prime,1} \circ \mu_{n^\prime,1} \circ (g \we \iota_1) \circ (\Sbb^{r^\prime} \we V \we t_{T,T^{q^\prime+n^\prime}}) \circ (1 \we (-1)_T^{q^\prime} \we 1)\\
= \chi_{n^\prime,1} \circ \iota_*(g) \circ (\Sbb^{r^\prime} \we V \we (-1)_T^{q^\prime+n^\prime}  \we T^{q^\prime+n^\prime}) \circ (1 \we (-1)_T^{q^\prime} \we 1)\\
= \chi_{n^\prime,1} \circ \iota_*(g) \circ (1 \we (-1)_T^{n^\prime} \we 1)\\
$

Stabilizing this, we obtain $[\chi_{n^\prime,1} \circ \iota_*(g) \circ (1 \we (-1)_T^{n^\prime} \we 1)] = d [\iota_*(g)] = [g]$, 
because $d$ acts trivially by semistability. Hence we have 
$[\iota_*(f) \cdot g] = [f \cdot \chi_{n^\prime,1} \circ \iota_*(g) \circ (1 \we (-1)_T^{n^\prime} \we 1)] = f \cdot [\chi_{n^\prime,1} \circ \iota_*(g) \circ (1 \we (-1)_T^{n^\prime} \we 1)] = f \cdot [g] = [f \cdot g]$.

\textbf{3. Compatibility with the signs}:
This follows by the naturality of the permutation map $\eta_{q+n,q^\prime + n^\prime}^{r, r^\prime, U, V}$ together with the second property of
Definition \ref{vorzeichen-def}:
\begin{itemize}
\item
$[(\Sbb^r \we U \we (-1)_T\we T^{q+n-1}) \we 1] \circ \eta^{r,r^\prime,U,V}_{q+n,q^\prime +n^\prime} = \eta^{r,r^\prime,U,V}_{q+n,q^\prime +n^\prime} \circ (1 \we (-1)_T \we T^{q+n-1} \we T^{q^\prime +n^\prime})$
\item
$[1 \we (\Sbb^{r^\prime} \we U \we (-1)_T\we T^{q^\prime+n^\prime-1})] \circ \eta^{r,r^\prime,U,V}_{q+n,q^\prime +n^\prime} = \eta^{r,r^\prime,U,V}_{q+n,q^\prime +n^\prime} \circ (1 \we T^{q+n} \we (-1)_T \we T^{q^\prime +n^\prime-1})$
\item
$1 \we T^{q+n} \we (-1)_T \we T^{q^\prime +n^\prime-1} = 1 \we (-1)_T \we T^{q+n-1} \we T^{q^\prime +n^\prime}$
\end{itemize}

\textbf{4. Commutativity}: 
We have a commutative diagramm:\\
$\centerline{\xymatrix@=13pt{
\Sbb^{r+r^\prime} \we U \we V \we T^{q+q^\prime+n+ n^\prime} \ar[r]^*!/u4pt/{\labelstyle s_{r,r^\prime} \we 1 \we (-1)^{q^\prime n}_T \we 1} \ar[d] & \Sbb^r\we \Sbb^{r^\prime} \we U \we V \we T^{q+n} \we T^{q^\prime+n^\prime} \ar[r]^*!/u4pt/{\labelstyle \eta^{r,r^\prime,U,V}_{q+n,q^\prime + n^\prime}} \ar[d]^{t_{\Sbb^r,\Sbb^{r^\prime}} \we t_{U,V} \we t_{T^{q+n},T^{q^\prime+n^\prime}} } & (\Sbb^r \we U \we T^{q+n}) \we (\Sbb^{r^\prime} \we V \we T^{q^\prime+n^\prime}) \ar[d] \\
\Sbb^{r^\prime+r} \we U \we V \we T^{q^\prime+q+ n^\prime+n} \ar[r]_*!/d3pt/{\labelstyle s_{r^\prime,r} \we 1 \we t_{U,V} \we (-1)^{q n^\prime}_T \we 1} & \Sbb^{r^\prime}\we \Sbb^r \we V \we U \we T^{q^\prime+n^\prime} \we T^{q+n} \ar[r]_*!/d3pt/{\labelstyle\eta^{r^\prime,r,V,U}_{q^\prime + n^\prime,q+n}} &  (\Sbb^{r^\prime} \we V \we T^{q^\prime+n^\prime}) \we (\Sbb^r \we U \we T^{q+n})
}}$\\
where the right hand vertical map is
the permutation $t_{\Sbb^r \we U \we T^{q+n},\Sbb^{r^\prime} \we V \we T^{q^\prime+n^\prime}}$ and the left one is \\ $\Sbb^{t_{r,r^\prime}} \we U \we V \we [((-1)_T^{q n^\prime} \we 1) t_{T^{q+n},T^{q^\prime+ n^\prime}} ((-1)_T^{q^\prime n} \we 1)]$, for which we have $((-1)_T^{q n^\prime}\we 1) t_{T^{q+n},T^{q^\prime+ n^\prime}} ((-1)_T^{q^\prime n}\we 1) = (-1)_T^{q n^\prime + (q+n)(q^\prime + n^\prime) q^\prime n} \we T^{q+q^\prime + n + n^\prime-1} = (-1)_T^{q q^\prime + n n^\prime} \we T^{q+q^\prime + n + n^\prime-1}$. If $f: \Sbb^r \we U \we T^{q+n} \rightarrow R_n$ is central (e.g. if $R$ is commutative), we have
$\chi_{n,n^\prime} \circ \mu_{n,n^\prime} \circ (f \we 1) = \mu_{n^\prime,n} \circ (1 \we f)\circ t_{\Sbb^r \we U \we T^{q+n},R_{n^\prime}} $. 
Together with the above commutative diagram, for $g \in [\Sbb^{r^\prime} \we V \we T^{q^\prime + n^\prime},R_{n^\prime}]$ we then deduce
$\chi_{n,n^\prime}(f \cdot g) (1 \we (-1)_T^{n n^\prime} \we 1)
= \chi_{n,n^\prime} \circ \mu_{n,n^\prime} \circ (f \we g) \circ \eta_{q+n,q^\prime+n^\prime}^{r,r^\prime,U,V} \circ (s_{r,r^\prime} \we 1 \we (-1)_T^{q^\prime n} \we 1) \circ (1 \we (-1)_T^{n n^\prime} \we 1)\\
= \mu_{n^\prime,n} \circ (g \we f) \circ [t_{\Sbb^r \we U \we T^{q+n},\Sbb^{r^\prime} \we V \we T^{q^\prime+n^\prime}} \circ \eta_{q+n,q^\prime+n^\prime}^{r,r^\prime,U,V} \circ (s_{r,r^\prime} \we 1 \we (-1)_T^{q^\prime n} \we 1)] \circ (1 \we (-1)_T^{n n^\prime} \we 1)
= \mu_{n^\prime,n} \circ (g \we f) \circ [\eta_{q^\prime+n^\prime,q+n}^{r^\prime,r,V,U} \circ (s_{r^\prime,r} \we 1 \we t_{U,V} \we (-1)_T^{qn^\prime} \we 1) \circ (\Sbb^{t_{r,r^\prime}} \we U \we V \we [(-1)_T^{q q^\prime + n n^\prime} \we 1])] \circ (1 \we (-1)_T^{n n^\prime} \we 1)
= (g \cdot f) (\Sbb^{t_{r,r^\prime}} \we 1 \we (-1)_T^{q q^\prime} \we 1)$. 
As $R$ is semistable, this implies $[f \cdot g] = [\chi_{n,n^\prime}(f \cdot g) (1 \we (-1)_T^{n n^\prime} \we 1)]$, which yields commutativity.
\end{proof}

To obtain an internal product on stable homotopy groups, we assume from now on
that there are natural transformations $diag^U: U \rightarrow U \we U$ and 
$\omega_U: U \rightarrow S^0$ for any $U \in \Bst^\prime$ making $U$ 
a commutative comonoid in $\Dst$.

\begin{example}
In $sSet_*$ or $M_\cdot(S)$ we have $diag: K_+ \xrightarrow{diag} (K \times K)_+ \cong K_+ \we K_+$ for any $K$ in $sSet$ or $M(S)$.
\end{example}

We set $l^2_{T^l} := l_{T^l} \circ (S^0 \we l_{T^l})$ and define $c_l$ 
to be the map $\Sbb^0 \we U \we T^l \xrightarrow{\cong \we \omega_U \we T^l} S^0 \we S^0 \we T^l \xrightarrow{l^2_{T^l}} T^l$. In particular, $c_l \we T^n = c_{l+n}$, and $c_l=id$ if $U=S^0$.

\begin{propdef}[multiplicative structure on stable homotopy groups]
\label{ringspec-sthgrmonoid-def}
Let $R$ be a semistabile symmetric $T$-ring spectrum. Then we have a natural
(in $R$) structure of a $N\times \Zbb$-graded ring on $\pi_{*,*}^U(R) := \oplus_{(r,q) \in N \times \Zbb}{\pi_{r,q}^U(R)}$, which is induced by taking colimits
of the following biadditive maps ($q+n,q^\prime+n^\prime \geq 1$ as usual):\\
$\cdot: [\Sbb^r \we U \we T^{q+n}, R_n] \times [\Sbb^{r^\prime} \we U \we T^{q^\prime+n^\prime}, R_{n^\prime}] \rightarrow [\Sbb^{r+r^\prime} \we U \we T^{q+q^\prime+n+ n^\prime}, R_{n + n^\prime}]$.

Here a pair $(f,g)$ is mapped to the composition\\ 
$\Sbb^{r+r^\prime} \we U \we T^{q+q^\prime+n+ n^\prime} \xrightarrow{s_{r,r^\prime} \we diag^U \we (-1)_T^{q^\prime n} \we 1} \Sbb^r\we \Sbb^{r^\prime} \we U \we U \we T^{q+n} \we T^{q^\prime+n^\prime} \xrightarrow{\Sbb^r \we t_{\Sbb^{r^\prime}\we U, U \we T^{q+n}} \we T^{q^\prime + n^\prime}} \Sbb^r \we U \we T^{q+n} \we \Sbb^{r^\prime} \we U \we T^{q^\prime+n^\prime} \xrightarrow{f \we g} R_n \we R_{n^\prime} \xrightarrow{\mu_{n,n^\prime}} R_{n + n^\prime}$ ($q+n,q^\prime+n^\prime \geq 1$)\\
The product is compatible with the signs, and graded commutative if $R$
is commutative:\\
\[\centering
f \cdot g = (-1)_T^{q \cdot q^\prime} t_{r^\prime,r} (g \cdot f),
\]
for any $f \in \pi_{r,q}^U(R)$ and $g \in \pi_{r^\prime,q^\prime}^U(R)$.
For the latter equality it suffices that $f$ is represented by
a central map.
\end{propdef}
\begin{proof}
The multiplication decomposes as the external product 
of Lemma \ref{ringspec-outprod-def} and the diagonal:\\
$\centerline{\xymatrix{
m_{r,q,r^\prime,q^\prime}^{U,R}: \pi_{r, q}^U(R) \times \pi_{r^\prime, q^\prime}^U(R) \xrightarrow{m_{r,q,r^\prime,q^\prime}^{U,U,R}} \pi_{r+r^\prime, q+q^\prime}^{U\we U}(R) \xrightarrow{\pi_{r+r^\prime, q+q^\prime}^{diag^U}(R)} \pi_{r+r^\prime, q+q^\prime}^{U}(R)
}}$\\
because the map $\Sbb^r \we \Sbb^{r^\prime} \we V \we U \we T^{q+n} \we T^{q^\prime + n^\prime} \xrightarrow{1 \we t_{V,U} \we 1} \Sbb^r \we \Sbb^{r^\prime} \we U \we V \we T^{q+n} \we T^{q^\prime + n^\prime} \xrightarrow{\eta_{q+n,q^\prime+n^\prime}^{r,r^\prime,U,V}} \Sbb^r \we U \we T^{q+n} \we \Sbb^{r^\prime} \we V \we T^{q^\prime + n^\prime}$ coincides with
$\Sbb^r \we t_{\Sbb^{r^\prime} \we V,U \we T^{q+n}} \we T^{q^\prime +n^\prime}$, because $diag^-$ is cocommutative. As it is also coassociative,
the product is also associative. Compatibility with the signs is clear,
and commutativity follows from $(diag^U)^* \circ t_{U,U}^* 
=(t_{U,U} \circ diag^U)^* = (diag^U)^*$.
Another computation using the previous Lemma shows that
$[f] = [\iota_*(f)] = [f] \cdot [\iota_1 c_1]$
and (note that $\iota_1 c_1$ is central)
$[f] \cdot [\iota_1 c_1] = (-1)_T^{0\cdot q} t_{0,r} ([\iota_1 c_1]\cdot[f]) = [\iota_1 c_1]\cdot[f]$.
\end{proof}

\subsection{Localization of ring spectra}
\label{sect-locringspec}
We are now ready to define the localization of a symmetric ring spectrum
with respect to a central map, generalizing \cite[Example I.4.65]{S1}:
\begin{propdef}
\label{ringspec-local-def}
Let $R$ be a symmetric ring spectrum and $x: T^l \rightarrow R_m$ a central
map. Then we define a symmetric ring spectrum
$R[1/x]$ together with a map of symmetric ring spectra $j: R \rightarrow 
R[1/x]$ as follows. Levelwise, we set
$R[1/x]_p = Hom(T^{lp},R_{(1+m)p})$. There are maps $\Delta_{s,p}: \Sigma_p \rightarrow \Sigma_{sp},\, \Delta_{s,p}(\gamma)(i+ s\cdot(j-1)) = i + s\cdot(\gamma(j)-1),\, 1\leq i \leq s,\, 1\leq j \leq p$ permuting the $p$ summands of
$sp = s + s + \cdots + s$. Now $\Sigma_p$ via $\Delta_{l,p}$ acts
on $T^{lp}$, then via $\Delta_{1+m,p}$ on $R_{(1+m)p}$ and finally
by conjugation on $Hom(T^{lp},R_{(1+m)p})$. Hence the square\\
$\centerline{\xymatrix{
R[1/x]_p \we T^{lp} \ar[rr]^{\gamma \we 1} \ar[d]^{1 \we \Delta_{l,p}(\gamma^{-1})} && R[1/x]_p \we T^{lp} \ar[d]^{ev} \\
R[1/x]_p \we T^{lp} \ar[r]^{ev} & R_{(1+m)p} \ar[r]^{\Delta_{1+m,p}(\gamma)} & R_{(1+m)p}
}}$\\
is commutative. The multiplication
$\mu_{p,q}: R[1/x]_p \we R[1/x]_q \rightarrow R[1/x]_{p+q}$ is by definition the adjoint of $R[1/x]_p \we R[1/x]_q \we T^{l(p+q)} \xrightarrow{1 \we t_{R[1/x]_q,T^{lp}} \we 1} R[1/x]_p \we T^{lp} \we R[1/x]_q \we T^{lq} \xrightarrow{ev \we ev} R_{(1+m)p} \we R_{(1+m)q} \xrightarrow{\mu_{(1+m)p,(1+m)q}} R_{(1+m)(p+q)}$.
The unit of $R[1/x]$ is the composition of the unit of $R$ with $j$.
The map $j$ is defined by $j_p: R_p \rightarrow R[1/x]_p$ being the adjoint
to $R_p \we T^{lp} \xrightarrow{1 \we x^p} R_p \we R_{mp} \xrightarrow{\mu_{p,mp}} R_{p+mp} \xrightarrow{\xi_{m,p}} R_{(1+m)p}$. Here $x^p$ means of course 
$T^{lp} \xrightarrow{x^{\we p}} R_m^{\we p} \xrightarrow{\mu_{m,m,\dots,m}} R_{mp}$, and $x^0=\iota_0^R$. The permutation  $\xi_{m,p} \in \Sigma_{(1+m)p}$ 
is defined as follows: \\ 
$\xi_{m,p}(k) = \left\{ \begin{array}{ll} 1+ (1+m)\cdot (k -1) & \textnormal{if} \ 1\leq k \leq p \\1+j+(1+m)(i-1) & \textnormal{if} \ k = p + mi+ j \ \textnormal{mit} \ 1\leq i\leq p, 1 \leq j \leq m \end{array} \right.$ 
\end{propdef}

\begin{proof}
Again, this is very long but essentially straightforward.
To show the required properties (the multiplication maps
are equivariant, the multiplication is associative, the claims about 
the unit and about $j$) one shows them for the adjoints. 
For example: for the equivariance of the $\mu$,
let $(\gamma,\delta) \in \Sigma_p\times \Sigma_q \subseteq \Sigma_{p+q}$. 
We must show that $(\gamma + \delta) \cdot \mu_{p,q}^{R[1/x]} = \mu_{p,q}^{R[1/x]} \cdot (\gamma \we \delta)$. The left hand side is adjoint to\\
$\Delta_{1+m,p+q}(\gamma + \delta) \cdot \widehat{\mu_{p,q}^{R[1/x]}} \cdot (R[1/x]_p \we R[1/x]_q \we \Delta_{1+m,p+q}(\gamma+ \delta)^{-1})$.
The right hand side is adjoint to $\widehat{\mu_{p,q}^{R[1/x]}} \cdot (\gamma \we \delta \we T^{l(p+q)})$, and one shows that these adjoints coincide.
The claims about $j$ also use the fact that central elements are
stable under multiplication.
\end{proof}

The next results will be used to prove the Main Theorem 
\ref{ringspec-local-semist-th}.

\begin{lemma}
\label{ringspec-lem-deltasign}
Let $\gamma \in \Sigma_p$ and $s \in \Nbb$. Then $sgn(\Delta_{s,p}(\gamma)) = sgn(\gamma)^s$.
\end{lemma}
\begin{proof}
The map $\Delta_{s,p}: \Sigma_p \rightarrow \Sigma_{sp}$ is a group 
homomorphism by definition, so we only need to show the claim for
the generators ($\sigma_i = \tau_{i,i+1}, 1 \leq i \leq p-1$). For those,
we have  $\Delta_{s,p}(\sigma_i) = (s(i-1) + \chi_{s,s} + s(p-(i+1)))$ 
and thus $sgn(\Delta_{s,p}(\sigma_i)) = sgn(\chi_{s,s}) = (-1)^{s^2} = (-1)^s = sgn(\sigma_i)^s$.
\end{proof}

\begin{corollary}
\label{ringspec-loc-jxcent}
For any $f \in \pi^U_{r,q}(R[1/x])$ we have $f \cdot j_*([xc_l]) = j_*[(xc_l)] \cdot ((-1)_T^{(l-m)q} f)$.
\end{corollary}
\begin{proof}
One checks that $j_m x$ and hence $j_m x c_l$ is central and that 
$[j_m x c_l] = j_*([x c_l])$. Now the claim follows from
the commutativity claim in Proposition \ref{ringspec-sthgrmonoid-def} 
and $t_{0,r} = id$.
\end{proof}

\begin{lemma}
\label{ringspec-local-stabladj-lem}
    Let $R$ be a symmetric $T$-ring spectrum and $x: T^l \rightarrow R_m$ 
a central map in $\Dst$.
    Let $f: \Sbb^r \we U \we T^{q+n} \rightarrow R[1/x]_n$ be a map
in $Ho(\Dst)$ and $\hat{f} := ev \circ (f \we T^{ln}): \Sbb^r \we U \we T^{q+n} \we T^{ln} \rightarrow R_{(1+m)n}$. Then for
    $\iota_*^\alpha(f) = \sigma_n^{\alpha, R[1/x]} \circ(f \we T^\alpha): \Sbb^r \we U \we T^{q+n} \we T^\alpha \rightarrow R[1/x]_{n+\alpha}$, $\alpha \in \Nbb$ we have for the associated map
    $\widehat{\iota_*^\alpha(f)} := ev \circ (\iota_*^\alpha(f) \we T^{l(n + \alpha)}) = (1 + \xi_{m,\alpha}) \circ \mu^R_{(1+m)n+\alpha,m\alpha} \circ (\iota_*^\alpha(\hat{f}) \we x^\alpha) \circ (\Sbb^r \we U \we T^{q+n} \we t_{T^\alpha,T^{ln}} \we T^{l\alpha})$.
\end{lemma}
\begin{proof}
    Because of $\sigma^{\alpha,R[1/x]} = \mu_{n,\alpha}^{R[1/x]} \circ (R[1/x]_{n} \we\iota_\alpha^{R[1/x]})$,
    $\iota_\alpha^{R[1/x]} =j_\alpha \circ \iota_\alpha^R$ and $\sigma^{\alpha,R} = \mu_{n,\alpha}^R \circ (R_n \we\iota_\alpha^R)$ we have for the associated map\\
    $\widehat{\iota_*^\alpha(f)} = ev \circ (\iota_*^\alpha(f) \we T^{l(n+\alpha)})
    = ev \circ ([\mu_{n,\alpha}^{R[1/x]} \circ (f \we (j_\alpha\circ \iota_\alpha^R))] \we T^{l(n+\alpha)})
    \\= \mu_{(1+m)n,(1+m)\alpha} \circ ((ev \circ (f \we T^{ln}))\we (ev \circ ((j_\alpha\circ \iota_\alpha^R) \we T^{l\alpha}))) \circ (1 \we t_{T^\alpha,T^{ln}} \we 1)
    \\= \mu_{(1+m)n,(1+m)\alpha} \circ (\hat{f}\we (\xi_{m,\alpha} \circ\mu_{\alpha,m\alpha} \circ (\iota_\alpha^R \we x^\alpha)) \circ (1 \we t_{T^\alpha,T^{ln}} \we 1)
    \\= (1 + \xi_{m,\alpha}) \circ \mu_{(1+m)n,(1+m)\alpha} \circ (1 \we \mu_{\alpha,m \alpha}) \circ (\hat{f} \we \iota_\alpha^R \we x^\alpha) \circ (1 \we t_{T^\alpha,T^{ln}} \we 1)
    \\= (1 + \xi_{m,\alpha}) \circ \mu_{(1+m)n+\alpha,m\alpha} \circ (\mu_{(1+m)n,\alpha} \we 1) \circ (\hat{f} \we \iota_\alpha^R \we x^\alpha) \circ (1 \we t_{T^\alpha,T^{ln}} \we 1)
    \\= (1 + \xi_{m,\alpha}) \circ \mu_{(1+m)n+\alpha,m\alpha} \circ (\iota_*^\alpha(\hat{f}) \we x^\alpha) \circ (1 \we t_{T^\alpha,T^{ln}} \we 1) \\$
\end{proof}

\begin{lemma}
\label{ringspec-local-semist-lem}
Let $R$ be a levelwise fibrant semistabile symmetric $T$-ring spectrum and 
$x: T^l \rightarrow R_m$ a central map. Then for any $f, g \in [\Sbb^r \we U \we T^{q+n},R[1/x]_n]$ with $\hat{f} = (-1)_T^{\nu}(\xi \circ \hat{g})$
for some fixed $\nu \in \Zbb, \xi \in \Sigma_{(1+m)n}$, we have 
$[f] = ((-1)_T^{\nu} |\xi|_T) [g]$ in $\pi_{r,q}^U(R[1/x])$.
\end{lemma}
\begin{proof}
As $R$ is semistable, there is an $\alpha \in \Nbb$
for which $\iota^\alpha_*(|\xi|_T (\xi \circ \hat{g})) = 
\iota^\alpha_*(\hat{g})$, hence 
$\iota_*^\alpha(\hat{f}) = \iota_*^\alpha((-1)_T^{\nu}(\xi \circ \hat{g})) 
= (-1)_T^{\nu}|\xi|_T (\iota_*^\alpha(|\xi|_T(\xi \circ \hat{g}))) = 
(-1)_T^{\nu}|\xi|_T (\iota_*^\alpha(\hat{g}))$. Applying
Lemma \ref{ringspec-local-stabladj-lem} we deduce \\ 
$\widehat{\iota_*^\alpha(f)} =  (-1)_T^{\nu}|\xi|_T \widehat{\iota_*^\alpha(g)} = \widehat{v}$ mit $v = [(-1)_T^{\nu}|\xi|_T] \iota_*^\alpha(g)$.
As $R$ is levelwise fibrant, the map\\ 
$[\Sbb^r \we U \we T^{q+n+\alpha}, Hom(T^{l(n+\alpha)}, R_{(1+m)(n+\alpha)})] \xrightarrow{ev \circ (- \we T^{l(n+\alpha)})} [\Sbb^r \we U \we T^{q+n+\alpha+l(n+\alpha)},R_{(1+m)(n+\alpha)}]$ is bijective. Therefore we have
$\iota_*^\alpha(f) = [(-1)_T^{\nu}|\xi|_T]\iota_*^\alpha(g)$.
\end{proof}

We are now able to state the Main Theorem of this section,
which is a generalization of \cite[Corollary I.4.69]{S1}.
(The definition of $c_l$ is before Proposition-Definition 
\ref{ringspec-sthgrmonoid-def}.)
\begin{theorem}
\label{ringspec-local-semist-th}
Assume that the standard assumptions of the beginning of section 4
hold (these are satisfied e.g. in the motivic case by Proposition
\ref{thmain-ex1}).
Let $R$ be a levelwise fibrant semistable symmetric $T$-ring spectrum and 
$x: T^l \rightarrow R_m$ a central map.
Then $R[1/x]$ is semistable, and for all $U \in \Bst^\prime$ 
the ring homomorphism $\pi_{*,*}^U(R) \xrightarrow{j_*} \pi_{*,*}^U(R[1/x])$ 
is a $[x c_l]$-localization.
\end{theorem}

\begin{proof}
\textbf{Semistability}:
Using Theorem \ref{motsemi-theoremorg-th} it suffices to show that the cycle operator $d$ acts trivially on $\pi_{r,q}^U(R[1/x])$. Let $f \in [\Sbb^r \we U \we T^{q+n},R[1/x]_n]$ represent an element in $\pi_{r,q}^U(R[1/x])$.
After stabilization, we may assume that $n$ is even.
Then $d f$ is represented by $\chi_{n,1} \circ \iota_*(f)$ as $|\chi_{n,1}|_T 
= 1$. It remains to show that after stabilization $[\chi_{n,1} \circ \iota_*(f)] = [f] = [\iota_*(f)]$. This reduces to the following: For
$f \in [\Sbb^r \we U \we T^{q+n},R[1/x]_n]$ and $\gamma \in \Sigma_n$ with 
$|\gamma|_T = 1$ we have $[\gamma \circ f] = [f]$ in $\pi^U_{r,q}(R)$.
To show this, consider the adjoint $\widehat{\gamma \circ f} = ev \circ ((\gamma \circ f) \we T^{ln}) = \Delta_{1+m,n}(\gamma) \circ\hat{f} \circ (1 \we \Delta_{l,n}(\gamma)^{-1})$ (compare Proposition-Definition \ref{ringspec-local-def}). As $|\gamma|_T^s = 1$ is the sign of $\Delta_{s,n}(\gamma)$ (Lemma 
\ref{ringspec-lem-deltasign}), we obtain $\widehat{\gamma \circ f} = |\Delta_{1+m,n}(\gamma)|_T (\Delta_{1+m,n}(\gamma) \circ\hat{f})$ by Definition \ref{vorzeichen-def}. Applying Lemma \ref{ringspec-local-semist-lem} yields 
$[\gamma \circ f] = [f]$ as claimed.

\textbf{Localization}: By Proposition-Definition 
\ref{ringspec-sthgrmonoid-def}, we know that
$\pi_{*,*}^U(R[1/x])$ is a ring and $j_*: \pi_{*,*}^U(R) \xrightarrow{} 
\pi_{*,*}^U(R[1/x])$ a ring homomorphism. It remains to show that $j_*$ 
is a $[x c_l]$-localization. For this, we will check that the three conditions
of Proposition \ref{monloc-prop} are satisfied (note that the Ore 
condition holds by Corollary \ref{ringspec-loc-jxcent}. 
First, we show that $j_*([x c_l])$ is a unit in $\pi_{*,*}^U(R[1/x])$. The map
$j_m\circ x \circ c_l$ represents $j_*([x c_l])$ and this element has
$(\pm 1)_T [y c_{1+m}]$ as a left inverse (up to sign).
Here $y: T^{1+m} \rightarrow R[1/x]_{1+l}$ denotes the adjoint to 
$\mu_{1+m+l,ml} \circ (\iota_{1+m+l} \we x^l)$.
We now show that $[(y c_{1+m})\cdot (j_m x c_l)]$ 
equals (up to sign) the unit in $\pi^U_{\ast,\ast}(R)$. 
By definition $f := (y c_{1+m})\cdot (j_m x c_l) = \mu_{1+l,m}^{R[1/x]} \circ (y \we (j_m x)) \circ (c_{1+m} \we c_l) \circ (\Sbb^0 \we t_{\Sbb^0 \we U, U \we T^{1+m}} \we T^l) \circ (s_{0,0} \we diag^U \we (-1)_T^{(l-m)(1+l)} \we 1) 
= \mu_{1+l,m}^{R[1/x]} \circ (y \we (j_m x)) \circ (l^2_{T^{1+m}} \we l^2_{T^l}) \circ (S^0 \we t_{S^0 \we S^0, S^0 \we T^{1+m}} \we T^l) \circ ([l^{-1}_{S^0}\cong^{\Sbb^0}] \we ([\omega_U \we \omega_U] \circ diag^U) \we (-1)_T^{(l-m)(1+l)} \we 1)$.
Using $(\omega_U \we \omega_U) \circ diag^U = (\omega_U \we S^0) \circ \rho_U^{-1} = \rho_{S^0}^{-1} \circ \omega_U$ and $l^2_{T^{1+m+l}}=S^0 \we S^0 \we T^{1+m+l} \cong S^0 \we S^0 \we S^0 \we S^0 \we T^{1+m+l} \cong S^0 \we S^0 \we T^{1+m} \we S^0 \we S^0 \we T^l \cong T^{1+m} \we T^l$ we get
$f = \mu_{1+l,m}^{R[1/x]} \circ (y \we (j_m x)) \circ ((-1)_T^{(l-m)(1+l)} \we 1) \circ c_{1+m+l}$. The adjoint of 
$f$ is ($c := c_{1+m+l} \we T^{l(1+l+m)}, a := (1+m)(1+l)$):\\
$\hat{f} = ev \circ (f \we T^{l(1+l+m)})
= \mu_{(1+m)(1+l),(1+m)m} \circ ((ev \circ (y \we T^{l(1+l)}))\we (ev \circ (j_m x \we T^{lm}))) \\ \circ (1 \we t_{T^l,T^{l(1+l)}} \we 1) \circ ((-1)_T^{(l-m)(1+l)} \we 1) \circ c
\\= \mu_{a,(1+m)m} \circ ((\mu_{1+m+l,ml} \circ (\iota_{1+m+l} \we x^l))\we (\xi_{m,mm} \circ x^{1+m})) \circ (1 \we t_{T^l,T^{l(1+l)}} \we 1) \circ ((-1)_T^{(l-m)(1+l)} \we 1) \circ c
\\= (a+\xi_{m,mm}) \circ [\mu_{a,m+mm} \circ (\mu_{1+m+l,ml} \we 1)] \circ ((\iota_{1+m+l} \we x^l)\we x^{1+m}) \circ (1 \we t_{T^l,T^{l(1+l)}} \we 1) \circ ((-1)_T^{(l-m)(1+l)} \we 1) \circ c
\\= (a+\xi_{m,mm}) \circ [\mu_{1+m+l,ml+m+mm} \circ (1 \we \mu_{ml,m+mm})] \circ (\iota_{1+m+l} \we (x^l\we x^{1+m})) \circ (1 \we t_{T^l,T^{l(1+l)}} \we 1) \circ ((-1)_T^{(l-m)(1+l)} \we 1) \circ c
\\= (a+\xi_{m,mm}) \circ \mu_{1+m+l,m(1+l+m)} \circ (\iota_{1+m+l} \we x^{1+l+m})  \circ c \circ (1 \we t_{T^l,T^{l(1+l)}} \we 1) \circ (1 \we (-1)_T^{(l-m)(1+l)} \we 1)$\\
where we used $\mu_{sm,tm} \circ (x^s \we x^t) = x^{s+t}$,
i.e. the associativity of $R$.

The unit $[\iota_1^{R[1/x]} c_1]$ in $\pi_{*,*}^U(R)$ 
is also represented by $g := \iota_*^{l+m}(\iota_1^{R[1/x]} c_1) = 
\iota_{1+l+m}^{R[1/x]} \circ c_{1+l+m}$ which is adjoint to
$\hat{g} = \xi_{m,1+l+m} \circ \mu_{1+l+m,m(1+l+m)} \circ (\iota_{1+l+m} \we x^{1+l+m}) \circ c$ ist. Therefore $\hat{f} = \xi^\prime \circ \hat{g} \circ (1 \we t_{T^l,T^{l(1+l)}} \we 1) \circ (1 \we (-1)_T^{(l-m)(1+l)} \we 1) = (-1)_T^{\nu} \xi^\prime \circ \hat{g}$ with $\xi^\prime = (a+\xi_{m,mm}) \circ \xi_{m,1+l+m}^{-1}$ and $\nu = l^2(1+l) + (l-m)(1+l)$.
Applying \ref{ringspec-local-semist-lem} yields $[f] = ((-1)_T^{\nu}|\xi^\prime|_T) [g]$ and finally $(((-1)_T^{\nu}|\xi^\prime|_T)[y c_{1+m}]) \cdot j_\ast([x c_l]) = ((-1)_T^{\nu}|\xi^\prime|_T)[f] = [g] = 1$ in $\pi^U_{r,q}(R[1/x])$.
By Corollary \ref{ringspec-loc-jxcent} $j_\ast([x c_l])$ has then also
a right inverse.

The second condition amounts to show that for any
$z \in \pi_{*,*}^U(R[1/x])$ -- repesented by some 
$f \in [\Sbb^r \we U \we T^{q+n},R[1/x]_n]$ -- there is some
$u \in \pi_{*,*}^U(R)$ and some $p \in \Nbb$ 
satisfying $z \cdot j_*([xc_l])^{\cdot p} = j_*((\pm 1)_T u)$.
For $u$ we choose $\hat{f}$ as representative and set $p = n$. Then
$j_*(u)$ is represented by $g := j_{(1+m)n} \circ \hat{f}$ which is adjoint to
$\hat{g} := ev \circ (g \we T^{l(1+m)n}) = \xi_{m,(1+m)n} \circ \mu_{(1+m)n,m(1+m)n} \circ (\hat{f} \we x^{(1+m)n})$.
The element $z \cdot j_*([xc_l])^{\cdot n} = z \cdot j_*([xc_l]^{\cdot n})$ 
is represented by $h := f \cdot (j_{mn} \circ (xc_l)^{\cdot n})$,
where $(xc_l)^{\cdot n}$ is given by $x^n \circ ((-1)_T^{(l-m)m(n-1)n/2} \we 1) \circ c_{nl}$, as we show by induction:\\ 
$(x^n \circ ((-1)_T^{(l-m)m(n-1)n/2} \we 1) \circ c_{nl})\cdot (xc_l) = \mu_{mn,m}^R \circ ((x^n\circ ((-1)_T^{(l-m)m(n-1)n/2} \we 1))\we x)\\ \circ ((-1)_T^{(l-m)mn} \we 1) \circ c_{(n+1)l} = x^{n+1} \circ ((-1)_T^{(l-m)m[(n-1)n/2 + n]} \we 1) \circ c_{(n+1)l}$ (cf. also the computation of $f$ above).
\\Furthermore, $h = \mu_{n,mn}^{R[1/x]} \circ (f \we (j_{mn} \circ (xc_l)^{\cdot n})) \circ (\Sbb^r \we t_{\Sbb^0 \we U, U \we T^{q+n}} \we T^{ln}) \circ (s_{r,0} \we diag^U \we (-1)_T^{(l-m)n^2} \we 1)$ implies \\
$\hat{h} = ev \circ (h \we T^{l(n+mn)}) = \mu_{(1+m)n,(1+m)mn}^R \circ [(ev \circ (f \we T^{ln}) ) \we (ev \circ ((j_{mn} \circ (xc_l)^{\cdot n}) \we T^{lmn}) )] \\ \circ (1\we t_{\Sbb^0 \we U \we T^{ln},T^{ln}} \we 1) \circ (\Sbb^r \we t_{\Sbb^0 \we U, U \we T^{q+n}} \we 1) \circ (s_{r,0} \we diag^U \we (-1)_T^{(l-m)n^2} \we 1)\\
= \mu_{(1+m)n,(1+m)mn}^R \circ [\hat{f} \we (\xi_{m,mn} \circ \mu_{mn,mmn}^R \circ ((x^n \circ ((-1)_T^{(l-m)m(n-1)n/2} \we 1) \circ c_{nl})\we x^{mn} ))] \\ \circ (1\we t_{\Sbb^0 \we U \we T^{ln},T^{ln}} \we 1) \circ (\Sbb^r \we t_{\Sbb^0 \we U, U \we T^{q+n}} \we 1) \circ (s_{r,0} \we diag^U \we (-1)_T^{(l-m)n^2} \we 1)\\
= (a^\prime+\xi_{m,mn}) \circ \mu_{(1+m)n,(1+m)mn}^R \circ (\hat{f} \we (\mu_{mn,mmn}^R \circ (x^n \we x^{mn} )) \circ (1 \we (-1)_T^{(l-m)n^2 + (l-m)m(n-1)n/2 + (ln)^2} \we 1)\\
=(a^\prime+\xi_{m,mn}) \circ \mu_{(1+m)n,(1+m)mn}^R \circ (\hat{f} \we x^{(1+m)n}) \circ (1 \we (-1)_T^{(l-m)[n^2 + m(n-1)n/2] + (ln)^2} \we 1)$,\\
where $a^\prime := (1+m)n$. Here the second last step uses\\
$\Sbb^r \we [(U \we T^{q+n} \we T^{ln} \we [l^2_{T^{ln}} \, (\cong^{\Sbb^0} \we \omega_U \we T^{ln})]) \, (U \we T^{q+n} \we t_{\Sbb^0 \we U \we T^{ln},T^{ln}}) \, (t_{\Sbb^0 \we U, U \we T^{q+n}} \we T^{ln} \we T^{ln})] \we T^{lmn} \circ (s_{r,0} \we diag^U \we T^{\tilde{a}})
= \\ \Sbb^r \we [(U \we T^{q+n} \we t_{T^{ln},T^{ln}}) \, (U \we T^{q+n} \we l^2_{T^{ln}} \we T^{ln}) \, (t_{S^0 \we S^0, U \we T^{q+n}} \we T^{ln})] \we T^{lmn} \circ (\rho^{-1}_{\Sbb^r} \we ((\omega_U \we U)\, diag^U) \we T^{\tilde{a}})\\
= \Sbb^r \we [(U \we T^{q+n} \we (-1)_T^{(ln)^2} \we 1) \circ (U \we T^{q+n} \we l^2_{T^{ln}} \we T^{ln}) \circ (t_{S^0 \we S^0, U \we T^{q+n}} \we T^{ln})] \we T^{lmn} \circ (\rho^{-1}_{\Sbb^r} \we l^{-1}_U \we T^{\tilde{a}})\\ = \Sbb^r \we U \we (-1)_T^{(ln)^2} \we T^{\tilde{a} -1}$ with $\tilde{a} := (q+n)+ln+l(n+mn)$.
Hence $\hat{h}$ and $\hat{g}$ only differ by a permutation and a sign,
and Lemma \ref{ringspec-local-semist-lem} then implies $z \cdot j_*([xc])^{\cdot p} = [h] = (\pm 1)_T[g] = (\pm 1)_T j_*(u) = j_*((\pm 1)_T u)$.

It remains to verify the third condition:
For any $[f],[g] \in \pi_{*,*}^U(R)$ with $j_*([f]) = j_*([g])$, we have\\
$[f] \cdot [xc_l]^{\cdot n} = [g] \cdot [xc_l]^{\cdot n}$ for some 
$n \in \Nbb$. We may assume that $f,g \in [\Sbb^r \we U \we T^{q+n},R_n]$ 
and that $j_n \circ f = j_n \circ g$. Using $(x c_l)^{\cdot n} = x^n \circ 
((-1)_T^{(l-m)m(n-1)n/2} \we 1) \circ c_{nl}$ we obtain \\ $f \cdot (xc_l)^{\cdot n} = \mu_{n,mn}^R \circ (f \we x^n) \circ (-1)_T^{(l-m)[m(n-1)n/2 + n^2]} = \xi_{m,n}^{-1} \circ \widehat{j_n \circ f} \circ (-1)_T^{(l-m)[m(n-1)n/2 + n^2]}$, as \\
$\Sbb^r \we [(U \we T^{q+n} \we l^2_{T^{ln}}) \circ (U \we T^{q+n} \we \cong^{\Sbb^0} \we \omega_U \we T^{ln}) \circ t_{\Sbb^0 \we U, U \we T^{q+n}} \we T^{ln}] \circ (s_{r,0} \we diag^U \we (-1)_T^{(l-m)n^2} \we 1) \\
= \Sbb^r \we [(U \we T^{q+n} \we l^2_{T^{ln}}) \circ t_{S^0 \we S^0, U \we T^{q+n}} \we T^{ln}] \circ (\rho^{-1}_{\Sbb^r} \we ([\omega_U \we U] diag^U) \we (-1)_T^{(l-m)n^2} \we 1)\\
= \Sbb^r \we [(U \we T^{q+n} \we l^2_{T^{ln}}) \circ t_{S^0 \we S^0, U \we T^{q+n}} \we T^{ln}] \circ (\rho^{-1}_{\Sbb^r} \we l^{-1}_U \we (-1)_T^{(l-m)n^2} \we 1) = 1 \we (-1)_T^{(l-m)n^2} \we 1$.\\
This also holds for $g$, thus $f \cdot (xc_l)^{\cdot n} = g \cdot (xc_l)^{\cdot n}$ and hence $[f] \cdot [xc_l]^{\cdot n} = [f \cdot (xc_l)^{\cdot n}] = [g \cdot (xc_l)^{\cdot n}] = [g] \cdot [xc_l]^{\cdot n}$ as desired.
\end{proof}

We have used the following standard criterion for localizations 
above:

\begin{proposition}
\label{monloc-prop}
Let $M, N$ be two rings and $x \in M$. Assume that for any  
$x_1 \in M$ there is an  $x_2 \in M$ with $x_1 x = x x_2$
(Ore condition). Assume further that there is a ring homomorphism
$j: M \rightarrow N$ satisfying the following.
\begin{enumerate}[1.)]
\item There are $y, y^\prime \in N$ with $y j(x) = 1$ and $j(x) y^\prime = 1$,
\item For all $z \in N$ there is some $p \in \Nbb$ and some $u \in M$ with $z j(x)^p = j(u)$
\item For all $a, b \in M$ with $j(a) = j(b)$ there is an $n \in \Nbb$ with 
$a x^n = b x^n$
\end{enumerate}
Then $j$ is an $[x]$-localization. If $M$ and $N$ are graded,
then $j$ is a graded ring homomorphism. If moreover
$x$ is homogenous, then it suffices to check the above
conditions for homogenous elements $x_1, x_2, a$ and $b$.
\end{proposition}


Stephan H\"ahne, Schumannstra{\ss}e 7, 53113 Bonn, StephanHP@gmx.de

Jens Hornbostel, Bergische Universit\"at Wuppertal,
FB C, Mathematik und Informatik, Gau{\ss}stra{\ss}e 20, 42119 Wuppertal,
hornbostel@math.uni-wuppertal.de


\begin{thebibliography}{langgg} 

    \bibitem[DG]{DG} M. Demazure, P. Gabriel: \emph{Introduction to Algebraic Geometry and Algebraic Groups}. North-Holland, 1980.
    \bibitem[DL\O RV]{DLORV} B. I. Dundas, M. Levine, P. A. \O stv\ae r, O. R\"ondigs, V. Voevodsky: \emph{Motivic homotopy theory}. Lectures at a Summer School in Nordfjordeid, Norway, August 2002, Springer-Verlag Berlin Heidelberg 2007.
    \bibitem[DR\O]{DRO} B. I. Dundas, O. R\"ondigs, P. A. \O stv\ae r: \emph{Motivic Functors}. Doc. Math. \textbf{8} (2003) 489--525.
\bibitem[GS]{GS} D. Gepner, V. Snaith: \emph{On the motivic spectra representing 
algebraic cobordism and algebraic K-theory}. Doc. Math 14 (2009), 319-396.

 \bibitem[H]{H} S. H\"ahne: Semistabile symmetrische Spektren in der
${\Abb}^1$-Homotopietheorie. Diplomarbeit, Universit\"at Bonn, June 2010.
    \bibitem[Hi]{Hi} P. S. Hirschhorn: \emph{Model Categories and Their Localizations}. Mathematical Surveys and Monographs Vol. \textbf{99}, American Mathematical Society, 2003.
    \bibitem[Hor1]{HB1} J. Hornbostel: \emph{Motivic Chromatic Homotopy 
Theory}. Math. Proc. Cambridge Philos. Soc.  140  (2006),  no. 1, 95--114.
    \bibitem[Hor2]{Hor2} J. Hornbostel: \emph{Preorientations of the derived motivic multiplicative group}. Alg. Geom. Topology 13 (2013), 2667--2712.
    \bibitem[Hov1]{H3} M. Hovey: \emph{Model Categories}. Math. Survey and Monographs Vol. \textbf{63}, American Mathematical Society, 1999.
    \bibitem[Hov2]{H1} M. Hovey: \emph{Spectra and symmetric spectra in general model categories}.
J. Pure Appl. Algebra  165  (2001),  no. 1, 63--127.
    \bibitem[HSS]{H2} M. Hovey, B. Shipley, J. Smith: \emph{Symmetric Spectra}.
J. Amer. Math. Soc. 13 (2000), 149-208.
   \bibitem[Ja1]{J2} J. F. Jardine: \emph{Simplicial presheaves}. Journal of Pure and Applied Algebra \textbf{47} (1987), 35--87, North-Holland.
    \bibitem[Ja2]{J} J. F. Jardine: \emph{Motivic symmetric spectra}. Doc. Math. \textbf{5} (2000), 445--552.
    \bibitem[MMSS]{MMSS} M. A. Mandell, J. P. May, S. Schwede, B. Shipley:
\emph{Model categories of diagram spectra}.  Proc. London Math. Soc. (3)  
82  (2001),  no. 2, 441--512.
    \bibitem[Mo]{Mo} F. Morel: \emph{An introduction to $\Abb^1$-homotopy theory}. In Contemporary Developments in Algebraic K-theory, ICTP Lecture notes \textbf{15} (2003), M. Karoubi, A.O. Kuku, C. Pedrini ed., 357-441.
    \bibitem[MV]{MV} F. Morel, V. Voevodsky: \emph{$\Abb^1$-homotopy theory of schemes}. Publ. Math. IHES no. 90 (1999), 45-143.
    \bibitem[NS]{NS} N. Naumann, M. Spitzweck:
\emph{Brown representability in $\Abb^1$-homotopy theory}. J. K-Theory 7 (2011), no. 3, 527--539. 
    \bibitem[PPR1]{P1} I. Panin, K. Pimenov, O. R\"ondigs: \emph{On Voevodsky's algebraic K-theory spectrum}. Algebraic topology, 279--330, Abel Symp., 4, Springer, Berlin, 2009 
    \bibitem[PPR2]{P2} I. Panin, K. Pimenov, O. R\"ondigs: \emph{A universality theorem for Voevodsky's algebraic cobordism spectrum}. 
Homology, Homotopy Appl.  10  (2008),  no. 2, 211--226.
\bibitem[PY]{PY} I. Panin, S. Yagunov: \emph{Rigidity for orientable 
functors}. J. Pure Appl. Algebra 172 (2002), no. 1, 49--77.
    \bibitem[RS\O]{RSO} O. R\"ondigs, M. Spitzweck, P. A. \O stv\ae r: \emph{Motivic strict ring models for K-theory}. Proc. AMS 138 (2010), 3509--3520.
    \bibitem[Sch07]{S1} S. Schwede: \emph{An untitled book project about symmetric spectra}. Preprint v2.4 / July 12, 2007, http://www.math.uni-bonn.de/people/schwede/SymSpec.pdf.
    \bibitem[Sch08]{S4} S. Schwede: \emph{On the homotopy groups of symmetric spectra}. Geometry \& Topology \textbf{12} (2008), 1313--1344.
 \bibitem[Sch12]{S6} S. Schwede, \emph{Symmetric Spectra}, Book project,
draft version 2012.
    \bibitem[S\O]{SO} M. Spitzweck, P. A. \O stv\ae r: \emph{The Bott inverted projective space is homotopy algebraic K-theory}. Bulletin of the London Mathematical Society \textbf{41}(2): 281--292, Feb 19, 2009.
    \bibitem[Vo2]{V2} V. Voevodsky: \emph{$\Abb^1$-homotopy theory}. Doc. Math. ICM \textbf{I} (1998), 417--442.
Vol. \textbf{10}(2), 2008, pp. 212--226.
    
\end{thebibliography}
\end{document}